\documentclass[12pt,centertags,oneside]{amsart}
\usepackage{amsmath,amstext,amsthm,amscd,typearea,hyperref}
\usepackage{amssymb}
\usepackage{a4wide}
\usepackage[mathscr]{eucal}
\usepackage{mathrsfs}
\usepackage{typearea}
\usepackage{charter}
\usepackage{pdfsync}
\usepackage{stmaryrd}

\usepackage{tikz-cd}
\usepackage{pgfplots}
\tikzset{>=latex}
\pgfplotsset{compat=newest}
\usepackage{caption}
\usepackage{adjustbox}
\usepackage[a4paper,width=16.5cm,top=3cm,bottom=3cm]{geometry}

\numberwithin{equation}{section}

\usepackage[french,english]{babel}

\allowdisplaybreaks
\emergencystretch=\maxdimen
\usepackage{multicol}

\usepackage{xcolor}

\newtheorem{theorem}{Theorem}[section]
\newtheorem{definition}[theorem]{Definition}
\newtheorem{proposition}[theorem]{Proposition}
\newtheorem{corollary}[theorem]{Corollary}
\newtheorem{lemma}[theorem]{Lemma}
\newtheorem{remark}[theorem]{Remark}

\newtheorem*{definition*}{Definition}

\newcommand{\cali}[1]{\mathscr{#1}}

\newcommand{\supp}{{\rm supp}}

\renewcommand{\Im}{\mathop{\mathrm{Im}}}

\newcommand{\dist}{\mathop{\mathrm{dist}}\nolimits}

\newcommand{\ddc}{{\rm dd^c}}
\newcommand{\dc}{{\rm d^c}}

\newcommand{\dd}{{\rm d}}

\newcommand{\dbar}{\overline\partial}

\newcommand{\SH}{{\rm SH}}

\newcommand{\id}{{\rm id}}

\newcommand{\ep}{\epsilon}
\newcommand{\vep}{\varepsilon}

\newcommand{\Leb}{{\rm Leb}}
\newcommand{\Lip}{\mathop{\mathrm{Lip}}\nolimits}

\newcommand{\Vol}{{\rm Vol}}

\newcommand{\Area}{{\rm Area}}

\newcommand{\sing}{{\rm sing}}

\newcommand{\sign}{{\rm sign\ \!}}

\newcommand{\Ac}{\cali{A}}

\newcommand{\Cc}{\cali{C}}

\newcommand{\Oc}{\cali{O}}

\newcommand{\W}{\cali{W}}

\newcommand{\cM}{\mathcal{M}}
\newcommand{\cI}{\mathcal{I}}

\newcommand{\FS}{{\rm FS}}
\newcommand{\B}{\mathbb{B}}

\newcommand{\D}{\mathbb{D}}
\newcommand{\C}{\mathbb{C}}

\newcommand{\N}{\mathbb{N}}
\newcommand{\Z}{\mathbb{Z}}
\newcommand{\R}{\mathbb{R}}

\renewcommand\P{\mathbb{P}}

\newcommand{\E}{\mathbf{E}}
\newcommand{\s}{\mathbf{s}}

\newcommand{\lp}{\langle}
\newcommand{\rp}{\rangle}

\newcommand{\norm}[1]{\lVert#1\rVert}

\newcommand{\oA}{\mathcal{A}}
\newcommand{\oB}{\mathcal{B}}
\newcommand{\oF}{\mathcal{F}}

\newcommand{\oE}{\mathcal{E}}

\newcommand{\oQ}{\mathcal{Q}}
\newcommand{\oS}{\mathcal{S}}
\newcommand{\oL}{\mathcal{L}}
\newcommand{\diam}{{\rm diam}}
\newcommand{\oI}{\mathcal{I}}

\newcommand{\oD}{\mathcal{D}}

\newcommand{\oT}{\mathcal{T}}

\newcommand{\fh}{\mathfrak{h}}

\newcommand{\bH}{\mathbf{H}}

\newcommand{\bQ}{\mathbf{Q}}

\newcommand{\bP}{\mathbf{P}}
\newcommand{\bU}{\mathbf{U}}


\title[Hole event for random holomorphic sections on compact Riemann surfaces]{Hole event for random holomorphic sections\\ on compact Riemann surfaces}

\author{Tien-Cuong Dinh}
\address{Department of Mathematics,  National University of Singapore - 10, Lower Kent Ridge Road - Singapore 119076}
\email{matdtc@nus.edu.sg}

\author{Subhroshekhar Ghosh}
\address{Department of Mathematics,  National University of Singapore - 10, Lower Kent Ridge Road - Singapore 119076}
\email{matghos@nus.edu.sg}

\author{Hao Wu}
\address{Department of Mathematics,  National University of Singapore - 10, Lower Kent Ridge Road - Singapore 119076}
\email{e0011551@u.nus.edu}

\thanks{}

\begin{document}

\begin{abstract}
Let $X$ be a compact Riemann surface and $\oL$ be a positive line bundle on it.
We study the conditional zero expectation of all the holomorphic  sections of $\oL^n$ which do not vanish on $D$ for some fixed open subset $D$ of $X$. We prove that as $n$ tends to infinity, the zeros of these sections are equidistributed outside $D$ with respect to a probability measure $\nu$. This gives rise to a surprising forbidden set.
\end{abstract}

\clearpage\maketitle
\thispagestyle{empty}


\medskip


\setcounter{tocdepth}{1}
\tableofcontents

%%%%%%%%%%%%%%%%%%

\section{Introduction and main results}\label{s:intro}

The study of random polynomial initiated from about one century ago by  Littlewood, Erd\"os, Offord, Kac, Bloch, P\'olya, Hammersley \cite{blo-plya-plms,erdos-turn-annals,erdos-offord-plms,ham-pro,kac-bams,kac-plms,lit-efford-jlms,lit-offord-rec}.  For polynomials of degree $n$ equipping with  certain random coefficients, they investigated the stochastic properties of the   zeros of the random polynomials as $n$ tends to infinity, such as the number of real zeros and the distribution of the zeros.  Around 1990's, this topic was fast developed, see e.g.\ \cite{ble-di-jsp,bog-boh-leb-prl,bog-boh-leb-jsp,ede-kos-bams,ibr-zei-tams,mas-aka-2,mas-aka,she-van-tams} etc.

\smallskip

Over the last twenty years, the random series called \textit{Gaussian entire function} attracted a lot of attention, which is defined as
$$ \psi(z):= \sum_{k=0}^ \infty \zeta_k {z^k\over \sqrt{k!}}, $$
where the $\zeta_k$'s are i.i.d.\ standard complex normal distributions. Such a random series first appeared in  \cite{bog-boh-leb-prl,bog-boh-leb-jsp}. The zeros of $\psi$ are  equidistributed with respect to the Lebesgure measure on $\C$.  It is worth mentioning that  Sodin-Tsirelson \cite{sod-tsi-ijm-1,sod-tsi-isrj,sod-tsi-ijm-2} proved upper and lower bounds for the hole probability, which states that the probability  that $\psi$ has no zeros inside $\D(0,r)$ is bounded from above and below by $e^{-c_1 r^4}$ and $e^{-c_2 r^4}$ respectively. Several subsequently
generalized works are studied, see e.g.\  \cite{buc-nis-ron-sod-ptrf,gho-nis-con,cho-zei-imrn,kri-jsp,nis-imrn,nis-jdm,per-vir-acta,shi-zel-imrn}.
  In particular, \cite{nlo-lev-poten,blo-tho,blo-shi-mrl,zhu-anapde,zre-michigan}  gave a vision of random polynomials in several complex variables. The reader may refer to the survey \cite{gaussian-book} for more background.
 
 However, the distribution of the zeros of the hole event remains unknown for quite a long time, until the breakthrough work written by the second author and Nishry \cite{gho-nis-cpam} recently. More precisely, the conditional zero distribution of $\psi$ having no zeros in $\D(0,r)$ is equal to 
$$ e r^2 \,\Leb_{|z|=r}+ \mathbf 1_{|z|\geq \sqrt e r} ( 2\pi)^{-1} \, i \dd z\wedge \dd \overline z+O(r\log ^2 r) \quad \text{as}\quad r\to \infty,$$
where $\Leb_{|z|=r}$ is the Lebesgue probability measure on $b\D(0,r)$.
It is pretty  surprising that the conditional zero distribution has very little mass on $\{r< |z|< \sqrt e r\}$. This is so called the \textit{forbidden region}.  Nishry-Wennman \cite{Nishry-Wennman} generalized this result  to any simply connected domain $D$ with piecewise $\Cc^2$ boundary instead of $\D(0,r)$.

\medskip

In 1999, Shiffman-Zelditch \cite{shi-zel-cmp} extended the notion of random polynomials to compact complex manifold of dimension $k\geq 1$,  known as random sections. We will introduce it below. Since in this article, only Riemann surfaces are considered. So we just let $k=1$ for simplicity to reduce confusion.

\smallskip

 Let  $X$ be a compact Riemann surface of genus $g$ and fix a K\"ahler form $\omega_0$ on $X$. Consider a holomorphic line bundle $\oL\to X$ of positive degree $\deg(\oL)$ which is endowed with a Hermitian metric $\fh$ whose curvature form is strictly positive. Denote by $\omega$ this curvature form divided by $\deg(\oL)$. This is another K\"ahler form on $X$ which defines a probability measure on $X$ as its integral on $X$ is equal to 1. We will see later that for the proof of the main results we can, for simplicity, reduce our study to the case where $\deg(\oL)=1$.

Denote  the space of global holomorphic sections of $\oL$ by $H^0(X,\oL)$.
We fix a non-vanishing local holomorphic section  $e_{\oL}$ over some open subset $B$ of $X$. This also forms a local frame of $\oL$. Then on $B$, $\norm{e_{\oL}}_\fh=e^{-\varphi}$ for some smooth function $\varphi$ and 
$$\omega= -{1\over  \, \deg (\oL)} \ddc \log \norm{e_{\oL}}_\fh ={1\over \deg (\oL)} \ddc \varphi.$$
Here,  $\dc:=\frac{i}{2\pi}(\dbar-\partial)$ and $\ddc=\frac{i}{\pi}\partial\dbar$.

\smallskip

Let $\oL^n$ denote the $n$-th power of the line bundle $\oL$ for $n\geq 1$. It is also a positive line bundle over $X$ of degree $n\deg (\oL)$. The metric $\fh$ on $\oL$ induces a metric $\fh_n$ on $\oL^n$ whose curvature form is equal to $n\deg(\oL)\omega$. It is defined by $\norm{s^{n}(x)}_{\fh_n}:=\norm{s(x)}_\fh^n$ for any section $s$ of $\oL$. We use the notation $(\cdot,\cdot)_n$ to denote the hermitian inner product corresponding to $\fh_n$.  The metric $\fh_n$, together with $\omega_0$, induce an inner product $\lp\cdot,\cdot\rp_n$ on the complex vector space $H^0(X,\oL^n)$ of holomorphic sections of $\oL^n$:
\begin{equation*}
\lp s_1,s_2\rp_n:=\int (s_1,s_2)_n \,\omega_0     \quad  \text{for} \quad  s_1,s_2\in H^0(X,\oL^n).
\end{equation*}

 Denote by $m+1$ the dimension of $H^0(X,\oL^n)$ which is equal to $n\deg(\oL)+1-g$ (so that $m=n\deg(\oL)-g$) for $n\geq 2g-1$, according to  Riemann-Roch theorem. The unitary group $\bU(m+1)$ acts on $H^0(X,\oL^n)$ and its projectivization $\P H^0(X,\oL^n)$. There is an unique invariant probability measure on $\P H^0(X,\oL^n)$ which is given by the volume form of the Fubini-Study metric on $\P H^0(X,\oL^n)$. We denote by $V_n^\FS$ for both this volume form and the probability measure it defines.

\smallskip

The zero set of a section in $H^0(X,\oL^n)\setminus\{0\}$ doesn't change if we multiply it by a constant in $\C^*$. Therefore, we can denote by $Z_s$ the zero set of a section $s$ in $H^0(X,\oL^n)\setminus\{0\}$ or of an element $s$ in $\P H^0(X,\oL^n)$, where the points are counted with multiplicity. So $Z_s$ defines an effective divisor of degree $n\deg(\oL)$ that we still denote by $Z_s$. Denote by $[Z_s]$ the sum of Dirac masses of the points in $Z_s$ and $$\llbracket Z_s \rrbracket:= n^{-1}\deg(\oL)^{-1} [Z_s]$$ the \textit{empirical probability measure} on $Z_s$. 
 One can also  write $s=f e_L^n$ for some holomorphic function $f$ locally. Poincar\'e-Lelong formula gives that
$$[Z_s ]=\ddc\log |f|=\ddc\log \norm{s}_{\fh_n}+n \deg (\oL) \omega.$$

The study of the zeros of random sections in $H^0(X,\oL^n)\setminus\{0\}$ with respect to the complex Gaussian on $H^0(X,\oL^n)$ is equivalent to the study of the zeros of random elements of 
$\P H^0(X,\oL^n)$ with respect to the probability measure $V_n^\FS$.  From now on, we denote by $\bP_n$  this probability distribution and   $\E_n$ its expectation.

\smallskip

 Shiffman-Zelditch \cite{shi-zel-cmp} proved that   under the distribution $V^{\FS}_n$,
\begin{equation}\label{shi-zel-theorem}
\mathbf{E}_n  (\llbracket Z_s \rrbracket) \longrightarrow   \omega  \quad \text{weakly  as}\quad  n\to \infty.  
\end{equation} 
More precisely, for any test function $\phi$ on $X$, one has
$$ \lim_{n\to \infty} \int_{\P H^0(X,\oL^n)}  \lp \llbracket Z_s \rrbracket, \phi \rp  \, \dd V^{\FS}_n(s)  =\int_X \phi \,\omega.$$
After that, they made several great progress   \cite{ble-shi-zel-invent,shiiffman-jga,shi-zel-crelle,shi-zel-gafa,shi-zel-pamq,shi-zel-zh-jus} on this theory. 
The hole probability was obtained in \cite{shi-zel-zre-ind}. However, the conditional expectation of the random zeros of the hole event is still an  open problem, even for the simplest case $(X,\omega)=(\P^1, \omega_\FS)$. To make the literature complete, in addition to the above works of Shiffman-Zelditch, there are many articles related to random sections,  see e.g.\ \cite{bay-ind,bay-com-mar-TAMS,coman-marin-Nguyen,din-ma-mar-jfa,din-ma-ngu-ens,din-mar-sch-jsp,din-sib-cmh,dre-liu-mar} and the survey \cite{survey-random-hol}.

\medskip

Our aim of this article is to find the zero distribution of the hole event under  the setting of random sections for compact Riemann surfaces. To state our result, we introduce some terminologies first.

\smallskip

Let $\alpha \geq 0$. Recall that an open set $D$ in $X$ is of $\Cc^\alpha$-\textit{boundary} if for every $x\in b D$, there is a neighborhood $B$ of $x$ and a $\Cc^\alpha$-diffeomorphism $\psi: B\to \D$, where $\D$ is the unit disc in $\C$, such that 
$$\psi(D\cap B)=\D\cap \{\Im (z) >0\}, \quad    \text{and}\quad    \psi(b D \cap B)=\D\cap \R.$$
The open set $D$ is of \textit{piecewise  $\Cc^\alpha$-boundary}  if the above conditions hold for all  expect finitely many points in $b D$.

We  fix an open subset $D$ of $X$, not necessarily connected, with piecewise $\Cc^{1+\ep}$ boundary and $\overline D\not=X$. Denote by $\P H^0(X,\oL^n)_D$ the set of $s$ in $\P H^0(X,\oL^n)$ such that $Z_s\cap D=\varnothing$.
According to the Abel-Jacobi theorem, there is an integer $n_0$ such that $\P H^0(X,\oL^n)_D$ has non-empty interior, when $n\geq n_0$. We fix such an integer $n_0$ with $n_0\geq 2g-1$. Denote by $V_{n,D}^\FS$ the probability measure obtained by dividing the restriction of $V_n^\FS$ to $\P H^0(X,\oL^n)_D$ by the volume of the later set. By random section non-vanishing on $D$, we mean a random element of  $\P H^0(X,\oL^n)_D$ with respect to the probability measure $V_{n,D}^\FS$. To study the sequences of sections $\s=(s_{n_0},s_{n_0+1},\ldots)$ which do not vanish on $D$, we consider the space $\prod_{n\geq n_0} \P H^0(X,\oL^n)_D$ endowed with the probability measure $\prod_{n\geq n_0} V^\FS_{n,D}$.

\medskip

Recall that a function $\phi$ on $X$ with values in $\R\cup \{-\infty\}$ is \textit{quasi-subharmonic} if locally it is the difference of a subharmonic function and a smooth one.  This notation was introduced by Yau \cite{yau-cpam} and it is very useful  since all subharmonic functions on a compact Riemann surface are constants. It is not hard to see that if $\phi$ is quasi-subharmonic, then there exists a constant $c\geq 0$ such that $\ddc \phi\geq -c\, \omega$. When $c=1$, we call $\phi$ an \textit{$\omega$-subharmonic function}. Denote  by $\SH(X,\omega)$  the space of $\omega$-subharmonic functions.  By Stoke's formula, if $\phi\in \SH(X,\omega)$, then $\int_X \ddc \phi+\omega=\int_X \omega=1$ and thus, $\ddc \phi +\omega$ defines a probability measure on $X$.

\smallskip

For any probability measure $\mu$ on $X$, we can write $\mu=\omega+\ddc U_\mu$, where $U_\mu$ is the unique quasi-subharmonic function such that $\max U_\mu=0$. We call $U_\mu$ the \textit{$\omega$-potential of type M} of $\mu$ (M stands for Maximum). Denote by $\cM(X\setminus D)$ the set of probability measures on $X\setminus D$ and define the functional $\cI_{\omega,D}$ on $\mu\in \cM(X\setminus D)$ with value in $\R\cup\{+\infty\}$ by 
$$\cI_{\omega,D}(\mu):=-\int_X U_\mu \, \omega -\int_X U_\mu \, \dd\mu.$$
This functional was discovered by Zeitouni-Zelditch \cite{zei-zel-imrn}. Notice that the domain is slightly different from there.

\begin{theorem} \label{t:main-1}
There is a unique probability measure $\nu\in  \cM(X\setminus D)$ which minimizes the functional $\cI_{\omega,D}$ on $\cM(X\setminus D)$. Moreover, the following properties hold.
\begin{enumerate}
\item The $\omega$-potential $U_\nu$ of type M of $\nu$ is continuous  on $X\setminus bD$. It satisfies the equation
$$U_\nu=\sup\big\{ \phi  \in\SH(X,\omega):\, \phi\leq 0 \text{ on $X$ and } \phi\leq U_\nu \text{ on } D\big\},$$
and the following inequality for every $\mu\in \cM(X\setminus D)$,
$$\int_X U_\mu \, \dd\nu \leq \int_X U_\nu \, \dd\nu, \quad \text{where } U_\mu \text{ is the $\omega$-potential of type M of $\mu$}.$$ 
\item If $\nu_{bD}$ and $\nu_S$ denote the restrictions of $\nu$ to $bD$ and to $S:=\{U_\nu=0\}$, then $\nu=\nu_{bD}+\nu_S$, $\nu_{bD}\not=0$ and $\nu_S=\omega|_{S}$. 
\item Suppose moreover that $bD$ is the union of finitely many Jordan curves. If $D_0$ is a connected component of $D$ such that $\nu(bD_0)\not=0$, then $bD_0\subset \supp(\nu)$, $\overline D_0\cap S\not=\varnothing$ and $bD_0\cap S\subset \sing(D)$. In particular, $\nu$ vanishes on $D$ and on the non-empty open set $X\setminus (\overline D\cup S)$. If $\sing(D)=\varnothing$, then $\supp(\nu_S)\cap \supp(\nu_{bD})=\varnothing$. 
\end{enumerate}
\end{theorem}

We call $\nu$ the \textit{equilibrium measure} of $D$ and  $X\setminus (\overline D\cup S)$ the \textit{forbidden set}, which is always non-empty. Our second theorem says that the random zeros of $\P H^0(X,\oL^n)_D$ are equidistributed with respect to the minimizer of $\oI_{\omega, D}$.

\begin{theorem} \label{t:main-2}
For random section $s\in \P H^0(X,\oL^n)_D$ with respect to the measure $V^\FS_{n,D}$, we have 
$$\lim_{n\to\infty} \E_n(\llbracket Z_s \rrbracket) =\nu.$$
Moreover, for almost every sequence $\s=(s_{n_0},s_{n_0+1},\ldots)$ in $\prod_{n\geq n_0} \P H^0(X,\oL^n)_D$ with respect to the probability measure $\prod_{n\geq n_0} V^\FS_{n,D}$, we have 
$$\lim_{n\to\infty} \llbracket Z_{s_n} \rrbracket = \nu.$$
\end{theorem}

From the above theorems, we see that the weak limit is independent of $\omega_0$.  For the special case that $D$ is empty, observe that $\omega$ is the minimizer of $\oI_{\omega, \varnothing}$. Thus, the weak limit of $\E_n (\llbracket Z_s \rrbracket)$ will be $\omega$.  This  is exactly \eqref{shi-zel-theorem} above.

The appearance of the forbidden set is  surprising here as well as the Gaussian entire function case \cite{gho-nis-cpam,Nishry-Wennman}. In particular, when $D$ is an open disc in $X$. The forbidden set separates $\supp(\nu_S)$ and $bD$.   Besides,  the measure $\nu_S=\omega|_S$ is also unexpected. Because the hole probability of the conditional set $\P H^0(X,\oL^n)_D$ is exponentially decreasing on $n$ (c.f.\ \cite[Theorem 1.4]{shi-zel-zre-ind}). There is almost no hope that the zeros of $\P H^0(X,\oL^n)_D$ away from $D$ could  share the original distribution of the zeros of $\P H^0(X,\oL^n)$. At the end of this article, we also give some examples that $\nu_S=0$.

\smallskip

For the proof of Theorem \ref{t:main-1}, we borrow some idea from \cite{Nishry-Wennman}. However, instead of considering an equivalent obstacle problem, we work on the minimizer problem of $\oI_{\omega, D}$ directly using potential theory and also some technical from pluripotential theory. Of course,  new  tools need to be built since the metric is not flat in our case.

The proof of Theorem \ref{t:main-2} is based on a large deviation principle related to $\big(\P H^0(X,\oL^n), V_{n,D}^{\FS}\big)$ and the functional $\oI_{\omega,D}$. Such a type of theorem was firstly proposed by  Zelditch in his remarkable works \cite{zei-zel-imrn,zel-imrn}, which  provides us a new way to study random holomorphic sections on Riemann surfaces. This subject lies in the overlapping of complex geometry and probability theory. Before that,  the approaches using to study random sections are more  concentrated on complex geometry.      In this paper, instead of $\P H^0(X,\oL^n)_D$, we look at the  space $\cM (X\setminus D)$, which can be viewed as an ``extension" of $\P H^0(X,\oL^n)_D$, by using the empirical measure $\llbracket Z_s \rrbracket$ for every $s\in \P H^0(X,\oL^n)_D$. In  $\cM (X\setminus D)$, there is a nature Wasserstein metric, which allows us to measure the distance between  holomorphic sections in  $\P H^0(X,\oL^n)_D$.

\smallskip

In this article, we will mainly consider the case that when the genus $g$ of $X$ is positive. For $X=\P^1$, the proof will be much simpler, and all arguments are contained in the proof of the positive genus case. One difficulty of the case $g\geq 1$ is that the number of zeros of any holomorphic section of $\oL^n$ is strictly greater than the dimension of $\P H^0(X,\oL^n)$. It is very hard to describe the conditional set $ \P H^0(X,\oL^n)_D$. Thus, the weak convergence of $\E_n(\llbracket Z_s \rrbracket)$ with respect to  $V^\FS_{n,D}$ requires some delicate estimates.

\medskip

\noindent\textbf{Organization.}  The paper is organized as follows. In Section \ref{s:qpot}, we introduce the quasi-potentials of a probability measure and prove some simple facts on it, which are the main tools used in the proof of Theorem \ref{t:main-1}, given in Section \ref{sec:proof-theorem -1}.

Section \ref{s:bundle} concerns the study of Abel-Jacobi theory. This is the main difference between $g=0$ and $g\geq 1$. We give a parametrization $\Ac_m$ of the zeros of the sections in $\P H^0(X, \oL^{m+g})$. The density of zeros of  $\P H^0(X, \oL^{m+g})$ can be explicitly written down by using $\Ac_m$, which was contributed by  Zelditch. 

Section \ref{s:proof-main-2} is devoted to the proof of Theorem \ref{t:main-2}, where  all the technical pre-request  lemmas are packed  into Sections \ref{s:Wasserstein} and \ref{sec:regularization}.    At the beginning of Section \ref{s:proof-main-2}, we prove a large deviation principle and also discuss the difficulties for finding the equilibrium measure of hole event.  A remark on hole probability is also given there.

Some examples of $\P^1$ and a torus are  included in Section \ref{s:example}.  We state a interesting result that $\nu_S$ may vanish at the end of the section.

\medskip

\noindent\textbf{Notations:}
The symbols $\lesssim$ and $\gtrsim$ stand for inequalities up to a positive multiplicative constant. If both hold, then we write $\simeq$. The dependence of these constants on certain parameters, or lack thereof, will be clear from the context. We denote by $\D$ the unit disc in $\C$ and $\D(a,r)$ the open disc with center $a$ and radius $r$.  For open balls in other Riemann surfaces or complex manifolds, we use $\B(a,r)$ instead.   We will use $\Vol$ to denote the integral of a measurable set with respect to the Lebesgue volume form on a complex manifold.

%%%%%%%%%%%%%%%%%%%%%%%%

\medskip

\section{Quasi-potentials} \label{s:qpot}

We first start with some properties for general probability measures on $X$ and then focus on the equilibrium measure of $D$.

\smallskip 

 Consider the product complex manifold $X \times X$ and let $\pi_1,\pi_2$ be the two canonical projections. Denote by $\Delta$  the diagonal of $X \times X$. Notice that $\pi_1^*\omega +\pi_2^*\omega$ is a canonical K\"ahler form on $X\times X$.

\begin{lemma} \label{l:Green}
There is a symmetric  function $G(x,y)$ on $X\times X$, smooth outside $\Delta$, such that for every $x\in X$ we have
$$\int_X G(x,\cdot) \, \omega = 0 \qquad \text{and} \qquad \ddc G(x,y) = [\Delta]-\omega(x) - \omega(y) -\alpha(x,y),$$
where $\alpha$ is an $(1,1)$-form which is a finite sum of products of a smooth  $1$-form in $x$ and a smooth $1$-form in $y$. Moreover, the function $\varrho(x,y):=G(x,y)-\log \dist(x,y)$ is Lipschitz.  
\end{lemma}
\proof
By K\"unneth formula, we have
$$H^{1,1}(X\times X,\C)\simeq \sum_{q,l} H^{q,l}(X,\C) \otimes H^{1-q,1-l}(X,\C).        $$
There exists a real smooth $(1,1)$-form
$\alpha(x,y) =\sum_{j=1}^k \beta_j(x)\wedge \gamma_j(y)$
with $\beta_j,\gamma_j$ closed smooth $1$-forms, and a function $G'(x,y)$ such that 
$$\ddc G'(x,y)= [\Delta] -\omega(x)-\omega(y) -\alpha(x,y).$$
Replacing $\alpha(x,y)$ by $1/2 \, \alpha(x,y)+1/2 \, \alpha(y,x)$, we can assume $\alpha(x,y)$ is symmetric and hence $G'(x,y)$ is also symmetric.

Since $\pi_2^*\omega$ is closed, we have for any smooth test function $\phi$ on $X$, 
\begin{align*}
\int \phi (x)  \, (\pi_1)_*\big(\ddc (G'  \cdot \pi_2^*\omega)\big)=\int \phi (x)  \,  (\pi_1)_*\big( \ddc G' \cdot \pi_2^* \omega      \big)  
=\int \phi (x)\, \omega -\int \phi (x)\, \omega =0.
\end{align*}
This implies that $\ddc(\pi_1)_* (G'  \cdot \pi_2^*\omega)=0$ and  $(\pi_1)_* (G'  \cdot \pi_2^*\omega)$ is some constant function $c$. We set $G(x,y):=G'(x,y)-c$. Easy to check that $\int_X G(x,\cdot) \, \omega = 0$.

\smallskip

Clearly, the function $\varrho$ is Lipschitz outside $\Delta$. Since the degree of $\alpha$ on $y$ is not $(1,1)$, we have
for every fixed $x$,  $\ddc G(x,\cdot) =\delta_x -\omega$. Together with that 
 $\ddc \log\dist (x,\cdot)= \delta_x + \eta$ for some smooth $(1,1)$-form $\eta$ near $x$, we conclude that $G(x,\cdot)-\log\dist (x,\cdot)$ is Lipschitz.
\endproof

The function $G$ above is called the \textit{Green function} of $(X,\omega)$.
For $\mu$  a probability measure on $X$,  define $U'_\mu(x):=\int_X G(x,\cdot) \, \dd\mu$.

\begin{corollary} \label{c:Green}
We have  $\int_X U'_\mu \,\omega=0$ and $U_\mu=U'_\mu-\max U'_\mu$.
\end{corollary}
\proof
The first equality follows from $\int_X G(x,\cdot) \, \omega = 0$.
From the definition of $\omega$-potential of type M, we only need to show $\ddc U'_\mu =\mu -\omega$. By Lemma \ref{l:Green},
\begin{align*}\ddc U'_\mu = \ddc\int G(x,\cdot) \, \dd\mu =\ddc\int G'(x,\cdot) \, \dd\mu =\int \ddc_x  G'(x,y) \, \dd\mu (y) \\
=\int \delta_y - \omega (x) \,\dd \mu(y) =\mu-\omega.
\end{align*}
Here, the term involving $\alpha$ disappear because its degree on $x$ is not $(1,1)$.
\endproof

More generally, for any positive measure $\mu$ (may not be probabilistic) on $X$, we can define 
\begin{equation}\label{defn-potentil-type-I}
U'_\mu(x):=\int_X G(x,\cdot) \, \dd\mu  \quad\text{and}\quad   U_\mu:=U'_\mu-\max U'_\mu.
\end{equation}

Immediately, the functional $\oI_{\omega,D}$ can be rewrote as
\begin{equation}\label{second-formula-I}
\cI_{\omega,D}(\mu)= -\int_X  U'_\mu \,\dd \mu +2 \max U'_\mu.
\end{equation}
One advantage of this formula is the following commutative property.

\begin{lemma}\label{commut-potential}
	For two positive measures $\mu_1,\mu_2$ on $X$, we have 
	$$ \int_X  U'_{\mu_1}\,\dd \mu_2=\int_X U'_{\mu_2}\,\dd \mu_1.$$
\end{lemma}

\begin{proof}
	Let $c_1,c_2$ be the masses of $\mu_1,\mu_2$ respectively. Then by definition,
	$$\ddc U'_{\mu_1}=\mu_1-c_1 \omega \quad \text{and}\quad \ddc U'_{\mu_2}=\mu_2-c_2 \omega .$$
	Using Stoke's formula, we have 
	\begin{align*} \int_X  U'_{\mu_1}\,\dd \mu_2= \int_X  U'_{\mu_1}\,& (\ddc U'_{\mu_2}+c_2\omega)=\int_X  U'_{\mu_1}\, \ddc U'_{\mu_2}\\
	&=\int_X  U'_{\mu_2}\, \ddc U'_{\mu_1}=\int_X  U'_{\mu_2}\, (\ddc U'_{\mu_1}+c_1\omega)=\int_X U'_{\mu_2}\,\dd \mu_1.
	\end{align*}
	Here we use that $\int_X U'_{\mu_1}\,\omega=\int_X U'_{\mu_2}\,\omega=0$.
\end{proof}

We have the following regularity for a probability measure near the maximum of its quasi-potential.

\begin{lemma} \label{lem-regularity}
Let $\mu$ be any probability measure on $X$. Let $x\in X$ such that $U_\mu(x)=0$. 
Then  $\mu\big(\B(x,r)\big)\leq C r^2$ for some constant $C>0$ independent of $r$ and $x$.
\end{lemma} 
\begin{proof}
	Since this problem is local, we may assume that $U_\mu$ is an $\omega$-subharmonic function on $\D$ and $0$ is a maximum of $U_\mu$ in $\D$ with $U_\mu(0)=0$. 
	For every $r<1/2$, we consider the dilation map $\lambda_r: \overline\D(0,r)\to \overline\D, z\mapsto z/r$ and let $\Phi_r$  be the function on $\overline\D$ defined by 
	$$\Phi_r(z):=(\lambda_r)_*(U_\mu|_{\D(0,r)})(z)=U_\mu(r z).$$
	By assumption, $\Phi_r \leq 0$ and $\max \Phi_r=0$ on $\overline\D$.
	 Moreover, 
	  \begin{align*}
	\ddc (r^{-2} \Phi_r)=r^{-2}\ddc (\lambda_r)_*(U_\mu|_{\D(0,r)})\geq r^{-2}  (\lambda_r)_* (-\omega)  \geq -c\omega
	 \end{align*} 
	 for some $c>0$.
	 Hence the functions $c^{-1}r^{-2}\Phi_r,0<r<1$ are in a compact family of quasi-subharmonic functions on $\overline \D$ and thus, we have 
	$\big\|\ddc (c^{-1}r^{-2} \Phi_r ) \big\|\leq C'$
	for some constant $C'>0$ independent of $r$ and $x$.
	Finally, using that 
	$\ddc (\lambda_r)_*(U_\mu|_{\D(0,r)})=(\lambda_r)_*(\ddc U_\mu|_{\D(0,r)})$ and $\omega\big(\D(0,r)\big)\lesssim r^2$,
	we conclude that $ (\ddc U_\mu+\omega)\big(\D(0,r)\big)\leq Cr^2$ for some $C>0$. 
\end{proof}

Note that the lemma above still holds for any $x$ which is a local maximum of $U_\mu$.

\medskip

Now let $\mu$ be a probability measure $\mu \in \cM(X\setminus D)$. Consider the following envelop 
$$\widehat U_\mu:=\sup \big \{ \phi \in\SH(X,\omega):\, \phi\leq 0 \text{ on } X, \, \phi \leq  U_\mu   \text{ on } \overline D       \big\}^*$$
and the probability measure 
$$\widehat\mu:=\ddc \widehat U_\mu+\omega.$$
Here, the notation $*$ means the upper semicontinuous regularization.

\begin{proposition} \label{p:envelop}
The function $\widehat U_\mu$ is continuous on $X\setminus bD$ and the open set $\{\widehat U_\mu<0\}\setminus \overline D$ is non-empty. Moreover, we have
$\widehat\mu(bD)\not =0$, $\widehat\mu=0$ on $\{\widehat U_\mu<0\}\setminus bD$ and the restriction of $\widehat\mu$ to the compact set $\{\widehat U_\mu=0\}$ is equal to the restriction of $\omega$ to this set. In particular, the support of $\widehat\mu$ is contained in the union of $bD$ with $\{\widehat U_\mu=0\}\setminus\overline D$. 
\end{proposition}

We separate the proof of the proposition above into several lemmas below.
 The support of $\widehat\mu$ is given by the next lemma. It should be classical, but we present the proof here for the convenience of the readers. 

\begin{lemma}\label{prop-ddc-eql-omega}
	We have $\ddc \widehat U_\mu =-\omega$ on $D$ and $\{\widehat U_\mu<0\}\setminus \overline D$.
\end{lemma}

\begin{proof}
	Since $U_\mu$ is upper semicontinuous, then condition $\phi\leq U_\mu$ on $\overline D$ gives $\widehat U_\mu\leq  U_\mu$ on $\overline D$. 
	On the other hand, by definition, $U_\mu$ itself is $\omega$-subharmonic and non-positive on $X$, which implies $U_\mu\leq \widehat U_\mu$ on $X$. Therefore, $\widehat U_\mu= U_\mu$ on $\overline D$ and thus, $\ddc \widehat U_\mu =-\omega$ on $D$.
	
	\smallskip

	Since $\widehat U_\mu$ is upper semicontinuous,  $\{\widehat U_\mu<-\vep\}\setminus \overline D$ is open for any $\vep>0$. Fix a small $\vep>0$ and   take two  open balls $B_1,B_2$ with  $\overline B_1\subset B_2\subset \{\widehat U_\mu<-\vep\}\setminus \overline D$. We will prove that $\ddc \widehat U_\mu =-\omega$ on $B_1$.
	
	After shrinking $B_2$, we may assume that $\omega$ admits a smooth local potential $\varphi$, i.e., $\ddc\varphi=\omega$ on $B_2$. Moreover, after shrinking $B_2$ again and adding a constant to $\varphi$, we can also assume $-\vep\leq \varphi\leq 0$ on $B_2$. Then $\widehat U_\mu +\varphi$ is subharmonic on $B_2$. By \cite[Proposition 9.1]{bedford-1982}, we can take a function $\Psi$ which is subharmonic on $B_2$, harmonic on $B_1$ and it satisfies $\Psi=\widehat U_\mu+\varphi$ on $B_2\setminus B_1$ and $\Psi\geq \widehat U_\mu+\varphi$ on $B_2$ ($\Psi$ is also the solution of Dirichlet problem on $B_1$).

	Now consider the function $\Phi$ defined by 
	$$\Phi=\widehat U_\mu   \quad\text{on}\quad   X\setminus B_1       \quad\text{and}\quad        \Phi=\Psi-\varphi   \quad\text{on}\quad   B_2.    $$
	It is well-defined because $\widehat U_\mu=\Psi-\varphi$ on $B_2\setminus B_1$. Using that $\Psi$ is subharmonic on $B_2$, we have $\ddc \Phi=\ddc\Psi-\ddc \varphi \geq -\omega$ on $B_2$. Therefore, $\Phi$ is $\omega$-subharmonic on $X$. 
	
	\medskip
	\noindent {\bf  Claim.} $\Phi\leq 0$ on $X$.
	\proof[Proof of Claim] We only need to show $\Phi\leq 0$ on $B_1$.  On the boundary of $B_1$, we have $\Psi=\widehat U_\mu+\varphi\leq -\vep$. So maximal modulus principle gives that $\Psi\leq -\vep$ on $B_1$. Therefore, on $B_1$, we have $\Phi=\Psi-\varphi\leq -\vep +\vep=0$. 
	\endproof
	
	By the above claim and  the definition of $\widehat U_\mu$, we have $\Phi\leq \widehat U_\mu$ on $X$. On the other hand,  $\Phi\geq \widehat U_\mu$ on $X$ by the definition of $\Psi$.  Hence $\Phi=\widehat U_\mu$ on $X$. In particular, on $B_1$, we have $\ddc \widehat U_\mu =\ddc\Phi =\ddc \Psi-\ddc \varphi =-\omega$. This ends the proof of the lemma.
\end{proof}

From the above lemma, we deduce that the support of $\widehat \mu$ is contained in the union of $bD$ and $\{\widehat U_\mu=0\}\setminus \overline D$. 

\begin{lemma}\label{forbidden-nonempty}
	The set  $\{\widehat U_\mu<0\}\setminus\overline D$ is non-empty. Hence it contains a non-empty open subset.
\end{lemma}

\begin{proof}
	Suppose for contradiction, for all $x\in X\setminus \overline D$,  $\widehat U_\mu(x)=0$. The upper semi-continuity of $\widehat U_\mu$ yields that $\widehat U_\mu(x)=0$ for $x\in b(X\setminus \overline D)$. Recall that $D$ has piecewise $\Cc^{1+\ep}$ boundary, which implies $b(X\setminus \overline D)=bD$. Therefore, $\widehat U_\mu(x)=0$ for $x\in bD$. But notice that $-\widehat U_\mu$ is subharmonic on $D$ because $-\ddc \widehat U_\mu=\omega$ on $D$. By maximal modulus principle, $\widehat U_\mu=0$ on $D$. Contradiction.
\end{proof}

In the following, we will discuss the behavior of $\widehat U_\mu$ on the boundary of $\{\widehat U_\mu=0\}$ outside $\overline D$ and prove that $\widehat\mu$ has positive mass on $bD$.  We need the following comparison inequality \cite[Proposition 6.11]{demailly:ptbook} for subharmonic functions.

\begin{lemma}   \label{comparison}
	Let $B$ be an open subset of $X$, and let $\varphi$ and $\psi$ be bounded subharmonic functions on $B$, then as positive measures,
	$$ \mathbf 1_{\varphi\leq \psi}\ddc \psi +\mathbf 1_{\psi<\varphi}\ddc\varphi\leq \ddc \max(\varphi,\psi).$$
\end{lemma}

Using the comparison inequality, we  deduce the structure of $\widehat\mu$ on $bD$.

\begin{lemma}\label{prop-mass-nu}
	$\widehat U_\mu$ is continuous on $X\setminus bD$ and  the mass of $\widehat\mu$ on $bD$ is non-zero.	
\end{lemma}

\begin{proof}
	Clearly, $\widehat U_\mu$ is continuous on $D$.
	We fix a small open ball $B\subset X\setminus \overline D$. Again, let $\varphi$ be a local potential of $\omega$ on $B$. Then $\widehat U_\mu +\varphi$ is subharmonic on $B$ and we have $\widehat U_\mu+\varphi \leq \varphi$.
	As a function on $B\cap \{\widehat U_\mu=0\}$, $\widehat U_\mu +\varphi$ is continuous and it is a local potential of the measure $\widehat \mu$ on $B$. Using the classical result \cite[Theorem III.2]{tsuji-book}, we get that $\widehat U_\mu +\varphi$ is continuous on $B$. This gives the first assertion.

	Applying Lemma \ref{comparison} to $\widehat U_\mu+\varphi$ instead of $\psi$, we obtain
	$$\mathbf 1_{\widehat U_\mu = 0} \ddc (\widehat U_\mu+\varphi)+1_{\widehat U_\mu < 0} \ddc \varphi\leq \ddc \varphi.$$
	Consequently,  $\ddc(\widehat U_\mu +\varphi)\leq \ddc \varphi=\omega$ on $B \cap \{\widehat U_\mu=0\}$ as positive measures. Since $B$ is arbitrary chosen outside $\overline D$, we conclude that the probability measure $\widehat\mu$ is bounded by the K\"ahler form $\omega$ on $\{\widehat U_\mu=0\} \setminus \overline D$.

	Using that $\widehat\mu$ has no mass on $D$ and $\widehat\mu,\omega$ both are probability measures, we conclude that $\widehat\mu$ should have positive mass on $bD$. 
\end{proof}

If $\omega$ has a real analytic local potential, using a theorem of Sakai \cite{sak-ASNS}, one can deduce that locally, $b \{\widehat U_\mu=0\}$ is a piecewise analytic curve or are countably many points, which implies that the $\omega$-measure of $b \{\widehat U_\mu=0\}$ is $0$, see also \cite[Section 2.9]{Nishry-Wennman}. But it is not the case for general $\omega$.

\begin{lemma}\label{prop-nu-S}
	We have $\ddc \widehat U_\mu=0$ on $\{\widehat U_\mu=0\}\setminus bD$ as a measure. In other words, the restriction of $\widehat\mu$ to $X\setminus \overline D$ is equal to $ \omega|_{\{\widehat U_\mu=0\}}$. 
\end{lemma}

\begin{proof}
	It is enough to prove the equality on $\{\widehat U_\mu=0\}\setminus \overline D$ since we know $\widehat\mu$ has no mass on $D$. It is also clear for all the points in $\{\widehat U_\mu=0\}^o$. We only need to consider the points in $b \{\widehat U_\mu=0\}$. 
	Take any small open ball $B$ in $X\setminus \overline D$. Let $0\leq \chi\leq 1$ be a measurable function on $B$ such that 
	\begin{equation}\label{defn-chi}
	\ddc (\widehat U_\mu+\varphi) = \chi\omega.
	\end{equation}
	We know that $\chi=0$ on $B \setminus \{\widehat U_\mu=0\}$ and $\chi=1$ on $B\cap \{\widehat U_\mu=0\}^o$.  Since the problem is local, we can use local coordinates to assume that $B=\D$ and $\omega=f(z) \, i\dd z\wedge  \dd \overline z$ for $z\in \D$, where $f$ is a smooth positive function on a neighborhood of  $\D$.
	Lebesgue density theorem gives that 
	\begin{equation}\label{lebesgue point}
	\lim_{\delta\to 0} {1\over \Vol(\D(x,\delta))} \int_{\D(x,\delta)\cap \{\widehat U_\mu=0\}} \omega =1 
	\quad\text{for a.e. }\, x\in \D\cap \{\widehat U_\mu=0\}.
	\end{equation}
	Take any $x\in \D\cap\{\widehat U_\mu=0\}$ satisfying \eqref{lebesgue point} and let $0<\delta<1/2$.
	Consider the dilation map $\lambda_\delta: \D(x,\delta)\to \D(0,1)$ sending $x$ to $0$.  Equality \eqref{lebesgue point} gives that $$\Vol \big(\D(x,\delta)\big)-\Vol \big(\D(x,\delta)\cap \{\widehat U_\mu=0\}\big) = o(\delta^2) \quad\text{as}\quad \delta \to 0.$$
Thus,  $\delta^{-2}(\lambda_\delta)_*(\widehat U_\mu|_{\D(x,\delta)})$ converges to $0$ in $L^1$ as $\delta\to 0$. It follows that 
	$$\delta^{-2}\ddc\big[(\lambda_\delta)_*(\widehat U_\mu|_{\D(x,\delta)})\big] \to 0  \quad\text{weakly as}\quad \delta\to 0.$$
	Pulling back last current by $\lambda_\delta$, yields
	$$\lim_{\delta\to 0}  {1\over\delta^2   }\int_{\D(x,\delta)} \ddc \widehat U_\mu=0.$$
	
	On the other hand, since $\omega=f(z) \, i\dd z\wedge  \dd\overline z$ for $z\in \D$ and $f$ is positive smooth, we have 
	$\delta^2\simeq\delta^2\int_{\D} \omega\simeq \int_{\D(x,\delta)}\omega  \quad\text{for}\quad  0<\delta<1/2$.
	Therefore, we conclude that 
	$$\lim_{\delta\to 0} {1\over \Vol(\D(x,\delta))} \int_{\D(x,\delta)} \ddc \widehat U_\mu=0.$$ 
	Combining with \eqref{defn-chi}, we obtain
	$$\lim_{\delta\to 0} {1\over \Vol(\D(x,\delta))} \int_{\D(x,\delta)}\chi \omega =1.$$
	Last equality holds for a.e.\ $x\in \D\cap \{\widehat U_\mu=0\}$.
	By Lebesgue differential theorem, we have $\chi=1$ a.e.\ on $\D\cap \{\widehat U_\mu=0\}$. This finishes the proof of the lemma.
\end{proof}

We have completed the proof of Proposition \ref{p:envelop}. Next, we come back to the discussion of equilibrium measure $\nu$. One will see that $\nu$ has the same properties as $\widehat \mu$ above.
The measure $\nu$ is defined using the following lemma. 

\smallskip

We fix a topology for the space $\cM(X)$.
Define the \textit{Wasserstein distance} between two probability measures $\mu_1, \mu_2$ on $X$ by
$$\dist_W(\mu_1,\mu_2):=\sup_\phi \Big|\int_X \phi \, \dd\mu_1 - \int_X \phi \, \dd\mu_2\Big|,$$
where $\phi$ runs in the set of all Lipschitz functions on $X$ with Lipschitz constant $1$ with respect to the metric induced by $\omega_0$.

\begin{lemma}\label{l:convex}
		As functionals on $\mathcal M (X\setminus D)$, $\mu\mapsto \max U'_\mu$ is continuous,
	$\mu\mapsto -\int_X  U'_\mu \,\dd\mu$ is lower semicontinuous,
 $\oI_{\omega,D}$ is lower semicontinuous and strictly convex. In particular, $\oI_{\omega,D}$ admits a unique minimizer $\nu$ that we call the equilibrium measure of $D$.
\end{lemma}
\proof
Using the formula \eqref{second-formula-I},
one can replace $K$ by $X$ in 
\cite[Proposition 24]{zei-zel-imrn} and \cite[Proposition 9]{zel-imrn},  getting the desired identities on $\cM (X)$. In particular, the same properties hold on $\cM(X\setminus D)$.
\endproof

Observe immediately that $\nu$ cannot charge on polar sets, otherwise $\oI_{\omega,D}(\nu)=\infty$.
 We have the following characterization of $\nu$.

\begin{lemma} \label{l:min}
For any probability measure $\mu\in \cM(X\setminus D)$, we have
$$\int_X U_\mu \, \dd\nu \leq \int_X U_\nu \, \dd\nu.$$
\end{lemma}
\proof
 We may  assume $\int_X U_\mu \,\dd \nu>-\infty$, otherwise the  inequality holds trivially. 
	Consider the probability measures $\mu_t:=(1-t)\nu+t\mu\in \mathcal  M(X\setminus D)$ for $0<t<1$. Observe that  $U'_{\mu_t} = (1-t)U'_\nu +t U'_\mu$.  By Lemma \ref{commut-potential} and \eqref{second-formula-I}, we have
	 \begin{align*}
	&\oI_{\omega,D}(\mu_t)=-\int_X \big((1-t)U'_\nu+tU'_\mu\big) \,\dd\big((1-t)\nu+t\mu \big)+2\max \big((1-t)U'_\nu+tU'_\mu\big)\\
	&\leq  -(1-t)^2\int_X U'_\nu \,\dd\nu-2(1-t)t\int_X U'_\mu\,\dd\nu -t^2\int_X U'_\mu\,\dd\mu  +2(1-t)\max U'_\nu+2t\max U'_\mu \\
	&=\oI_{\omega,D}(\nu)+2t\Big(\int_X U'_\nu \,\dd\nu -\int_X U'_\mu\,\dd\nu-\max U'_\nu  +\max U'_\mu \Big)+O(t^2).
	\end{align*}
	
	Since $\nu$ is the   minimizer of $\oI_{\omega,D}$, letting $t\to 0$, we deduce that 
  \begin{equation}\label{Unu-Umu-max}
   \int_X U'_\nu \,\dd\nu -\int_X U'_\mu\,\dd\nu-\max U'_\nu  +\max U'_\mu \geq 0. 
  \end{equation}
	This ends the proof.
\endproof

\begin{corollary}\label{cor-U-mu-widehat}
	We have $\widehat U_\nu =U_\nu$ and $\widehat \nu =\nu$.
\end{corollary}

\begin{proof}
	From Proposition \ref{p:envelop}, we see that $\widehat \nu \in \cM(X\setminus D)$.
	 So Lemma \ref{l:min} gives  that  $$\int_X  U_{\widehat \nu} \,\dd \nu \leq \int_X U_\nu \,\dd \nu.$$ But by definition, $U_{\widehat\nu} =\widehat U_{\nu} \geq U_\nu$. We conclude that $\widehat U_\nu =U_\nu$ and the lemma follows.
\end{proof}

Therefore, in view of  Proposition \ref{p:envelop}, we have $\nu=\nu_{bD}+\nu_S$ with $\nu_{bD}\not=0$ and $\nu_S=\omega|_{S}$, where $S=\{ U_\nu =0\}$. Actually $\nu$ has no mass on $S\cap bD$.

\begin{lemma}\label{lema-S-no-mass}
		We have $\nu(S \cap bD)=0$.
\end{lemma}

\begin{proof}
 Since $D$ is piecewise $\Cc^{1+\ep}$, $bD$ is of Hausdorff dimension $1$. For any $\delta>0$,	we can cover $bD$ by finitely many  open discs $\{B_j\}_{ 1\leq j\leq k}$  such that $\sum_{1\leq j\leq k}\omega(B_j)<\delta$. Applying Lemma \ref{lem-regularity}, we deduce that
 $$\nu (S\cap bD)\lesssim \sum_{1\leq j\leq k}  \omega(B_j)<\delta. $$
By taking $\delta \to 0$ we finish the proof.
\end{proof}

\medskip

\section{Proof of Theorem \ref{t:main-1}} \label{sec:proof-theorem -1}

In this section, we assume that $bD$ is the union of finitely many Jordan curves. By  \cite[Theorem I.12]{tsuji-book}, every continuous function on $bD$ can be extended to a harmonic function on $D$ which is continuous up to $bD$.
We start with 
the following crucial proposition, which will be applied several times. It provide a efficient way to determine whether a probability measure is the equilibrium measure $\nu$.  The original idea comes from Nishry-Wennman \cite{Nishry-Wennman}. However, we need to overcome three extra difficulties:
\begin{enumerate}
	\item  the K\"ahler form $\omega$ may not have real analytic local potential;
	\item   $D$ is not assumed to be simply connected, or even connected;
	\item  $b D$ may have cusps.
\end{enumerate} 

\begin{proposition} \label{l:perturbation}
	Let $h$ be a continuous function on $\overline D$ and harmonic on $D$. Assume that $h+U_\nu\leq 0$ on $\overline D$. Then 
	$$\int_{bD} h \, \dd\nu \leq 0.$$ 
\end{proposition}

The proof of Proposition \ref{l:perturbation} is rather long. After dividing by  some positive constant, we may assume $|h|\leq 1$.  For every $0<\vep<1$, we   define the perturbation 
$$U^{\vep h}:= \sup \big \{ \phi \in\SH(X,\omega):\, \phi\leq 0 \text{ on } X, \, \phi\leq   U_\nu +\vep h  \text{ on } \overline D  \big\}^*.$$ 
Denote by $\nu_{\vep h}:=\ddc U^{\vep h}+\omega$ the probability measure associated to $U^{\vep h}$.  Observe that $U_\nu +\vep h \leq 0$ on $\overline D$. We have the following structure of $\nu_{\vep h}$ similar as $\widehat \mu$ in  Proposition \ref{p:envelop}. However, one cannot apply that proposition directly since $h$ is not globally defined.

\begin{proposition}\label{prop-struct-nu-vep}
	The probability measure $\nu_{\vep h}$ is supported in $S_{\vep h}\cup bD$, where $S_{\vep h}:=\{U^{\vep h}=0\}$, and $\nu_{\vep h}|_{X\setminus \overline D}=\omega |_{S_{\vep h}}$. 	For $x\in bD$, if $U^{\vep h}(x)<  U_\nu(x)+\vep h(x)$, then $x\notin \supp(\nu_{\vep h})$.
\end{proposition}

\begin{proof}
	Since $  U_\nu+\vep h$ is upper semicontinuous on $D$, then condition $\phi\leq   U_\nu +\vep h$ on $\overline D$ gives that $U^{\vep h} \leq   U_\nu+\vep h$ on $\overline D$. 
	
	To prove $\supp(\nu_{\vep h})\subset S_{\vep h}\cup bD$, we consider two cases. Either $U^{\vep h}(z) =   U_\nu (z)+\vep h (z)$ for some  $z\in D$, or $U^{\vep h}(z) <   U_\nu (z)+\vep h (z)$ for all $z\in D$.
	
	If $U^{\vep h}(z) =   U_\nu (z)+\vep h (z)$ for some point $z\in D$.
	Notice that $U^{\vep h}- (  U_\nu+\vep h)$ is subharmonic in $D$.  By maximal modulus principle,  $U^{\vep h} =   U_\nu +\vep h$ on $D$. Hence $\nu_{\vep h}$ is supported outside $D$.
	
	If  $U^{\vep h}(z) <   U_\nu (z)+\vep h (z)$ for every $z\in D$.  Then as in  the proof of Lemma \ref{prop-ddc-eql-omega}, for every $z\in D$, we can find a function $ \Phi$ in the admissible family of upper envelop such that $\ddc \Phi=-\omega$ on some small open ball $B \ni z$ and $  \Phi=U^{\vep h}$ on $B$. This also gives that $\nu_{\vep h}$ is supported outside $D$.
	
	The equality $\nu_{\vep h}|_{X\setminus \overline D}= \omega|_{S_{\vep h}}$  follows from the same lines as in  Lemma \ref{prop-nu-S} after replacing $  U_\nu, S$ by $  U_\nu+\vep h, S_{\vep h}$ respectively.
	
	\smallskip
	For the second assertion, we use the same argument as well.
	If $x\in bD$ and  $U^{\vep h}(x)<   U_\nu(x)+\vep h(x)$, we can find a function $ \Phi'$ in the admissible family of upper envelop for $U^{\vep h}$ such that $\ddc \Phi'=-\omega$ on some small open ball $B \ni x$ and $  \Phi'=U^{\vep h}$ on $B$. The implies that the mass of $\nu_{\vep h}$ near $x$ is $0$.
\end{proof}

Here, one really needs the condition that  $  U_\nu +\vep h\leq 0$ on $\overline D$. Otherwise, $\nu_{\vep h}$ may have mass on $D$.  The above proposition also implies that 
\begin{equation}\label{psi-vep-h-nu-vep}
 \int_{bD}\big(U^{\vep h}-  U_\nu-\vep h\big) \,\dd \nu_{\vep h}= 0    \quad\text{and}\quad  \int_{bD}\big(U^{\vep h}-  U_\nu-\vep h\big) \,\dd \nu\leq 0.
\end{equation}

\begin{lemma}\label{lem-difference-less-vep}
	For every $\vep>0$,  
	$|U^{\vep h}-  U_\nu|\leq \vep$ on $X$. 
\end{lemma}

\begin{proof}
	Since $|h|\leq 1$, we have 
		$$U_\nu -\vep \leq  U^{\vep h}\leq \sup \big \{ \phi \in\SH(X,\omega):\, \phi\leq \vep \text{ on } X, \, \phi\leq   U_\nu +\vep   \text{ on } \overline D  \big\}^*=U_\nu+\vep.$$
	This clearly implies the desired inequality.
\end{proof}

Immediately, we get that $\ddc U^{\vep h}\to \ddc U_\nu$ and $\nu_{\vep h}\to \nu$ weakly as $\vep$ tends to $0$. We also have the measure convergence for  $bD$ and $X\setminus D$.

\begin{lemma}\label{lem-nu-veph-set}
We have	$\lim_{\vep\to 0}\nu_{\vep h} (bD)=\nu (bD)$ and $\lim_{\vep\to 0}\nu_{\vep h} (X\setminus \overline D)=\nu (X\setminus \overline D)$.
\end{lemma}

\begin{proof}
Suppose for contradiction, $\lim_{\vep\to 0}\nu_{\vep h} (bD)\neq\nu (bD)$. Then we can find a sequence $\vep_j \to 0$ such that  $|\nu_{\vep_j h} (bD)-\nu (bD)| >\delta$ for some $\delta>0$

	Since $bD$ is assumed to be piecewise $\Cc^{1+\ep}$, we can find an   open set $E$ with piecewise smooth boundary containing $\overline D$ such that  $\omega(E\setminus \overline D)<\delta/4$. Indeed, one can consider the covering of $bD$ as in Lemma \ref{lema-S-no-mass}. Clearly $\omega (b E)=0$.  On $X\setminus \overline D$, both $\nu$ and $\nu_{\vep h}$ are  bounded by $\omega$, which means that $\nu(E\setminus \overline D)$ and $\nu_{\vep h}(E\setminus \overline D)$ are both bounded by $\delta/4$ and $\nu(bE)=\nu_{\vep h}(bE)=0$.
	By Portmanteau theorem, $\lim_{\vep\to 0} \nu_{\vep h} ( E)=\nu (E)$.  It follows that, for $\vep$ small enough,
	$$|\nu_{\vep h} (bD)-\nu (bD)| \leq |\nu_{\vep h} (E)-\nu (E)| +\nu_{\vep h}(E\setminus \overline D)+\nu(E\setminus \overline D) <\delta.$$
	Contradiction.
\end{proof}

\begin{lemma}\label{lemma-U-nu-U-veph}
As $\vep \to 0$, we have 
$$ \int_{X \setminus \overline D}  |U_\nu-U^{\vep h}|\,\dd \nu=o(\vep)  \quad\text{and}\quad \int_{X \setminus \overline D}  |U^{\vep h}-U_\nu|\,\dd \nu_{\vep h}=o(\vep).$$
\end{lemma}

\begin{proof}
	For every $\vep>0$, we consider the following perturbations of $U_\nu$ and $U^{\vep h}$:
	$$U^{-\vep}:=\sup \big \{ \phi \in\SH(X,\omega):\, u\leq 0 \text{ on } X, \, \phi\leq U_\nu -\vep  \text{ on } \overline D  \big\}^*,$$
		$$U^{\vep h -\vep}:=   \sup \big \{ \phi \in\SH(X,\omega):\, \phi\leq 0 \text{ on } X, \, \phi\leq  U^{\vep h}-\vep  \text{ on } \overline D  \big\}^*. $$
Since $|h|\leq 1$, it is not hard to see  that $U_\nu -\vep \leq U^{-\vep}\leq U_\nu$,  $U^{-\vep}\leq U^{\vep h}$ and   $U^{\vep h-\vep} \leq   U_\nu$.
	
		Let $\nu_{-\vep}:=\ddc U^{-\vep} +\omega$ and $S_{-\vep}:=\{U^{-\vep}=0\}$. Replacing $\widehat U_\mu$ by $U^{-\vep}$ in Proposition \ref{p:envelop}, yields that $\nu_{-\vep}=\nu_{-\vep}|_{bD}+ \omega|_{S_{-\vep}}$.
	By the same proof as Lemma \ref{lem-nu-veph-set}, we can show $\lim_{\vep \to 0} \nu_{-\vep}(X\setminus \overline D) =\nu(X \setminus \overline D)$, which gives $\lim_{\vep \to 0}\omega (S\setminus S_{-\vep})=0$.
	 On $S_{-\vep}$, $U^{-\vep}=U_\nu=0$. Thus, as $\vep\to 0$, we have
	$$ \int_{X\setminus \overline D}  | U_\nu-U^{-\vep}|\,\dd \nu =\int_{S\setminus S_{-\vep}}  ( U_\nu-U^{-\vep})\,\omega  \leq \vep \cdot \omega(S\setminus S_{-\vep})=o(\vep).$$
Similarly, 
 $$\int_{X \setminus \overline D}  |U^{\vep h}-U^{\vep h -\vep}|\,\dd \nu_{\vep h}=o(\vep)$$
	
	The first equation follows by the fact that $U^{-\vep} \leq  U^{\vep h}\leq U_\nu=0$ on $S$. The second one  follows by the fact that  $U^{\vep h -\vep}\leq  U_\nu \leq U^{\vep h}=0$  on $S_{\vep h}$.
\end{proof}

\begin{lemma}\label{lem-Uveh-Unu-veh-bD}
	As $\vep \to 0$, we have 
	$$\Big| \int_{bD} \vep h  \,\dd(\nu-\nu_{\vep h}) \Big|=o(\vep) \quad\text{and}\quad \Big|\int_{bD} \big( U^{\vep h}-  U_\nu\big)  \,\dd(\nu-\nu_{\vep h})\Big| =o(\vep).$$
\end{lemma}

\begin{proof}
	For every $\delta>0$, let $E_\delta$ be an open set with piecewise smooth boundary containing $\overline D$ such that $\omega(E_\delta \setminus \overline D)<\delta$. We extend $h$ to a continuous function $h_\delta$ on $X$ satisfying $|h|\leq 1$,  $h_\delta =h$ on $\overline D$ and $\supp(h_\delta)\subset E_\delta$.  Then 
	$$\int_{bD}h \,\dd(\nu-\nu_{\vep h}) =\int_X h_\delta  \,\dd(\nu-\nu_{\vep h}) - \int_{X\setminus \overline D} h_\delta \,\dd \nu + \int_{X\setminus \overline D} h_\delta \,\dd \nu_{\vep h}.$$
	The second and third integrals are $O(\delta)$ since $\nu$ and $\nu_{\vep h}$ are bounded by $\omega$ on $X\setminus \overline D$. The first integral is $o(1)$ as $\vep\to 0$. So we take $\delta\to 0$ and finish the proof of the first equation. 
	
	\smallskip
	
	For the second one,   by \eqref{psi-vep-h-nu-vep}, we have 
	$$\int_{bD} \big( U^{\vep h}-  U_\nu\big)  \,\dd(\nu-\nu_{\vep h}) \leq  \int_{bD} \vep h  \,\dd(\nu-\nu_{\vep h}). $$
  It remains to prove the lower bound.	Using Stoke's formula, 
		$$ \int_X \big(U^{\vep h}-  U_\nu\big) \,\dd (\nu-\nu_{\vep h})=\int_X \dd \big(U^{\vep h}-  U_\nu\big)\wedge \dc\big(U^{\vep h}-  U_\nu\big)\geq 0.  $$
	The second equation follows by Lemma \ref{lemma-U-nu-U-veph}.
\end{proof}

We conclude the following corollary from the above  lemmas.

\begin{corollary}\label{cor-energy-o-vep}
	As $\vep\to 0$, we have 
	$$\int_{X}\big( U^{\vep h} -  U_\nu \big) \,\dd\nu = \vep \int_{bD} h \,\dd \nu +o(\vep).$$
\end{corollary}

\begin{proof}
	 We decompose  the integral on the left hand side into
	$$\int_{X\setminus \overline D}\big( U^{\vep h} -  U_\nu \big) \,\dd\nu+\int_{bD} \big( U^{\vep h}-  U_\nu -\vep h\big)  \,\dd(\nu-\nu_{\vep h})  + \int_{bD} \big( U^{\vep h} -  U_\nu  -\vep h\big) \,\dd\nu_{\vep h}+\int_{bD}\vep h \,\dd \nu.$$
	The first and second terms are  $o(\vep)$ by Lemmas \ref{lemma-U-nu-U-veph} and \ref{lem-Uveh-Unu-veh-bD}. The third term is $0$ by \eqref{psi-vep-h-nu-vep}. 
	The proof of the corollary is complete.
\end{proof}

Now we can prove Proposition \ref{l:perturbation}, using Corollary \ref{cor-energy-o-vep}.

\begin{proof}[Proof of Proposition \ref{l:perturbation}]
	Recall	Lemma \ref{l:min} that for any probability measure  $\mu\in \cM(X\setminus D)$, one has
	$\int_X U_\mu \, \dd\nu \leq \int_X U_\nu \, \dd\nu$.  In particular, we can take $\mu= \nu_{\vep h}$, getting
	$$\int_X  U^{\vep h} \,\dd \nu- \int_X U_\nu  \,\dd \nu \leq 0.$$
	Combining with Corollary \ref{cor-energy-o-vep}, we obtain
	$$\vep \int_{bD} h \,\dd \nu +o(\vep) \leq 0   \quad\text{as}\quad \vep \to 0.$$
	This gives $\int_{bD} h\,\dd\nu \leq 0$.
\end{proof}

In view of Proposition \ref{l:perturbation}, given a probability measure $\mu$, if we want to show that $\mu$ is not the minimizer of $\oI_{\omega, D}$, it is enough to find a function $h$ on $\overline D$, which is harmonic on $D$, continuous on $bD$, and has positive $\mu$-integral on $bD$.
So the  following estimates on harmonic extension will be useful.

\begin{lemma}\label{lem-poisson}
	Fix a point $x_0\in b\D$ and let
	$0\leq f_r\leq 1$ be a family of continuous functions on $b \D$ such that $f_r(z)=1$ for $\dist(z,x_0)\leq r$ and $r>0$ small.   Fix a closed arc $A$ in $b\D$  and let $0\leq f_A\leq 1$ be a continuous function on $b \D$ such that $f_A(z)=1$ for $z\in A$. Let $F_r, F_A$ be the harmonic extensions of $f_r,f_A$ to $\D$ respectively. Fix also a set $\Lambda$ in $\D$ such that $ \dist\big(\Lambda, \supp(f_A)\big)\geq \alpha$
	for some constant $\alpha>0$.
	Then for all $z\in \Lambda$, we have 
	$$F_r(z)\geq C r \dist(z,b \D)     \quad\text{and}\quad  F_A(z)\leq C \dist(z,b \D),$$
	where  $C>0$ is a constant independent of $f_r, f_A$.
\end{lemma}

\begin{proof}
	The Poisson kernel on $\D$ gives that
	$$F_r(z)={1\over 2\pi} \int_0^{2\pi} f_r(e^{i\theta}){1-|z|^2\over |e^{i\theta}-z|^2 }\,\dd \theta.$$ 
	Since $f_r$ is non-negative, we have for every $z\in \D$,
	$$F_r(z)\geq \int_{e^{i\theta}\in \D(x_0,r)} {1-|z|^2 \over 2^2} \,\dd \theta  \simeq r(1-|z|) = r \dist(z,b \D).$$
	
	Similarly, since $f_A \leq 1$,  for  $z\in \Lambda$, we have
	$$F_A(z) = \int_{e^{i\theta}\in\supp(f_A)} f_A (e^{i\theta}){1-|z|^2\over |e^{i\theta}-z|^2 }\,\dd \theta \leq  \int_{e^{i\theta}\in\supp(f_A)} {1-|z|^2 \over \alpha^2} \,\dd \theta \lesssim \dist(z,b \D).  $$
	The proof of the lemma is finished.
\end{proof}

We have the analog for simply connected open set with $\Cc^{1+\ep}$ boundary. One should notice that it is false if we only assume the boundary to be $\Cc^1$.

\begin{corollary} \label{cor-poisson}
	Let $E\subset X$ be a simply connected open set with $\Cc^{1+\ep}$ boundary.
	Fix a point $x_0\in b E$ and let
	$0\leq f_r\leq 1$ be a family of continuous functions on $b D_0$ such that $f_r(z)=1$ for $\dist(z,x_0)\leq r$ and $r>0$ small. Fix a  closed arc $A$ on $b E$ and let $0\leq f_A\leq 1$ be a continuous function on $b E$ such that $f_A(z)=1$ for $z\in A$. Let $F_r, F_A$ be the harmonic extensions of $f_r,f_A$ to $E$ respectively. Fix also a  set $\Lambda$ in $E$ such that  $ \dist\big(\Lambda, \supp(f_A)\big)\geq \alpha$
	for some constant $\alpha>0$.
	Then for all $z\in \Lambda$, we have 
	$$F_r(z)\geq C r \dist(z,b E)     \quad\text{and}\quad  F_A(z)\leq C \dist(z,b E),$$
	where  $C>0$ is a constant independent of $f_r, f_A$.
\end{corollary}

\begin{proof}
	By Kellogg's theorem (c.f.\  \cite[II.Theorem 4.3]{harmonic-book}), we can find a bijective map $\psi : \overline \D \to \overline E$, such that $\psi$ is conformal on $\D$ and  $\Cc^{1+\ep}$ up to the boundary. Let us look at the functions $\psi^* F_r, \psi^* F_A$ and the sets $\psi^{-1}(x_0), \psi^{-1}(A), \psi^{-1}(\Lambda)$.  Since $\psi$ is $\Cc^{1+\ep}$, there is some constant $\alpha'>0$ such that $\dist\big( \psi^{-1}(\Lambda), \psi^{-1}(\supp(f_A)   \big) >\alpha'$.
	Applying Lemma \ref{lem-poisson} to them, we can find a constant $C'>0$ such that for every $w\in \psi^{-1}(\Lambda)$,
	$$\psi^*F_r(w)\geq C' r \dist(w,b \D)     \quad\text{and}\quad  \psi^*F_A(w)\leq C' \dist(w,b \D).$$
	Using again that $\psi$ is $\Cc^{1+\ep}$,  we have for every $z\in \Lambda$, 
	$$F_r(z)\geq C r \dist(z,b E)     \quad\text{and}\quad F_A(z)\leq C \dist(z,b E)$$
	for some constant $C>0$. This ends the proof.
\end{proof}

Let $D_0$ be a connected component of $D$. Together with Proposition \ref{l:perturbation},
we have the following three consequences. Recall from the beginning of this section that every continuous function on $bD_0$ has a harmonic extension to $D_0$  preserving the continuity on $bD_0$.

\begin{corollary}\label{cor-supp-bD}
	If $bD_0 \setminus \supp(\nu) \neq \varnothing$, then $\nu(b D_0)=0$. 
\end{corollary}

\begin{proof}
	Suppose for contradiction, $\nu(b D_0)>0$. From assumption, there exist $x_0\in b D_0$ and $r_0>0$ such that $\nu( \B(x_0, r_0)\cap b D_0)=0$. We can also assume $\B(x_0, r_0) \cap \sing(D)=\varnothing$ because $\sing(D)$ is finite.
	We take three closed arcs $A_1\subsetneq A_2 \subsetneq A_3$ in $b D_0$ such that $\nu(A_1)>0$ and $A_3$ does not intersect $ \B(x_0, r_0)\cap bD_0, \sign(D)$ or $S$. This can be archived because $\nu$ has no mass on $\sing(D)$ or $S\cap D$ by Lemma \ref{lema-S-no-mass} and $U_\nu$ is upper semicontinuous. 
	Notice that $x_0$ and $A_1$ may not be in the same connected component of $bD_0$. But this is not a problem in our proof.
	
	For every $0<r<r_0/4$, let $f_r$ and $f_A$ be the cut-off functions on $b D_0$ given by  
	\begin{equation}\label{defn-fr}
	f_r(z)=1 \text{ for } z\in \B(x_0,r)\cap  b D_0,\quad   f_r(z)=0 \text{ for } z\in  b D_0\setminus  \B(x_0,2r);
	\end{equation}
	\begin{equation}\label{defn-fA}
	f_A(z)=1 \text{ for } z\in A_1 , \quad f_A(z)=0 \text{ for } z\in b D_0\setminus A_2. 
	\end{equation}  
	Denote by $F_r, F_A$ their harmonic extensions to $D_0$. Consider the function 
	$$H_r:=c( F_A -r^{-3/2} f_r),$$
	where $c>0$ is some small constant whose value will be determined later.
	It is harmonic on $D_0$ and continuous on $b D_0$. We have
	$$\int_{b D_0} H_r  \,\dd\nu \geq  c \, \nu (A_1)  -cr^{-3/2}\nu \big(\B(x_0, 2r)\cap b D_0 \big) =c \, \nu(A_1)>0.  $$
	
	In the following, we will show that $U_\nu +  H_r \leq 0$ on $\overline D_0$. 
	Take a simply connected open subset $D' \subset D_0$ with $\Cc^{1+\ep}$ boundary such that 
	$$\B(x_0, r_0/2) \cap b D_0 \subset b D',  \quad   A_3 \subset bD' \quad \text{and}\quad  b D'\cap \sing(D)=\varnothing.$$
	After that we fix a smooth simply curve $\Lambda \subset D'$ with two end points $\lambda_1, \lambda_2$ in different connected components of $A_3 \setminus A_2$. Denote by  $\Omega$ the open region enclosed by $\Lambda$ and $A_3$. After reducing $\Omega$, we may assume $\overline \Omega$ does not intersect $S$.
	
	Now let $f'_r, f'_A$ be the restriction of $F_r, F_A$ to $b D'$. Since the two end points  $\lambda_1, \lambda_2$ of $\Lambda$ are not contained in $\supp(f'_A)$, $\dist\big (  z, \supp(f'_A) \big) >0$ for all $z\in \Lambda$.  Applying Corollary \ref{cor-poisson} with $f'_r, f'_A, D', \Lambda$, we deduce that for every $z\in \Lambda$ 
	\begin{align*}
	H_r(z)= c F_A(z) -c r^{-3/2} f_r(z) &\leq c\, C \dist(z, b D') -cr^{-3/2} C r\dist(z, b D')\\
	&=c\, C(1-r^{-1/2}) \dist(z, bD'),
	\end{align*}
	which is negative for $r$ small enough.
	
	On the boundary of $D_0\setminus \Omega$, $H_r$ is non-positive. By maximal modulus principle, $H_r \leq 0$ on $\overline {D_0\setminus \Omega}$. On the other hand, since $\overline \Omega\cap S =\varnothing$, we have $U_\nu <0$ on $\overline \Omega$. We can find  $c$ small enough such that $U_\nu + H_r \leq 0$ on $\overline \Omega$. 
	
	\smallskip
	
	Lastly, we extend $H_r$  by $0$ to $\overline D\setminus \overline D_0$. On the set $\sing(D)$, $H_r=0$ from our construction above. So $H_r$ is continuous on $\overline D$. But $\int_{bD} H_r \,\dd \nu>0$ and $H_r +U_\nu \leq 0$ on $\overline D$.
	This  contradicts to Proposition \ref{l:perturbation}.
\end{proof}

\begin{corollary}\label{cor-sing-bD}
	If   there exists some point $ x_0\in bD_0\cap S$ but $x_0 \notin \sing(D)$,
	then $\nu(b D_0)=0$.
\end{corollary}

\begin{proof}
	Suppose for contradiction, $\nu(b D_0)>0$. By the same reason as in Corollary \ref{cor-supp-bD}, we can find $r_0>0$ such that $\B(x_0, r_0)\cap \sing(D)=\varnothing$ and three closed arcs $A_1\subsetneq A_2 \subsetneq A_3$ in $b D_0$ such that $\nu(A_1)>0$ and $A_3$ does not intersect $\B(x_0, r_0)\cap D_0,\sign(D)$ or $S$.
	For every $0<r<r_0/4$, let $f_r$ and $f_A$ be the cut-off functions on $b D_0$ defined as \eqref{defn-fr} and \eqref{defn-fA}.
	Denote by $F_r, F_A$ their harmonic extension to $D_0$. Consider the function 
	$$H_r:= c (F_A -r^{-3/2} f_r), \quad c>0,$$
	which is harmonic on $D_0$ and continuous on $b D_0$. Using Lemma \ref{lem-regularity}, we have for $r$ small enough,
	$$\int_{b D_0} H_r  \,\dd\nu \geq   c \, \nu (A_1)  -cr^{-3/2}\nu \big(\B(x_0, 2r) \big) =c\, \nu(A_1) -c\, O (\sqrt r) >0.  $$
	
	Repeating the same arguments as in  Corollary \ref{cor-supp-bD}, we can show that $U_\nu +H_r\leq 0$ on $\overline D_0$ for $c$ and $r$ small enough. Hence we get a contradiction from Proposition \ref{l:perturbation}.
\end{proof}

\begin{corollary}\label{cor-max-D}
	Let $D_0$ be a connected component of $D$.	If $\overline D_0\cap S=\varnothing$, then $\nu(b D_0)=0$.
\end{corollary}

\begin{proof}
		Suppose for contradiction, $\nu(b D_0)>0$.
	Let $\delta>0$  be a  constant such that $U_\nu +\delta<0$ on $\overline D_0$. Take closed arcs
	$A_1\subsetneq A_2$ in $b D_0$ such that $\nu(A_1)>0$ and $A_2$ does not intersect  $\sing(D)$.
	 Let $f$ be a cut-off function on $b D_0$ such that $f(z) =1$ for $z\in A_1$ and $f(z)=0$ for $z\in b D_0 \setminus A_2$. Then $\int_{b D_0} f \,\dd\nu >0$. Extend $f$ to a function $F$ which is harmonic on $D_0$ and continuous on $b D_0$. After that, we define $H(z):=F(z)$ for $z\in \overline D_0$ and $H(z):=0$ for $\overline D \setminus \overline D_0$. Since $F(z)=0$ for $z\in \sing(D)$, $H$ is continuous. By considering the function $\delta H$, we get a contradiction  from Proposition \ref{l:perturbation}.
\end{proof}

\proof[Proof of Theorem \ref{t:main-1}]
(1) is given by Lemmas  \ref{prop-mass-nu}, \ref{l:min} and Corollary \ref{cor-U-mu-widehat}.  (2) is given by Lemmas \ref{prop-mass-nu} and \ref{prop-nu-S}. (3) is given by  Corollaries  \ref{cor-supp-bD}, \ref{cor-sing-bD}  and \ref{cor-max-D}.
\endproof

We end this section with an example for the necessity of the condition $\nu(b D_0)>0$ in (3) of Theorem \ref{t:main-1}.  Let $D$ be an simply connected open subset of $X$ with $\Cc^{1+\ep}$ boundary and $\overline D\neq X$. In this case the  equilibrium measure $\nu$ of $D$ is $\nu_{bD}+\nu_S$ and $b D\cap S=\varnothing$. Take a small open ball $B$ in $X\setminus (\overline D\cup S)$ such that $b B \cap \supp(\nu_S)\neq \varnothing$.

Consider the new open set $\widetilde D:=D\cup B$ and the new  equilibrium measure $\widetilde \nu$ of $\widetilde D$. We also define $\widetilde S=\{U_{\widetilde\nu}=0 \}$.   Since $X \setminus \widetilde D$ is proper contained in $X\setminus D$, we have $\min  \oI_{\omega , \widetilde D}\geq \min \oI_{\omega,D}$. From $\supp(\nu)\subset X \setminus \widetilde D$, we see that $\oI_{\omega,\widetilde D}$ actually attends its minimum at $\nu$. The uniqueness of $\widetilde \nu$ implies that $\widetilde \nu=\nu$.  In this example, $B$ is a connected component of $b \widetilde D$ such that $\widetilde \nu(B)=0$ and $b B\cap \widetilde S\neq \varnothing$.

\smallskip

This example also shows that condition $x_0\notin \sing(D)$  cannot  be replaced by $x_0\notin \sing(D_0)$ in Corollary \ref{cor-sing-bD}.

\medskip

\section{Spaces of holomorphic sections of line bundles} \label{s:bundle}

In this section, we will recall the notation related to holomorphic sections of line bundles on compact Riemann surfaces.

Let $X$ be a compact Riemann surface of genus $g\geq 1$.  In what follows, let $n$ be a large positive integer  and let $m:=n-g$. For simplicity, we will first assume $\deg (\oL)=1$. In this case, $\oL^n$ is a positive line bundle of degree $n$.

\subsection{Symmetric space}
Let $X^{(n)}$ denote the set of effective divisors of degree $n$ on $X$, i.e.\ the set of unordered $n$-sum $p_1+\cdots+p_n$, where  $p_1,\dots,p_n\in X$, not necessarily distinct.  It is equal to the symmetric product obtained by dividing $X^n$ by the permutations between its factors. Let 
$$\pi_n:X^n\to X^{(n)},\quad   (p_1,\dots,p_n)\mapsto p_1+\cdots+p_n$$
be the canonical projection. Then $X^{(n)}$ inherits a complex structure from $X^n$ in the following way. Let $p_1+\cdots+p_n\in X^n$ and let $z_j$ be a local coordinate in a neighborhood $U_j$ of $p_j$ in $X$, where we take $U_j\cap U_k=\varnothing$ if $p_j\neq p_k$ and $z_j=z_k,U_j=U_k$ if $p_j=p_k$. Let $\sigma_1,\dots,\sigma_n$ be the elementary symmetric functions with respect to the variables $z_j$'s. By the fundamental theorem of algebra, the map
$$ p_1+\dots+p_n\mapsto (\sigma_1,\dots,\sigma_n)$$
gives a local coordinate near $p_1+\cdots+p_n$ on $\pi_n (U_1\times \cdots\times U_n)\subset X^{(n)}$. So $X^{(n)}$ is a complex manifold of dimension $n$ and $\pi_n$ is a ramified covering of degree $n!$. The ramification set of $\pi_n$ is given by $\mathbf R_n:=\big\{(p_1,\dots,p_n):\, p_j=p_k \text{ for some } j\neq k \big\}$. This is also the large diagonal of $X^n$.

\subsection{Jacobian variety}\label{sec-jac}
Let $\alpha_1,\dots,\alpha_g,\beta_1,\dots,\beta_g$ be the $1$-cycles in $X$ forming a canonical basis for $H_1(X,\Z)$, i.e., $\alpha_j$ intersects $\beta_j$ once positively, and does not intersect any others. Let $\{\phi_1,\dots,\phi_g\}$ be a basis for the complex vector space of holomorphic differential $1$-forms on $X$, and we denote by $\phi$ the column vector of length $g$ with entires $\phi_j$'s. The period matrix $\Omega$ is the $g\times 2g$ matrix defined by 
$$\Omega:=\Big(\int_{\alpha_1} \phi, \,\dots\,,  \int_{\alpha_g} \phi  , \, \int_{\beta_1}\phi,\,\dots\, ,\, \int_{\beta_g} \phi\Big).$$
Moreover, one can choose the basis $\phi$ so that $\Omega=(I,\Omega')$, where $I$ is the $g\times g$ identity matrix and $\Omega'$ is a $g\times g$ symmetric matrix having positive definite imaginary part. 

The $2g$ columns of  $\Omega$ are linearly independent over $\R$. Thus, they generate a lattice $\Lambda$ in $\C^g$. We define the \textit{Jacobian variety} of $X$ by
$$ \mathrm{Jac}(X):= \C^g/ \Lambda,$$
which is a $g$-dimensional complex torus. We fix a base point $p_\star\in X$ throughout this article. The Abel-Jacobi map $\oA:X \to \mathrm{Jac}(X)$ associated to $p_\star$ is defined by
$$\oA(x):=\Big(\int_{p_\star}^x \phi_1,\,\dots\,,\,\int_{p_\star}^x \phi_g \Big) \quad  (\text{mod} \,\,\Lambda).$$
The last vector  is independent of the choices of path from $p_\star$ to $x$ after taking the modulo. Of course, $\oA$ depends on the base point $p_\star$.  The  Abel-Jacobi maps associated to  $p_\star$ and $p_\star'$ just differ by a translation on the complex torus. The reader may also refer to \cite{PAG} for the definition of Jacobian variety. 

\smallskip

Denote by $W_1$ the image of $\oA$ in $\mathrm{Jac}(X)$ and $W_t=W_1+\cdots +W_1$ ($t$ times). They are all irreducible complex analytic subsets of $\mathrm{Jac}(X)$ with 
$$W_1\subset W_2\subset\cdots,\quad \dim (W_t)=t \,\text{ for }  1\leq t\leq g, \quad\text{and} \quad W_g=\mathrm{Jac}(X).$$ See  \cite{gun-book} for the description of these $W_t$'s.
We can also define $\oA_t: X^{(t)}\to \mathrm{Jac}(X)$ by 
$$\oA_t(p_1+\cdots+p_t):=\sum_{j=1}^t\oA(p_j)= \sum_{j=1}^t\Big(\int_{p_\star}^{p_j} \phi_1,\,\dots \,,\,\int_{p_\star}^{p_j} \phi_g \Big) \quad  (\text{mod} \,\,\Lambda).$$ 
We have $\oA_1=\oA$ is injective and $\oA_g$ is surjective by Jacobi inversion theorem.

\subsection{Divisors on compact Riemann surfaces}\label{sec-divisor}

Recall that on  the  compact Riemann surface $X$, a \textit{divisor} is a finite formal linear combination $\oD:=\sum a_j x_j$ with $a_j\in \Z$ and $x_j\in X$. If $a_j\geq 0$ for all $j$, then we call $\oD$ an \textit{effective divisor}. The sum $\sum a_j$ is the \textit{degree} of $\oD$. For a meromorphic function $f$ on $X$ and a holomorphic  section $s$ of $\oL$, we write $(f)$ (resp.\ $(s)$) the divisor associated to the zeros and poles of $f$ (resp.\ $s$) respectively. The degree of $(f)$ is always $0$. If a divisor $\oD$ is associated to some meromorphic function, then we call $\oD$  a \textit{principal divisor}. Abel's theorem says that  $\oA_t(p_1+\cdots+p_t)=\oA_t(p_1'+\cdots+p_t')$ if and only if  $p_1+\cdots+p_t-p_1'-\cdots-p_t'$ is a principal divisor. In this case,  we also say that $p_1+\cdots+p_t$ and $p_1'+\cdots+p_t'$ are \textit{linearly equivalent}.

There is a nature group homomorphism from the group of divisors to the group of line bundles, see \cite[page 134]{PAG}.
For a divisor $\mathcal D$ on $X$, we denote by $\Oc(\mathcal D)$ the line bundle corresponding to $\mathcal D$.

\smallskip

The Abel-Jacobi  map $\oA_g$ is injective outside a $g-2$ dimensional analytic subset  $\W$ of $X^{(g)}$, where for  $q_1+\cdots+q_g\in \W$, one has $\dim  H^0(X,\Oc(q_1+\cdots+q_g))\geq 2$. If $g=1$, then $\W$ is empty. The image of $\W$ under $\oA_g$ is called the \textit{Wirtinger subvariety}, denoted by $W_g^1$. It is of codimension $2$ in $\mathrm{Jac}(X)$. Therefore, we obtain a biholomorphic map
$$\oA_g : X^{(g)} \setminus \W \longrightarrow  \mathrm{Jac}(X)  \setminus   W_g^1.$$

As a corollary of Jacobi inversion theorem, every divisor of degree greater than or equal to $g$ on $X$ is linearly equivalent to an effective divisor.
Since we assume that $n$ is large enough,  for any $p_1+\cdots +p_m\in X^{(m)}$, there exists $q_1+\cdots+q_g\in X^{(g)}$ such that 
\begin{equation}\label{equal-p+q=L}
\Oc(p_1+\cdots +p_m)\otimes \Oc(q_1+\cdots+q_g)\simeq \oL^n. 
\end{equation}
Equivalently, if $\oL^n\simeq \Oc(w_1+\cdots+w_n)$, then 
\begin{equation}\label{equal-p+q=L-oA}
\oA_m(p_1+\cdots+p_m)+\oA_g(q_1+\cdots+q_g)=\oA_n(w_1+\cdots+w_n):=\oA_n(\oL^n).
\end{equation}
Here, $\oA_n(\oL^n)$ is well-defined because of Abel's theorem. Moreover,
the choice of $q_1+\cdots+q_g$ is unique if $\dim H^0(X,\Oc(q_1+\cdots+q_g))=1$, or equivalently,  $\oA_n(\oL^n)-\oA_m(p_1+\cdots+p_m)  \notin W_g^1$. For more details on the line bundles over compact Riemann surfaces and Abel-Jacobi map, the readers may refer to \cite{demailly:agbook,PAG,gun-book}.

\smallskip

Inspiring from the above discussion, for every $m$, we define the exceptional set
$$\bH_m:=\big \{ p_1+\cdots+p_m\in X^{(m)}:\, \oA_n(\oL^n)-\oA_m(p_1+\cdots+p_m)    \in W_g^1 \big\}.$$
It is an  analytic subset of $X^{(m)}$ with codimension $2$. One should distinguish the difference between $\bH_g$ and $\W$.
We obtain a holomorphic map 
$$\oB_m:X^{(m)}\setminus \bH_m \longrightarrow  X^{(g)} \setminus \W$$
given by  $\oB_m(p_1+\cdots+p_m):=q_1+\cdots+q_g$ such that \eqref{equal-p+q=L} or \eqref{equal-p+q=L-oA} holds. 
Consequently, this induces the holomorphic map
$$\Ac _m:X^{(m)}\setminus \bH_m  \longrightarrow   \P H^0(X,\oL^n),$$
defined by 
$$\Ac _m (p_1+\cdots+p_m):=[s],$$
with $s\in H^0(X,\oL^n)\setminus\{0\}$ such that $(s)= p_1+\cdots+p_m+\oB_m( p_1+\cdots+p_m)$.

\smallskip

The map $\Ac _m$ can be also read in another way.
Consider the complex vector bundle over $X^{(g)}$ given by
$$\big\{ (s,q_1+\cdots+q_g):\, s\in  H^{0} (X,\oL^n \otimes \Oc(-q_1-\cdots-q_g) ),\, q_1+\cdots+q_g\in X^{(g)} \big\},$$
and its projectivization,  
\begin{align*}
\widehat X^{(m)}:&=\big\{ ([s],q_1+\cdots+q_g):\, s\in  H^{0} (X,\oL^n \otimes \Oc(-q_1-\cdots-q_g) )\setminus\{0\} \big\}        \\
&\simeq \big\{ (p_1+\cdots+p_m, q_1+\cdots+q_g) :\,  \Oc(p_1+\cdots +p_m)\otimes \Oc(q_1+\cdots+q_g)\simeq \oL^n \big\}.     
\end{align*}
Observe that $\widehat X^{(m)}$ is an $m$-dimensional compact complex manifold by Riemann-Roch.
We also have a holomorphic surjective map
$$\widehat\Ac_m:\widehat X^{(m)}  \longrightarrow   \P H^0(X,\oL^n)$$
given by 
$$\widehat\Ac_m \big((p_1+\cdots+p_m, q_1+\cdots+q_g)\big)=[s]$$
with $s\in H^0(X,\oL^n)\setminus\{0\}$ such that $(s)= p_1+\cdots+p_m+q_1+\cdots+q_g$. 

Zelditch \cite[Proposition 3]{zel-imrn} proved that  $\widehat\Ac_m$ is a ramified holomorphic covering of degree $\binom {n}{m}$. Comparing with $\Ac _m$, the advantage of  $\widehat\Ac_m$ is that it is defined on the whole $\widehat X^{(m)}$.  Moreover,  their difference is ``negligible" by the next lemma.

\begin{lemma}\label{lem-analytic-0}
	For any proper analytic subset $A$ of $\widehat X^{(m)}$,  the measure of $\widehat\Ac_m(A)$ in  $\P H^0(X,\oL^n)$ with respect to $V^{\FS}_n$ is $0$. 
\end{lemma}

\begin{proof}
	Obviously, $\widehat\Ac_m$ is a proper map. Hence $\widehat\Ac_m(A)$ is also a proper analytic subset of $\P H^0(X,\oL^n)$, which has measure $0$ with respect to  the  measure $V^{\FS}_n$.
\end{proof}

Define the nature projection $\widehat\pi_m: \widehat X^{(m)}\to X^{(m)}$ by 
$$ \widehat\pi_m (p_1+\cdots+p_m, q_1+\cdots+q_g)= p_1+\cdots+p_m.$$
Then the following diagram commutes:
\begin{equation}\label{diagram-ob}
\begin{tikzcd}
\widehat X^{(m)} \setminus \widehat\pi_m^{-1}(\bH_m) \arrow[r, "\widehat\Ac_m"] \arrow[d, "\widehat\pi_m"]
& \P H^0(X,\oL^n) \arrow[d, "\id" ] \\
X^{(m)} \setminus \bH_m \arrow[r, "\Ac _m"]
& || \P H^0(X,\oL^n).
\end{tikzcd}
\end{equation}

Our aim is to find the formula for the pull-back of the volume form $V_n^{\FS}$ from $\P H^0(X,\oL^n)$ by the map $\Ac _m$. Before that, we need to introduce some more terminologies.

\subsection{Admissible metric}
For any point $x\in X$, we have $H^0(X,\Oc(x))=1$ and each holomorphic section of $\Oc(x)$ vanishes at $x$ because $g\geq 1$. Recall that every holomorphic section of $\Oc(x)$ corresponds to a meromorphic function  $f$ such that $(f)+x$ is an effective divisor. Denote by $\mathbf 1_{\Oc(x)}$ the holomorphic section corresponding to the  function $\mathbf 1$. 

\smallskip

Given the K\"ahler form $\omega$ of integral $1$, for any line bundle $\oL'$  of degree $1$ on $X$, there exists a Hermitian metric $\fh'_\omega$ on $\oL'$ unique up to a multiplicative constant, such that its curvature $(1,1)$-form is exactly $\omega$.  Indeed, let $\fh'$ be any hermitian metric on $\oL'$ with curvature $(1,1)$-form $\omega'$. Then $\int_{X} (\omega-\omega')=0$, i.e., they are in the same cohomology class. So there exists a smooth function $\varphi_0$ such that $\ddc \varphi_0=\omega-\omega'$. It is not hard to check that the metric $\fh'_\omega=e^{-\varphi_0}\fh'$ satisfies the requirement. If we normalize $\varphi_0$ so that $\int_{X} \varphi_0 \,\omega'=0$, then the choice of $\fh'_\omega$ is unique. In this case, we say $\fh'_\omega$ is \textit{$\omega$-admissible}.  For every point $x\in X$, we denote by $\fh_{\omega,x}$ the $\omega$-admissible hermitian metric of $\Oc(x)$. 
Define the function
$$\rho(x):= \int_X  \log \norm{\mathbf 1_{\Oc(x)}(z)}_{\fh_{\omega,x}}  \,\omega(z).$$
Observe that $\norm{\mathbf 1_{\Oc(x)}(z)}_{\fh_{\omega,x}} \simeq \dist(x,z)$. Thus, $\rho$  is continuous on $X$.

\subsection{Riemann theta divisor}
Recall that in Subsection \ref{sec-jac}, we choose the basis $\phi$ so that $\Omega=(I,\Omega')$. Put
$$e_j:=\Big(\int_{\alpha_j} \phi_1, \, \dots \,, \,  \int_{\alpha_j} \phi_g  \Big) ,\quad \Omega_j':= \Big(\int_{\beta_j} \phi_1,\, \dots \, , \,  \int_{\beta_j} \phi_g  \Big). $$
They are the $2g$ columns of $\Omega$, which form an $\R$-basis for $\C^g$.
The complex torus $\mathrm{Jac}(X)$ admits a line bundle $\oL_J$ with the curvature form 
$$\omega_J=\sum_{j=1}^g  \dd x_j \wedge \dd x_{j+g},$$
where $x_1,x_2,\dots,x_{2g}$ are the real coordinates  with respect to the $\R$-basis $\{e_1,\dots,e_g,\Omega_1',\dots,\Omega_g'\}$. Moreover, we have  $\dim H^0(X,\oL_J)=1$. See e.g.\ \cite[Page 333]{PAG}. Let $\fh_J$ be the $\omega_J$-admissible hermitian metric of $\oL_J$.

The line bundle $\oL_J$ has a global holomorphic section $\widetilde \theta$ represented by the \textit{Riemann theta function} $\theta$, which is holomorphic on $\C^g$ and satisfying
$$\theta(z+e_j)=\theta(z) \quad \text{and}\quad \theta(z+\Omega'_j)=e^{-2\pi i(z_j+\Omega'_{j,j}/2)}\theta(z)$$
for every $z\in \C^g$ and $1\leq j\leq g$. Here, $\Omega_{j,j}'$ is  the $(j,j)$-entry of the symmetric matrix $\Omega'$.

Let $\Theta:=(\widetilde \theta)$ be the divisor of $\widetilde\theta$, called the \textit{Riemann theta divisor}.
Riemann's  theorem (c.f.\ \cite[Page 338]{PAG}) states that there exists a unique point $\mathbf z_\star\in \mathrm{Jac}(X)$ such that $W_{g-1}=\Theta +\mathbf z_\star$. This $\mathbf z_\star$ depends on the choice of $p_\star$.

\subsection{$\oE_m$ and $\oF_m$}

For every point $\mathbf p:=p_1+\cdots+p_n\in X^{(n)},n\in \N$, we denote by $\delta_{\mathbf p}$ the empirical probability measure $(\delta_{p_1}+\cdots+\delta_{p_n})/n$. Here we omit the dependence on $n$ to ease the notation.

Now for $\mathbf p=p_1+\cdots+p_m\in X^{(m)}\setminus \bH_m$ and $\mathbf q=\oB_m(\mathbf p)=q_1+\cdots+q_g\in X^{(g)}$, we define
$$\oE_m(\mathbf p):={1\over m^2} \sum_{j\neq k} G(p_j,p_k),$$
and 
$$ \oF_m(\mathbf p):=\log\big\|  e^{U'_{\delta_{\mathbf p}}+gU'_{\delta_{\mathbf q}}/m}       \big\|_{L^{2m}(\omega_0)}.$$
Here, $G$ is the Green function and $U'_{\delta_{\mathbf p}},U'_{\delta_{\mathbf q}}$ are  defined in  \eqref{defn-potentil-type-I}.  The next lemma allows us to remove the $\omega_0$ in $\oF_m$.

\begin{lemma}\label{lem-f-m-p}
	We have 
	$$\sup_{\mathbf p \in X^{(m)}}\big| \oF_m(\mathbf p)- \max U'_{\delta_{\mathbf p}}\big|=O\Big({\log m\over m} \Big)  \quad \text{as}\quad m\to \infty$$
	and $U'_{\delta_{\mathbf p}} \leq C$ for some constant $C$ independent of $m$ and $\mathbf p$.
\end{lemma}

\begin{proof}
By  Lemma \ref{l:Green}, we have
	$$	U'_{\delta_{\mathbf p}}(x)=\int_X G(x,y) \,\dd \delta_{\mathbf p}(y) = {1\over m}\sum_{j=1}^m \Big( \log \dist(x,p_j)+\varrho(x,p_j) \Big).
	$$
It is bounded by $\log \diam(X)+\max \varrho$. This gives the second assertion. In particular,   $U'_{\delta_{\mathbf q}}$ is also uniformly bounded from above. Thus, to prove the first estimate, it is enough to show
\begin{equation}\label{U-2m-cm}
\omega_0 \big\{ U_{\delta_{\mathbf p}} + gU'_{\delta_{\mathbf q}}/m \geq  -2/m  \big\}  \geq c' /m^2
\end{equation} 
for all $m$ large enough  and some $c'>0$ independent of $m$ and $\mathbf p$. Here $U_{\delta_{\mathbf p}}$ is the $\omega$-potential of $\delta_{\mathbf p}$ of type M.  

It is not hard to see that $\{U'_{\delta_{\mathbf q}}:\, \mathbf q\in X^{(g)}\}$ is contained in a compact subset of quasi-subharmonic functions on $X$. There exists an $\alpha>0$ independent of $\mathbf q$ such that for every $M>0$,
$\omega_0 \{U'_{\delta_{\mathbf q}}\leq -M  \} \leq e^{-\alpha M}$ (c.f.\ \cite[Theorem A.22]{dinh-sibony:cime}). 
Replacing $M$ by $4\alpha^{-1} \log m$, we get 
\begin{equation}\label{U-q-small}
\omega_0 \{U'_{\delta_{\mathbf q}}\leq -4\alpha^{-1}  \log m  \} \leq m^{-4}.    
\end{equation}

\smallskip
\noindent {\bf  Claim.} $\omega_0\{ U_{\delta_{\mathbf p}} \geq  -1/m  \} \geq 2c'/m^2$ for some $c'>0$ independent of $m$ and $\mathbf p$.
\smallskip

\noindent
Assuming the claim, together with \eqref{U-q-small}, we verify \eqref{U-2m-cm} and finish the proof of the lemma.  

\smallskip

Let us prove the claim. Let $y$ be a maximum of $U_{\delta_{\mathbf p}}$, i.e.\ $U_{\delta_{\mathbf p}}(y)=0$. It is enough to show the assertion holds on $\B(y,1/m)$.
The problem is local. We can work on $\D$. Fix a K\"ahler form $\omega_\D$ on $\D$ and  let $U$ be a $\omega_\D$-subharmonic function on $\D$ such that $ U\leq 0$ and $U(0)=0$. 
 Fix also a smooth potential $\varphi$ of $\omega_\D$ and we may assume $\varphi(z)=|z|^2+O(|z|^3)$ as $|z|\to 0$. Applying sub-mean inequality to the subharmonic function $U+\varphi$ on $\D(0, 1/m)$,  we get 
 $$0\leq \int_{\D(0, 1/m)} \big(U (z)+\varphi(z)\big)\, i\dd z \wedge \dd \overline z  \leq \int_{|z|<1/m,U(z)<-1/m} -{1\over m} \, i\dd z \wedge \dd \overline z+ {\pi\over m^4}+O(m^{-5}).         $$
 We see that $\Area\big( \{U(z)<-1/m\}\cap\D(0,1/m)\big)  \lesssim m^{-3}$. This proves the claim since $\Area(\D(0,1/m))=\pi m^{-2}$.
\end{proof}

\subsection{Formula of $\Ac _m^*(V_n^{\FS})$}\label{forumla-Ac}

By Riemann-Roch, $\dim\P H^0(X,\oL^n)=n-g=m$. 
Recall that $\mathbf R_m$ is the ramification set of the projection $\pi_m: X^m\to X^{(m)}$. For $p_1+\cdots+p_m\in  X^{(m)}$ outside $\pi_m(\mathbf R_m)$, we can take the local coordinate near $p_1+\cdots+p_m$ as $(z_1,\dots,z_m)$, where $z_j$ is the local coordinate of $p_j$ on $X$. As a  universal holomorphic covering of $X$, the upper half plane $\mathbb H$ carries a canonical Euclidean coordinate. Thus, in the following, we let $(z_1,\dots,z_m)$ to be the ``uniform" local coordinate  near $(p_1,\cdots,p_m) \in X^{m}\setminus \mathbf R_m$, where each $z_j$ is induced from a fixed fundamental domain in $\mathbb H$.

  Zelditch \cite[Theorem 2]{zel-imrn} found the following explicit formula for  $\Ac _m^*(V_n^{\FS})$ under the coordinate we introduced just now.

\begin{proposition}\label{prop-formula}
	On $X^{(m)}\setminus \big(\bH_m \cup\pi_m(\mathbf R_m)\big)$,	the positive measure $\Ac _m^*(V_n^{\FS})$  has the form
	$$C_n  \exp\Big[m^2\Big(\oE_m(\mathbf p)- {2(m+1)\over m} \oF_m(\mathbf p)  \Big) \Big] \kappa_n,$$
	where, $\kappa_n$ is the positive continuous $(m,m)$-form on $X^{(m)}\setminus \big(\bH_m \cup\pi_m(\mathbf R_m)\big)$: 
	\begin{align*}
	\kappa_n:= e&^{-4\sum_{k=1}^m \rho(p_k)-2(m+1)\sum_{j=1}^g \rho(q_j) }\\
	&{\prod_{k=1}^m\prod_{j=1}^g \norm{\mathbf 1_{\Oc(p_k)}(q_j)}^2_{\fh_{\omega,p_k}} \prod_{j=1}^g  \norm{\mathbf 1_{\Oc(p_\star)}(q_j)}^2_{\fh_{\omega,p_\star}} \over  \big\|\widetilde\theta\big(  \oA_g(\mathbf q)-\oA_1(p_\star) -\mathbf z_\star \big)\big\|_{\fh_J}^2} \, i\dd z_1\wedge \dd\overline z_1 \wedge\cdots \wedge i\dd z_m\wedge \dd\overline z_m,
	\end{align*} 
	and $C_n$ is a normalized positive constant such that $\lim_{n\to \infty} 1/n^2 \log C_n=0$.
\end{proposition}

\begin{remark}\rm
	There is an extra factor $2$ in several places of $\Ac _m^*(V_n^{\FS})$ comparing with the formula in \cite{zel-imrn}. This is due to the different definition of the operator $\ddc$ between ours and \cite{zel-imrn}. The limit of $1/n^2 \log C_n$ was not explicitly stated in \cite{zel-imrn}, which can be easily obtained by letting $K=X$ there.
\end{remark}

Now we focus on the continuous form $\kappa_n$ and put 
\begin{equation}\label{defn-xi}
\xi(\mathbf q):={\prod_{j=1}^g \norm{\mathbf 1_{\Oc(p_\star)}(q_j)}^2_{\fh_{\omega,p_\star}}  \over  \big\|\widetilde\theta\big(  \oA_g(\mathbf q)-\oA_1(p_\star) -\mathbf z_\star \big)\big\|_{\fh_J}^2}.
\end{equation}
It is a function on $X^{(g)}$ independent of $n$. 

\begin{lemma}\label{lem-xi-bound}
	There exist  positive constants $c$ and $C$ such that for all  $\mathbf q\in X^{(g)}\setminus \W$,
	$$c\leq \xi(\mathbf q)\leq C.$$
\end{lemma}

\begin{proof}
	Observe that the zero set of numerator of $\xi$ is $p_\star + X^{(g-1)}$.
	Let us now analyze the zeros of the denominator of $\xi$. The zero set of $\widetilde\theta$ is $W_{g-1}-\mathbf z_\star$ by Riemann's  theorem. So the zero set of denominator of $\xi$ is given by
	$$\big\{\mathbf q:\, \oA_g(\mathbf q)\in W_{g-1}+\oA_1(p_\star)  \big\}.$$
	Equivalently,
	$\oA_g(\mathbf q)=  \oA_{g}( x_1+x_2+\cdots+x_{g-1}+p_\star)$ 
	for some $x_1,\dots,x_{g-1}$. 
	Using that $\oA_g$ is a biholomorphic map outside $\W$, we obtain
	$\mathbf q\in X^{(g-1)}+p_\star$. 
	Therefore, the zero set of  numerator of $\xi$ is exactly the zero set of the denominator for $\mathbf q\notin \W$.
	
	\smallskip
	
	To prove the desired bound, we separate the functions part and metrics part of $\xi$. Locally on $X$, we can  write $\norm{\mathbf 1_{\Oc(p_\star)}(z)}_{\fh_{\omega,p_\star}} =|\psi(z)| e^{-\varphi_1(z)}$ for some $\psi$ holomorphic  and $\varphi_1$ smooth. 
	So the numerator of $\xi$ equals to
	$ \prod_{j=1}^g |\psi(q_j)|^2 e^{-2\varphi_1(q_j)}$. Similarly,  the denominator of $\xi$ can be written as $\big|\theta\big(  \oA_g(\mathbf q)-\oA_1(p_\star) -\mathbf z_\star \big)\big|^2 e^{-2\varphi_2( \oA_g(\mathbf q)-\oA_1(p_\star) -\mathbf z_\star)}$ for some  $\varphi_2$ smooth  locally.
	
	From the above discussion, the  function $\prod_{j=1}^g \psi(q_j)^2/\theta\big(  \oA_g(\mathbf q)-\oA_1(p_\star) -\mathbf z_\star \big)^2$ is holomorphic outside $\W$. Since $\W$ is a codimension $2$ analytic subset, by Hartogs' extension theorem (cf. \cite[Section 11, Theorem 4]{Shabat-book}), this function can be extended to a holomorphic function through $\W$ locally.  Therefore, $ \xi$ can be extended to a smooth function on $X^{(g)}$. In particular, it is bounded from above on $X^{(g)}\setminus  \W$.
	
	For the lower bound of $\xi$, one can repeat the argument above with the function $1/ \xi$. 
\end{proof}

For $\mathbf p \in \pi_m(\mathbf R_m)$, $\oE_m(\mathbf p)=-\infty$. In this case, using Lemma  \ref{lem-f-m-p}, we get 
$$\oE_m(\mathbf p)- {2(m+1)\over m} \oF_m(\mathbf p)=-\infty.$$
Together with  Lemma \ref{lem-xi-bound}, we conclude that $\Ac_m^*(V_n^{\FS})=0$ on $\pi_m(\mathbf R_m)\setminus \bH_m$ from its formula in Proposition \ref{prop-formula}. In other words, $\Ac_m^*(V_n^{\FS})$ can be extended to a positive measure trivially through $\pi_m(\mathbf R_m)\setminus \bH_m$.  For simplicity, we still use $\Ac_m^*(V_n^\FS)$ to denote the extended positive continuous form on $X^{(m)}\setminus \bH_m$.

\medskip

\section{Wasserstein balls and its measures} \label{s:Wasserstein}

Recall the Wasserstein distance  on $\cM (X)$ introduced in Section \ref{s:qpot}. In this article, we mainly consider the probability  measures in $\cM(X\setminus D)$. So we denote by $\B_W^D(\sigma, r)$ the open ball of radius $r$ with center $\sigma$ in  $\cM(X\setminus D)$. 

For $\mathbf p=p_1+\cdots+p_m\in X^{(m)}\setminus \bH_m$,  recall that  $\oB_m(\mathbf p)=q_1+\cdots+q_g=:\mathbf q$ with \eqref{equal-p+q=L} or \eqref{equal-p+q=L-oA} holds, i.e., $\mathbf p+\mathbf q$ is the divisor of some holomorphic section of $\oL^n$. In order to take all the holomorphic sections non-vanishing on $D$, we define the following subset of $X^{(m)}\setminus \bH_m$:
\begin{equation}\label{defn-oR-m}
\oQ_m:=\big\{ \mathbf p\in (X\setminus D)^{(m)}\setminus \bH_m: \, \mathbf q\in (X\setminus D)^{(g)} \big\}.
\end{equation}
Observe that $\oQ_m= \Ac_m^{-1} \big(\P H^0(X,\oL^n)_D\big)$.
For $\zeta>0$, we define the following separated subset of $\oQ_m$:
\begin{align}\label{defn-oS-m}
\oS_m^\zeta:= \big\{\mathbf p\in \oQ_m: \,   \dist(p_j,p_k)\geq &\zeta/ \sqrt m    \,\text{ for all }\, 1\leq j\neq k\leq m,\\
 &\dist(p_j,q_l)\geq 1 /  m  \, \text{ for all } \, 1\leq j\leq m, 1\leq l\leq g \big\}. \nonumber
\end{align}

\begin{proposition}\label{bound-Sm}
	Given a smooth probability measure  $\sigma\in \mathcal M (X\setminus  D)$. We can find a constant $\zeta>0$ depending only on $\sigma$, such that for every $r<1$, 
	$$\Vol \big\{\mathbf p\in \oS_m^\zeta:\,  \delta_{\mathbf p}\in \B_W^D(\sigma,r)  \big\} \geq m^{-5m} \quad \text{for all } \, m>m_0,$$
	where $m_0$ is a positive integer  depending on  $\sigma$ and $r$, and $\Vol$ means the volume in $X^{(m)}$.
\end{proposition}

When $g=1$, $X$ is biholomorphic to $\mathrm{Jac}(X)$. The distance on $X$ is comparable with the distance in the complex torus, where the metric is flat. So it is easy to handle the volume.   But
when $g>1$, we need to control the points near the critical set $\W$.
Recall that $\oA_g$ is a biholomorphic map from $X^{(g)}\setminus \W$ to $\mathrm{Jac}(X)\setminus W_g^1$.  Thus, we have the following lemma.

\begin{lemma}\label{ineq-gradiant}
	 For $Q \in X^{(g)}\setminus \W$,  we have the Jacobian determinant
	$$ |\mathbf J_{\oA_g} (Q)|\geq c\dist(Q,\W)^{\vartheta},$$
	where $c>0,\vartheta>0$ are constants independent of $Q$. 
\end{lemma}

\begin{proof}
	In local coordinates, apply Lojasiewicz inequality (c.f.\ \cite[Theorem 6.4]{semianalytic} ) to the function $|\mathbf J_{\oA_g}|$ and the critical set $\W$.
\end{proof}

We also need  to compare the distances between $X^g$ and $X^{(g)}$ outside the diagonal. 

\begin{lemma}\label{lem-dist-X-g}
Let $\alpha$  be a positive constant.	For $Q_1,Q_2 \in X^{(g)}$ with $\dist \big(Q_1, \pi_g(\mathbf R_g)\big)>\alpha$, $\dist\big(Q_2, \pi_g(\mathbf R_g)\big)>\alpha$ and $\dist(Q_1,Q_2)\leq d$,  we can write 
	$$Q_1=\sum_{j=1}^g x_j \quad\text{and}\quad Q_2=\sum_{j=1}^g y_j  \quad \textit{such that} \quad \dist(x_j,y_j)\leq C_{\alpha} d  \quad \text{in }\, X$$
 for every $1\leq j\leq g$, where $C_{\alpha}>0$ is a constant independent of $Q_1,Q_2,d$.
\end{lemma}

\begin{proof}
	Use Lojasiewicz inequality also.
\end{proof}

The following combinatorial   lemma is independent of interest.

\begin{lemma}\label{lem-group}
Let $\{ v_j :\, 1\leq j\leq  2 g \ell \}$ be a collection of points in $X$ satisfying $\dist(v_{j_1},v_{j_2})\geq \zeta/\sqrt{2 g \ell}$ for some $\zeta>0$ if $j_1 \neq j_2$.   Then we can choose $\ell$ disjoint sets $V_t, 1\leq t \leq \ell$ from it,  containing $g$ elements each,  such that for every $t$, 
$$\dist(v, v')\geq \zeta ' \quad \text{if}\quad v\neq v' \quad \text{in }\, V_t.$$
Here, $\zeta'>0$ is a constant independent of $\ell$.
\end{lemma}

\begin{proof}
Using triangulation, we can work on $[0,1)\times [0,1)$ with the Euclidean distance, where we identify the parallel boundaries. Since a line segment of length $1$ contains at most $\sqrt{2g\ell}/\zeta$ points of $\{v_j\}$, we can find $0< a_1 < \cdots< a_{g+1} < 1$ and  rectangles $A_k:=(a_k,a_{k+1}) \times (0,1)$ for $ 1\leq k\leq g$, such that each $A_k$ contains $2\ell -o(\ell)$ points of $\{v_j\}$, as $\ell\to\infty$. These rectangles needs not cover  $\{v_j\}$.
We make two observations: 
\begin{enumerate}
\item $a_{k+1}-a_k\geq \sqrt 3\zeta^2/(2g)-o(1)$;
\item in each $A_k$, we can take a rectangle $B_k:=(a_k,b_k)\times (0,1)$, containing at least $\ell$ points of $\{v_j\}$ and $a_{k+1}-b_k \geq \sqrt 3 \zeta^2/(4g) -o(1)$.
\end{enumerate}
The first observation is easy, because the best way to put separating points into a rectangle  is using equilateral triangle. Hence a rectangle of length $1$ and width  $\delta$  can contain at most $4g \ell\delta/(\sqrt 3 \zeta^2)$ points in $\{v_j\}$.  For the second observations,  the width of $A_k\setminus B_k$ is $\delta_k:=a_{k+1}-b_k$.  It  can contain at most $4g \ell\delta_k/(\sqrt 3 \zeta^2)$ points of $\{v_j\}$. So the minimal $\delta_k$ we can take is $\sqrt 3 \zeta^2/(4g)-o(1)$.

Finally, for each $1\leq t\leq \ell$, we define  $V_t$ by taking exactly one point from every $B_k, 1\leq k\leq g$. We finish the proof with $\zeta'=\sqrt 3 \zeta^2/(8g)$.
\end{proof}

\begin{definition}\rm
	Let $\ell \in \N, \zeta>0$.
	We say a collection of points $\{ v_j :\, 1\leq j\leq g \ell \}$ is \textit{distributed $\zeta$-sparsely}  if 
	$$\dist(v_{j_1},\overline D)\geq \zeta/\sqrt{g\ell},\quad\dist(v_{j_1},v_{j_2})\geq \zeta/\sqrt{g \ell}$$ 
 for all $1\leq j_1\neq j_2 \leq g \ell$, and 
 $$\dist\big(V_k, \pi_g(\mathbf R_g) \big) \geq \zeta, \quad \dist(V_k, \W ) \geq \zeta/\sqrt{g \ell}$$ for  $1\leq k\leq  \ell$. Here,	$V_k:=v_{(k-1)g+1}+v_{(k-1)g+2}+\dots+v_{kg}$. 
\end{definition}

Since $bD $ is assumed to be piecewise $\Cc^{1+\ep}$, we can approximate $\sigma$ using atoms in $X\setminus \overline D$, such that half of these atoms are distributed $\zeta$-sparsely.

\begin{lemma}\label{lem-conver-sigma}
	Given a  probability measure $\sigma \in \cM (X\setminus D)$ with bounded density, i.e.\ $\sigma\leq c\omega_0$ for some $c>0$. We can find a constant $\zeta>0$ depending only on $\sigma$, and
	 a sequence of probability measures $\sigma_m$ converges to $\sigma$ weakly, with $\sigma_m:={1\over m}\sum_{j=1}^m\delta_{v_j^m}$, such that 
	$$\dist(v_{j_1}^m, \overline D) \geq \zeta/\sqrt m  \quad  \text{and}\quad \dist(v_{j_1}^m, v_{j_2}^m) \geq \zeta/\sqrt m \quad \text{for all}\quad  1\leq j_1\neq j_2 \leq m.$$
	Moreover, $\{v_j^m:\, 1\leq j \leq g\ell \}$ is distributed $\zeta$-sparsely for every $m=2g \ell$.
\end{lemma}

\begin{proof}
In the proof, we omit the sup-index $m$ for simplicity. We only need to consider when $m=2g \ell$.
	Since $\sigma\leq c\omega_0$,  according to the density of $\sigma$,  it is easy to take $\widetilde \sigma_m:= {1\over m}\sum_{j=1}^m\delta_{\widetilde v_j}$ converging to $\sigma$,  such that $\dist(\widetilde v_{j_1},\overline D)\geq \zeta'/\sqrt m$ and $\dist (\widetilde v_{j_1},\widetilde v_{j_2})\geq \zeta'/\sqrt m$ for some $\zeta'$ independent of $m$.   By Lemma \ref{lem-group}, we can assume $\dist \big(\widetilde V_k, \pi_g(\mathbf R_g)\big)\geq \zeta'$ for $k\leq \ell$, where $\widetilde V_k:=\widetilde v_{(k-1)g+1}+\dots+\widetilde v_{kg}$. We set $v_j:=\widetilde v_j$ for $j>g \ell$. In the following, we will modify $\widetilde v_j, j\leq g \ell$ to make $V_k$ not close to $\W$.

	 It is enough to assume $m$ is large enough. Locally, $\W$ is a union of smooth manifolds with dimension less that $g-1$. For any $\mathbf x\in \W$, there exists a sequence of points $\mathbf y_n\to\mathbf  x$   such that 
	$$1/2 \dist (\mathbf y_n,\mathbf x)\leq \dist(\mathbf y_n,\W)\leq  \dist(\mathbf y_n,\mathbf x) \quad\text{for all} \quad n.$$
 Moreover, the estimate is uniform on $\dist(\mathbf y_n,\W)$ by compactness of $\W$.
	  Using this, for every $k \leq \ell$, we can find $V_k\in X^{(g)}$ such that 
	  $$\dist(V_k,\widetilde V_k)\leq c_1/\sqrt m \quad\text{and}\quad \dist( V_k,\W)\geq c_2/\sqrt m,$$
	  where $c_1>c_2>0$ are small constants independent of $m$. 
	  
	   We now check $V_k, k\leq \ell$ satisfies all the  conditions. Observe that $\dist\big(V_k, \pi_g(\mathbf R_g)\big)\geq \zeta/2$ since $m$ is large. Applying Lemma \ref{lem-dist-X-g}, we can find a permutation $\{v_j :\, 1\leq j \leq g\ell \}$ such that $V_k =v_{(k-1)g+1}+\dots+ v_{kg}$ and $\dist(v_j, \widetilde v_j )\lesssim c_1/\sqrt{g \ell}$ for all $1\leq j \leq g \ell$. Thus, by choosing $c_1$ very small, we have $\dist(v_{j_1},\overline D)\geq \zeta/ (2\sqrt{g \ell})$ and $\dist( v_{j_1}, v_{j_2})\geq \zeta/(2\sqrt{g \ell})$ for $j_1\neq j_2$.  It is also not hard to see that  $\dist_W\big(\widetilde \sigma_m, {1\over m}\sum_j \delta_{v_j} \big)\lesssim c_1/\sqrt{g \ell}$. Hence all the desired properties holds for $\{v_j\}$ with $\zeta=c_2$. 
\end{proof}

Given a set which is distributed $\zeta$-sparsely, we can cover $\mathrm{Jac}(X)$   only using the points near the set by Abel-Jacobi map. This is the content of next lemma.

\begin{lemma}\label{lem-v-y}
	Let $\ell$ be a large integer and $\mathbf z\in \mathrm{Jac}(X)$. 
	Assume $\{ v_j :\, 1\leq j\leq g \ell \}$ is distributed $\zeta$-sparsely. There exists an $s_0\in \N$ depending only on $\ell$, such that  for every    $s \geq  s_0$ , we can find a new collection of points $\{ y_t^j :\,   1\leq j \leq g \ell, 1\leq t \leq s \}$, such that  
	\begin{equation*}
	\sum_{ 1\leq j \leq g \ell}\sum_{1\leq t \leq s} \oA_1 (y_t^j)=\mathbf z.  
	\end{equation*}
	 Moreover,  for all $1\leq j\leq g \ell$, we have
		\begin{equation*}
	\dist(y_{t_1}^j, v_j)\leq c/\sqrt{g \ell} \quad\text{and}\quad	\dist(y_{t_1}^j,y_{t_2}^j)\geq c/\sqrt {g \ell s }   \quad\text{if}\quad 1\leq t_1\neq t_2\leq s,
			\end{equation*}
	where,   $0<c<\zeta/8$ is a small constant independent of $\mathbf z,\{v_j\},\ell,s$.
\end{lemma}

\begin{proof}
  Fix a small constant $0<c_1<\zeta/8$.
	For every $1\leq j\leq g\ell$, we fix $s$ points  $x_t^j, 1\leq t\leq s$  in the ball $\B(v_j,c_1/\sqrt{g\ell})$ and they satisfy 
	\begin{equation}\label{condition-1}
	\dist(x_{t_1}^j,x_{t_2}^j)\geq c_2/\sqrt { g\ell s }   \quad\text{for}\quad t_1\neq t_2, 
	\end{equation}
	where $c_2>0$ is a small constant only depending on $c_1$. 
	This can be done   because the area of a ball of radius $c_1/\sqrt{g\ell}$ is $O(1/g\ell)$. 
	
	For every $1\leq k\leq \ell$ and $1\leq t\leq s$, we put 
	$$Q_t^k:= x_t^{(k-1)g+1}+x_t^{(k-1)g+2}+\dots+x_t^{kg}  \in (X\setminus \overline D)^{(g)},$$
	$$ \mathbf w:=\sum_{1\leq j\leq g\ell}  \sum_{1\leq t\leq s} \oA_1(x_t^j) =\sum_{1\leq k\leq \ell}  \sum_{1\leq t\leq s} \oA_g(Q_t^k )\in \mathrm{Jac}(X).$$
    
	Define the complex vector
	\begin{equation*}
	\mathbf c_s:= (\mathbf z-\mathbf w )/ (  \ell s) \in \C^g.
	\end{equation*}
	Of course, the definition of $\mathbf c_s$ depends on the choices of the representatives of $\mathbf z$ and $\mathbf w$ in $\C^g$. Here we  take  $\mathbf z$ and $\mathbf w$ with minimal distance on $\C^g$ , so that  $|\mathbf c_s|\leq  (\ell s)^{-1}\cdot\diam(\mathrm{Jac}(X))$. 
	Consider the points $\oA_g(Q_t^k)+\mathbf c_s$ for $1\leq k\leq \ell,1\leq t\leq s$ and set
	$$d_{k,t,s}:=\dist\big(\oA_g(Q_t^k),\oA_g(Q_t^k)+\mathbf c_s\big)\lesssim (\ell s)^{-1},$$ $$ Y_t^k:=\sum_{j=1}^g  y_t^{(k-1)g+j} :=\oA_g^{-1} \big(\oA_g(Q_t^k)+\mathbf c_s\big)  .$$
	Since $\{ v_j \}$ is distributed $\zeta$-sparsely and $\ell$ is large, using that $\dist(x_t^j, v_j)\leq c_1/\sqrt{g\ell}$, we see that 
	$  \dist( Q_t^k,    \W) \geq   c_3/\sqrt{g \ell}$ for some $c_3>0$ only depending on $c_1$ (we reduce the value of $c_1$ if necessary).
	 Thus, Lemma \ref{ineq-gradiant} gives
	\begin{equation}\label{oA-1+c}
	\dist(Q_t^k, Y_t^k) \lesssim  d_{k,t,s}/ {\mathbf J}_{\oA_g}(Q_t^k) \lesssim (\ell s)^{-1} \cdot \ell^{\vartheta/2} \leq s^{-3/4}
	\end{equation}
	provided that  $s\geq s_0\geq \ell^{2\vartheta-4}$. 
	From the definitions of $\mathbf w,\mathbf c_s$ and $Y_t^k$, we receive 
	$$ \sum_{ 1\leq j \leq g \ell}\sum_{1\leq t \leq s} \oA_1 (y_t^j)= \sum_{1\leq k\leq \ell} \sum_{1\leq t\leq s} \oA_g(Y_t^k)=\sum_{1\leq k\leq \ell}  \sum_{1\leq t\leq s} \big(\oA_g(Q_t^k )+\mathbf c_s\big)=\mathbf w+\ell s \mathbf c_s =\mathbf z.$$

On the other hand,   using again that  $\{ v_j \}$ is distributed $\zeta$-sparsely, we see that $\dist\big(Y_t^k,\pi_g(\mathbf R_g)\big)\geq \zeta/2$.  Applying Lemma  \ref{lem-dist-X-g} to \eqref{oA-1+c}, there is a permutation of $y_t^j$, such that 
$\dist(x_t^j,y_t^j)\lesssim s^{-3/4}$
for every $1\leq j\leq g\ell$ and $1\leq t\leq s$.
This implies $\dist(y_t^j, v_j)\leq 2c_1/\sqrt{g\ell}$ and $\dist(y_{t_1}^j,y_{t_2}^j)\geq c_2/(2\sqrt {g\ell s })$  from \eqref{condition-1} provided that $s$ is much large than $\ell$. 
\end{proof}

In order to prove Proposition \ref{bound-Sm}, we find a special $\mathbf p$ in $\oS_m^\zeta$ with $\delta_{\mathbf p}\in \B_W^D(\sigma,r/2)$ first.

\begin{lemma}\label{lemma-exist-one-p}
There exists a $\zeta>0$ depending only on $\sigma$, and an $m_0\in \N$ depending on $\sigma,r$, such that
for every $m\geq m_0$, we can find a $\mathbf p:=p_1+\cdots+p_m\in (X\setminus \overline D)^{(m)} \setminus \bH_m$, satisfying  $\dist_W(\delta_{\mathbf p}, \sigma)<r/2$, and 
$$   \dist(p_{j_1},\overline D)\geq \zeta/\sqrt m  ,   \quad   \dist(p_{j_1},p_{j_2})\geq \zeta/\sqrt m   \quad \text{for all}\quad 1\leq j_1\neq j_2\leq m.$$
Furthermore, write $\mathbf q:=\oB_m(\mathbf p)=q_1+\cdots+q_g$. We have for all $1\leq l\leq g, 1\leq j\leq m$,
$$\dist(\mathbf q, \W),\,\dist\big(\mathbf q,\pi_g(\mathbf R_g)\big),\, \dist(q_l,\overline D)  \geq \beta   \quad \text{and}\quad \dist(q_l, p_j)\geq \zeta/\sqrt m,$$
where $\beta>0$ is independent of $m$.
\end{lemma}

\begin{proof}

Using Lemma \ref{lem-conver-sigma}, we take an $\ell$ large and $\sigma_{2 g \ell}:={1\over 2g \ell} \sum_{j=1}^{2 g \ell } \delta_{v_j}$ such that $\dist_W(\sigma_{2 g \ell },\sigma)<r/4$,
\begin{equation}\label{defn-v-j}
\dist(v_{j_1}, \overline D)\geq \zeta_1/\sqrt {2 g \ell }, \quad \dist(v_{j_1}, v_{j_2}) \geq \zeta_1/\sqrt{2 g \ell } \quad \text{for} \quad 1\leq j_1\neq j_2\leq 2 g \ell  
\end{equation}
and $\{v_j:\, 1\leq j \leq  g \ell \}$ is distributed $\zeta_1$-sparsely for some $\zeta_1>0$ independent of $\ell $.  Take the $s_0$ associated to $\ell$ in  Lemma \ref{lem-v-y} and put $m_0:=  2 g \ell s_0 $.  We also let $s_0\geq \ell^2$ so that $m_0$ is much larger than $2 g \ell $.

Fix a small constant $0<c_1<\zeta_1 /8$. Similar as in \eqref{condition-1},
	for every $g\ell +1\leq j\leq 2g\ell$, we  define  $w_t^j, 1\leq t\leq s$  in the ball $\B(v_j,c_1/\sqrt{g\ell})$ and they satisfy 
	\begin{equation}\label{condition-2}
	\dist(w_{t_1}^j,w_{t_2}^j)\geq c_2/\sqrt {g\ell s }   \quad\text{for}\quad t_1\neq t_2, 
	\end{equation}
where $c_2>0$ is a small constant only depending on $c_1$.

\smallskip

For every $m\geq m_0$, we can uniquely write $m= 2g\ell s  +m_2$, with integers $s\geq s_0$ and $0\leq m_2<2 g\ell$. 
Fix $\{ u_k:\,  1\leq k\leq  m_2\}$ in $X\setminus \overline D$, such that 
\begin{equation} \label{defn-u-k}
\dist(u_k,\overline D)\geq    \zeta_1 /  \sqrt{8 g \ell } \quad \text{and} \quad    \dist( u_k, v_j )\geq \zeta_1 / \sqrt{8 g \ell } 
\end{equation}
for all $1\leq k \leq  m_2,  1\leq j \leq 2 g \ell $.
After that,  we choose a $\bQ:=q_1+\cdots+q_g$ in $(X\setminus \overline D)^{(g)}$ such that
\begin{equation}\label{q-k-qk-y-2}
\dist( q_l,\overline D),\,	\dist(\bQ, \W),\,  \dist\big(\bQ, \pi_g(\mathbf R_g)\big) \geq \beta ,    \quad    \dist( q_l, v_j) , \,    \dist( q_l, u_k)\geq \zeta_1/ \sqrt{8 g \ell }
	\end{equation} 
for all $1\leq l \leq g, 1\leq j\leq 2 g \ell , 1\leq k\leq m_2$. Here, $\bQ$ depends on $\ell $ but  $\beta >0$ not. Since $\overline D, \W ,\pi_g(\mathbf R_g)$ are fixed, the first three conditions are easy to handle. For the last two, note that the total area of $\B(v_j, \zeta_1/\sqrt{8 g \ell })$'s and $\B(u_k, \zeta_1/\sqrt{8 g \ell })$'s is $O(\zeta_1)$ as $\ell \to \infty$. We can reduce the value of $\zeta_1$ to ensure the existence of $\mathbf Q$.  

\smallskip

 For every $s\geq s_0$, we define the point in $\mathrm{Jac}(X)$:
$$\mathbf z_s:=  \oA_{m+g}  \big(\oL^{m+g} \big) -\sum_{j= g\ell_1+1}^{2 g \ell } \sum_{t=1}^s \oA_1 (w_t^j) - \sum_{k= 1}^{m_2} \oA_1 (u_k)  -\oA_g(\bQ).  $$
 Applying Lemma \ref{lem-v-y} for  $\mathbf z=\mathbf z_s$, we get a set $\{ y_t^j :\,   1\leq j \leq g \ell, 1\leq t \leq s \}$ in $(X\setminus \overline D)^{(g \ell s)}$  and a constant $0<c<\zeta_1 /8$ only depending on $\zeta_1$ such that 
\begin{equation}\label{new-y}
	\sum_{ j=1} ^{ g \ell}\sum_{t=1}^s \oA_1 (y_t^j)=\mathbf z_s,\quad  \dist(y_t^j, v_j)\leq c/\sqrt{g \ell} ,\quad \dist(y_{t_1}^j,y_{t_2}^j)\geq c/\sqrt{g \ell s}.
\end{equation}

Lastly, we define 
$$\mathbf p:= \sum_{j=1}^{ g \ell}\sum_{t=1 }^s  y_t^j  +  \sum_{j= g\ell+1}^{2 g \ell } \sum_{t=1 }^s  w_t^j +\sum_{k=1}^{m_2} u_k  \in (X\setminus \overline D)^{m} .$$
The first equality of \eqref{new-y}  gives $\bQ=\oB_m(\mathbf p)$.

\smallskip

Now we check $\mathbf p$ and $\bQ$ satisfy all the requirements. The point $\mathbf p$ is not in $\bH_m$ because $\mathbf Q \notin \W$. The fact that $m\geq 2 g \ell s$  and \eqref{defn-v-j},\eqref{condition-2}, \eqref{defn-u-k}, \eqref{q-k-qk-y-2}, \eqref{new-y} give all the distance inequalities in the lemma, by taking $\zeta >0$ small independent of $\ell,s$ and $m$.
It remains to show $\dist_W(\delta_{\mathbf p}, \sigma)<r/2$. We have 
$$\dist_W (\delta_{\mathbf p}, \sigma)  \leq \dist_W(\delta_{\mathbf p} ,\sigma_{2g \ell})+\dist_W(\sigma_{2g \ell},\sigma) < \dist_W(\delta_{\mathbf p} ,\sigma_{2g \ell})+r/4.   $$
Using that $\ell^2 \leq s$, $m_2<2g\ell$, $\dist(y_t^j, v_j)<c/\sqrt{g\ell}$ and $\dist(w_t^j,v_j)\leq c_1/\sqrt{g \ell}$, we obtain
$$  \dist_W(\delta_{\mathbf p} ,\sigma_{2g \ell})\lesssim m_2/(2g\ell s) + 1/\sqrt{g\ell} +1/\sqrt{ g\ell}   \lesssim 1/\sqrt \ell.   $$ 
We finish the proof of the lemma by taking $\ell$ large enough.
\end{proof}

We will use the above $\mathbf p$ to construct a subset of $\oS_m^\zeta$ with volume not too small.
We can ignore the exceptional set $\bH_m$ as its volume is $0$ in $X^{(m)}$. Hence we may assume $\oB_m(\mathbf p')$ is always well-defined.

\smallskip

\begin{proof}[Proof of Proposition \ref{bound-Sm}]
We take the $\zeta, m_0$ and $\mathbf p, \mathbf q$ in  Lemma \ref{lemma-exist-one-p}. For every $m\geq m_0$, consider the set 
	\begin{equation*}\label{defn-omega-m}
		\Omega_{m}:=\otimes_{1\leq j\leq  m} \B( p_j,m^{-2}).
		\end{equation*}
  Clearly, $\Vol(\Omega_m)\geq m^{-5m}$. In the following, we will show $\Omega_m\subset \oS_m^{\zeta/2}$ and $\dist(\delta_{\mathbf p'}, \sigma)<r$ for $\mathbf p'\in \Omega_m$, which ends the proof of  Proposition \ref{bound-Sm}.

\smallskip

 Take any $\mathbf p'\in \Omega_m$. Write $\mathbf p'=p'_1+\cdots+p'_m$ with $\dist(p'_j,p_j)<m^{-2}$.  The clearly implies $\dist(\delta_{\mathbf p'},\delta_{\mathbf p})\lesssim m^{-2}$ and hence  $\dist(\delta_{\mathbf p'}, \sigma)<r$ for $m$ large enough.

Now we verify that $\mathbf p'\in\oS_m^{\zeta/2}$. Let $\mathbf q':=\oB_m(\mathbf q')=q'_1+\cdots+q'_g$.  From Lemma \ref{lemma-exist-one-p}, it is not hard to see that $\mathbf p',\mathbf q'\in (X\setminus D)^{(g)}$ and $\dist(p'_j,p'_k)\geq \zeta/(2\sqrt m)$ for $j\neq k$.
 Since $\oA_1$ is injective and holomorphic,
\begin{align*} \dist\big(\oA_m(\mathbf p),\oA_m(\mathbf p')\big) &= \dist\Big( \sum_{j=1}^m \oA_1(p_j), \sum_{j=1}^m \oA_1(p'_j) \Big) \\
&\leq \sum_{j=1}^m \dist\big( \oA_1(p_j),\oA_1(p'_j) \big)   
 \lesssim  \sum_{j=1}^m \dist(p_j,p'_j)\leq m^{-1}.
\end{align*}
 Using the formula \eqref{equal-p+q=L-oA}, we get
 $$\dist\big(\oA_g(\mathbf q), \oA_g(\mathbf q' \big)=\dist\big(\oA_n(\oL^n)-  \oA_m(\mathbf p),\oA_n(\oL^n)- \oA_m(\mathbf p')   \big) \lesssim m^{-1}.$$
Recall from Lemma \ref{lemma-exist-one-p} that $\dist(\mathbf q, \W)\geq \beta$. So by Lemma \ref{ineq-gradiant}, we see that $\dist(\mathbf q, \mathbf q')\lesssim m^{-1}$.
Recall again from Lemma \ref{lemma-exist-one-p} that $\dist\big(\mathbf q,\pi_g(\mathbf R_g)\big)\geq \beta$. Applying   Lemma \ref{lem-dist-X-g}, we can fix a permutation of $q'_l$'s such that $\dist(q_l, q'_l)\lesssim m^{-1}$ for every $1\leq l\leq g$. Using Lemma \ref{lemma-exist-one-p} one more time, we obtain that for every $1\leq j\leq m, 1\leq l \leq g$, 
$$ \dist(p'_j, q'_l)\geq \dist(p_j, q_l) - \dist(q_l,q'_l)-\dist(p_j,p'_j)\geq \zeta/\sqrt m - O(m^{-1})-O(m^{-2})>1/m. $$
Thus, $\Omega_m\subset \oS_m^{\zeta/2}$ and the proof of Proposition \ref{bound-Sm} is now complete.
\end{proof}

\medskip

\section{Regularization of $\oI_{\omega, D}$} \label{sec:regularization}

We need the following lemma to approximate $\oI_{\omega,D}$ by using $\oE_m$ and $\oF_m$.

\begin{lemma}\label{sigma-delta-potential}
	Let $\eta$  be a smooth probability measure on $X$. For any $\mathbf p\in X^{(m)}$ satisfying  $\dist(p_j,p_k)\geq \zeta/ \sqrt m$ for $j\neq k$. We have 
	$$\Big|\oE_m(\mathbf p)-\int_X U'_\eta \,\dd \eta \Big|\leq C_{\eta,\zeta}\sqrt{\dist_W(\eta,\delta_{\mathbf p})}$$
	for some constant $C_{\eta,\zeta}>0$ independent of $m$ and $\mathbf p$.
\end{lemma}

\begin{proof}
	It is enough to assume $\dist_W(\eta,\delta_{\mathbf p})<1/2$, otherwise the inequality holds trivially. 	We fixed a constant $0<a<1/2$ whose value will be determined later. Observe that if $p_j\neq p_k$ for $j\neq k$, then
	$$\oE_m(\mathbf p)=\int_{(X\times X)\setminus \Delta}G(x,y)\,\dd \delta_{\mathbf p}(x)\dd \delta_{\mathbf p}(y).$$
	It follows that
	\begin{align}
	\Big|\oE_m(\mathbf p)-\int_X U'_\eta \,\dd \eta \Big|&=\Big|\int_{(X\times X)\setminus  \Delta}G(x,y)\,\dd \delta_{\mathbf p}(x)\dd \delta_{\mathbf p}(y)-\int_{X\times X}G(x,y)\,\dd \eta(x)\dd\eta(y)	\Big|\nonumber\\
	&\leq \Big| \int_{\dist(x,y)\geq a}G(x,y)\,\dd \delta_{\mathbf p}(x)\dd \delta_{\mathbf p}(y) -  \int_{\dist(x,y)\geq a }G(x,y)\,\dd \delta_{\mathbf p}(x)\dd\eta(y)  \Big| \label{term-1}\\
	&\quad+\Big|\int_{\dist(x,y)\geq a}G(x,y)\,\dd \delta_{\mathbf p}(x)\dd\eta(y)- \int_{\dist(x,y)\geq a}G(x,y)\,\dd \eta(x)\dd\eta(y) \Big|\label{term-2}\\
	&\quad+\Big| \int_{0<\dist(x,y)<a}G(x,y)\,\dd \delta_{\mathbf p}(x)\dd\delta_{\mathbf p}(y)  \Big|	 \label{term-3}\\
	&\quad+\Big| \int_{\dist(x,y)<a}G(x,y)\,\dd \eta(x)\dd\eta(y)  \Big|.	 \label{term-4}
	\end{align}

	Let $\|\cdot\|_{\Lip}$  denote the Lipschitz constant semi-norm.	By definition of Wasserstein distance,	\eqref{term-1} is equal to
	\begin{align*}
	\Big|\int_{\dist(x,y)\geq a}G(x,y)\,\dd \delta_{\mathbf p}(x)\dd (\delta_{\mathbf p}-\eta)(y)\Big|\leq \dist_W(\eta,\delta_{\mathbf p}) \Big\|\int_{\dist(x,y)\geq a}G(x,y)\,\dd \delta_{\mathbf p}(x)\Big \|_{\Lip(y)} .
	\end{align*}
	Using Lemma \ref{l:Green} that $G(x,y)=\log\dist(x,y)+\varrho(x,y)$, it is not hard to see  
	the second factor is	$\lesssim  a^{-1}$, and thus, \eqref{term-1} is $\lesssim a^{-1} \dist_W(\eta,\delta_{\mathbf p})$.
	Similarly, \eqref{term-2} is bounded by
	$$\dist_W(\eta,\delta_{\mathbf p})  \Big\|\int_{\dist(x,y)\geq a }G(x,y)\,\dd \eta(y)\Big \|_{\Lip(x)}\lesssim a^{-1} \dist_W(\eta,\delta_{\mathbf p}). $$
	
	Now we estimate \eqref{term-3}. Since  $p_j$ and $p_k$ are separated by $\zeta/\sqrt m$ if $j\neq k$,   for every fixed $p_j$ and $t\in\N$, then number of $p_k$'s with $$t/\sqrt m<\dist(p_j,p_k)\leq (t+1)/\sqrt m$$ is bounded by $C_\zeta t$ for some constant $C_\zeta >0$ independent of $m$. Therefore, \eqref{term-3} is equal to
	$$\Big|{1\over m^2}\sum_{0<\dist(p_j,p_k)<a} G(p_j,p_k)\Big|\leq  {1\over m^2}\sum_{j=1}^m  \sum_{t=1}^{\lfloor a\sqrt m \rfloor+1} C_\zeta t \Big( \big|\log(t/\sqrt m)\big| +\max | \varrho|\Big). $$
	Recall the Stirling's formula
	$\sum_{t=1}^n t\log t={n^2\over 2}\log n-{n^2\over 4}+{n\over 2}\log n +{1\over 12}\log n +O(1).$
	So by a direct computation, \eqref{term-3} is $\lesssim C_\zeta a^2 |\log a| \leq C_\zeta a$.
	
	For \eqref{term-4}, since $\eta$ is smooth, $G(x,y)=\log\dist(x,y)+\varrho(x,y)$ is integrable near the diagonal, and \eqref{term-4} is bounded by
	\begin{align*}
	&\int_{y\in X}\int_{\dist(x,y)<a} \Big( \big|\log\dist(x,y)\big| +O(1) \Big) \,\dd\eta(x)\dd\eta(y)\\
	&\leq  \int_{y\in X} \Big(C_\eta 	\int_{|z|<a} \big|\log |z| \big|\, i\dd z\wedge \dd\overline z+O(a^2)\Big)	\dd\eta(y)\\
	&\lesssim 	C_\eta (a^2|\log a|+a^2)\lesssim C_\eta a.
	\end{align*}

	Finally, we take $a:=\sqrt{\dist_W(\eta,\delta_{\mathbf p})}$ and finish the proof of the lemma.
\end{proof}

To achieve the smoothness condition of $\eta$ in last lemma, we need to regularize  $\oI_{\omega,D}$.

\begin{lemma}\label{lem-separate-measure}
Let $\sigma_j,\eta_j$ be two sequences  in $\cM(X\setminus D)$ converging weakly to $\sigma,\eta$, such that $\lim_{j\to \infty}\oI_{\omega, D}(\sigma_j)=\oI_{\omega, D}(\sigma),\lim_{j\to \infty}\oI_{\omega, D}(\eta_j)=\oI_{\omega, D}(\eta)$ and $\oI_{\omega, D}(\sigma),\oI_{\omega, D}(\eta)$ both are finite. Then for each $0\leq t \leq 1$, $$\lim_{j\to \infty} \oI_{\omega, D} \big( t\sigma_j +(1-t)\eta_j \big)= \oI_{\omega, D} \big( t\sigma +(1-t)\eta \big).$$
\end{lemma}

\begin{proof}
Put $\mu_j^t:= t\sigma_j +(1-t)\eta_j$ and $\mu^t := t\sigma+(1-t)\eta$.
 Since $\mu\mapsto \max U'_\mu$ is continuous by Lemma \ref{l:convex}, using the formula \eqref{second-formula-I}, we only need to show $\lim_{j\to \infty}\int_X U'_{\mu_j^t} \,\dd \mu_j^t=\int_X U'_{\mu^t} \,\dd \mu^t$.
Observe that  $U'_{\mu_j^t} = tU'_{\sigma_j} +(1-t) U'_{\eta_j}$ and $U'_{\mu^t} = tU'_{\sigma} +(1-t) U'_{\eta}$.  By Lemma \ref{commut-potential}, we have
	 \begin{align*}
	\int_X U'_{\mu_j^t} \,\dd\mu_j^t  &=-\int_X \big(tU'_{\sigma_j} +(1-t) U'_{\eta_j}\big) \,\dd\big(t\sigma_j +(1-t)\eta_j\big)\\
	&=  -t^2\int_X U'_{\sigma_j} \,\dd\sigma_j-2(1-t)t\int_X U'_{\sigma_j}\,\dd\eta_j -(1-t)^2\int_X U'_{\eta_j}\,\dd\eta_j .
	\end{align*}
By  the assumptions $\lim_{j\to \infty}\oI_{\omega, D}(\sigma_j)=\oI_{\omega, D}(\sigma),\lim_{j\to \infty}\oI_{\omega, D}(\eta_j)=\oI_{\omega, D}(\eta)$, we have $$ \lim_{j\to \infty} \int_X U'_{\sigma_j} \,\dd\sigma_j=\int_X U'_{\sigma} \,\dd\sigma,\quad \lim_{j\to \infty}U'_{\eta_j} \,\dd\eta_j=\int_X U'_{\eta} \,\dd\eta.$$

It remains to deal with the term $\int_X U'_{\sigma_j}\,\dd\eta_j$.  By Stoke's formula,
\begin{align}
\Big|\int_X U'_{\sigma_j}\,\dd\eta_j  -\int_X U'_{\sigma}\,\dd\eta \Big| &\leq \Big| \int_X U'_{\sigma_j}\,\dd\eta_j  -\int_X U'_{\sigma_j}\,\dd\eta   \Big|+\Big| \int_X U'_{\sigma_j}\,\dd\eta  -\int_X U'_{\sigma}\,\dd\eta   \Big|   \nonumber\\
&=\Big|   \int_X U'_{\sigma_j}\,\ddc ( U'_{\eta_j}-U'_\eta)\Big|+\Big|   \int_X (U'_{\sigma_j}-U'_\sigma)\,\ddc U'_\eta\Big|   \nonumber\\
&=\Big|   \int_X  \dd U'_{\sigma_j}\wedge\dc ( U'_{\eta_j}-U'_\eta)\Big|+\Big|   \int_X \dd (U'_{\sigma_j}-U'_\sigma)\wedge\dc U'_\eta\Big|. \label{sigma-eta}
\end{align}
 Using Cauchy-Schwarz inequality, the first term of \eqref{sigma-eta} is bounded by
 \begin{align*} \Big| \int_X  \dd U'_{\sigma_j} \wedge \dc U'_{\sigma_j}  \Big|   ^{1/2}   &\Big|  \int_X \dd ( U'_{\eta_j}-U'_\eta) \wedge \dc( U'_{\eta_j}-U'_\eta)\Big|^{1/2} \\
 &=\Big| \int_X   U'_{\sigma_j} \,\dd \sigma_j  \Big|    ^{1/2}   \Big|  \int_X ( U'_{\eta_j}-U'_\eta) \, \dd( \eta_j-\eta)\Big|^{1/2},
 \end{align*}
 which tends to $0$ because $\lim_{j\to \infty}\oI_{\omega, D}(\sigma_j)=\oI_{\omega, D}(\sigma)<\infty,\lim_{j\to \infty}\oI_{\omega, D}(\eta_j)=\oI_{\omega, D}(\eta)<\infty$. Similarly, the second term of \eqref{sigma-eta} also tends to $0$. This ends the proof.
\end{proof}

\begin{proposition} \label{l:reg-I}
	For any probability measure $\sigma\in \cM(X\setminus D)$ with $\oI_{\omega,D}(\sigma)$ finite. 
	There exists a sequence of smooth probability measures $\sigma_j\in \cM (X\setminus D)$ converging weakly to $\sigma$ such that 
	$$\lim_{j\to\infty}\oI_{\omega,D}(\sigma_j) = \oI_{\omega,D}(\sigma).$$
\end{proposition}

\begin{proof}
	\noindent \textbf{Case 1:} $\supp(\sigma)\subset X\setminus \overline D$. Using partition of unity and convolution, we can find a sequence of smooth probability measures $\sigma_j$ converging weakly to $\sigma$ with $\supp(\sigma_j)\subset X\setminus \overline D$. Moreover, since $G(x,y)=\log\dist(x,y)+\varrho(x,y)$, the assertion $\lim_{j\to\infty}\oI_{\omega,D}(\sigma_j) = \oI_{\omega,D}(\sigma)$ is guaranteed by Lemma \ref{lem-separate-measure} above and \cite[Lemma 2.2]{Zei-esaim}.

	\smallskip
	
	\noindent \textbf{Case 2:} $\supp(\sigma)\cap  b D \neq \varnothing$. By Lemma \ref{lem-separate-measure} and using \textbf{Case 1}, we can assume that $\sigma$ is supported on a small neighborhood of $ b D $. Recall that $ b D $ is piecewise $\Cc^{1+\ep}$. If $\sing(D)\cap \supp(\sigma) =\varnothing$, for every $x\in \supp(\sigma)$, we take a small open neighborhood of $x$ and consider the $\delta$-translation  with the inner normal direction of $ b D $ near $x$. By partition of unity, we obtain a new probability measure $\sigma_\delta$ with compact support in $X\setminus \overline D$.  It is not hard to see that $\sigma_\delta\to  \sigma$ and $\oI_{\omega,D}(\sigma_\delta)\to \oI_{\omega,D}(\sigma)$ as $\delta\to 0$.
	
	\smallskip
	
	It only remains to consider the case that $\supp(\sigma)$ contains some  $y\in \sing(D)$. By Lemma \ref{lem-separate-measure} again, we may assume that $y$ is the only singular point of $D$  in $\supp(\sigma)$. Then we consider the decomposition
	$\sigma=\mu_\delta^1+\mu_\delta^2:=\sigma |_{\B(y,\delta)}+ \sigma | _{\B(y,\delta)^c}$. 
	We have
	\begin{align*}
	\int_X U'_{\sigma} \,\dd\sigma =\int_X \big( U'_{\mu_\delta^1} +U'_{\mu_\delta^2} \big)\,\dd(\mu_\delta^1+\mu_\delta^2)=\int_X U'_{\sigma} \,\dd\mu_\delta^1+\int_X U'_{\mu_\delta^1} \,\dd\mu_\delta^2+\int_X U'_{\mu_\delta^2} \,\dd\mu_\delta^2.
	\end{align*}
	The assumption that $\oI_{\omega, D}$ is finite gives $\int_X  U'_{\sigma} \,\dd\sigma>-\infty$ and $\sigma$ does not have point mass. Thus,  by definition of integral, the integral $\int_X U'_{\sigma} \,\dd\mu_\delta^1$ tends to $0$ as $\delta\to 0$. By the same reason, 
	$\lim_{\delta\to 0} \int_X U'_{\mu_\delta^1} \,\dd\mu_\delta^2=\lim_{\delta\to 0} \int_X U'_{\mu_\delta^2} \,\dd\mu_\delta^1 =0$.
	Therefore, we conclude that $$\lim_{\delta\to 0}\int_X U'_{\mu_\delta^2} \,\dd\mu_\delta^2=\int_X U'_{\sigma} \,\dd\sigma.$$
	
	Define $\sigma_\delta:=\mu_\delta^2 /  \sigma( \B(y,\delta)^c)$. It is a probability measure supported outside $\B(y,\delta)$. Moreover, using that $\lim_{\delta\to 0} \sigma( \B(y,\delta)^c)=1$, we have
	$$\lim_{\delta\to 0}\int_X U'_{\sigma_\delta} \,\dd\sigma_\delta =\lim_{\delta\to 0} {1\over  \sigma( \B(y,\delta)^c)^2 }  \int_X U'_{\mu_\delta^2} \,\dd\mu_\delta^2 = \int_X U'_{\sigma} \,\dd\sigma.$$
	On the other hand, recall from Lemma \ref{l:convex} that $\mu \mapsto \max U'_\mu$ is continuous.
	Therefore, $\lim_{\delta\to 0} \max U'_{\sigma_\delta} =\max U'_\sigma$ and thus,   $\lim_{\delta\to 0}\oI_{\omega,D}(\sigma_\delta)=\oI_{\omega,D}(\sigma)$ by using the formula \eqref{second-formula-I}.
	Since $\sigma_\delta$ is a probability measure with support that does not  intersect the singular point of $ b D $, we can use the earlier cases to finish the proof.
\end{proof}

%%%%%%%%%%%%%%%%%%%%%%%

\medskip

\section{Large deviations and proof of Theorem \ref{t:main-2}} \label{s:proof-main-2}

 In this section, we will prove Theorem \ref{t:main-2}, by using a large deviation principle relating to the random section non-vanishing on $D$.
Let $m=n-g$ and we first assume that $\deg (\oL)=1$ and $g\geq 1$ for simplicity at this moment.

\medskip

For $s$  a holomorphic section of $\oL^n$ with $(s)=\mathbf p +\mathbf q$, the empirical measure $\llbracket Z_s \rrbracket$ is the probability measure $\delta_{\mathbf p+\mathbf q}$. So the Fubini-Study volume form  $V^{\FS}_n$ on  $\P H^0(X,\oL^n)$ induces a probability measure, denoted by $\bP_n$, on the (Polish) space $\mathcal M (X)$. For a measurable set $A\subset \mathcal M (X)$, if we set
\begin{equation}\label{defn-Ap}
A_\P:=  \big\{[s]\in\P H^0(X,\oL^n):\, \llbracket Z_s \rrbracket \in A      \big\},
\end{equation}
then by definition,
$$\bP_n (A)=\int_{A_\P} \dd V^{\FS}_n.    $$
Similarly, the probability measure $V^{\FS}_{n,D}$ on $\P H^0(X,\oL^n)_D$ also induces a probability measure, denoted by $\bP_{n,D}$, on $\mathcal M (X\setminus D)$.
We have the following large deviation principle for $\bP_{n,D}$.

\begin{theorem}\label{thm-large-devia}
	For any $\sigma\in \mathcal M (X\setminus D)$, we have 
	$$\lim_{r\to 0} \lim_{n\to \infty} {1\over n^2} \log \bP_{n,D} \big(\B_W^D(\sigma,r)\big)=      \min \oI_{\omega,D} -\oI_{\omega,D}(\sigma).$$ 
\end{theorem}

When $D$ is empty, the above result is a special case in  \cite{zel-imrn}. Instead of working directly on $\P H^0(X,\oL^n)$,  Zelditch considered the space $X^{(m)}$ there. As in Subsection \ref{forumla-Ac}, he obtained an explicit formula for $\Ac_m^*(V_n^{\FS})$ outside the analytic set $\bH_m$. Moreover, by Lemma \ref{lem-analytic-0} and the commutative digram \eqref{diagram-ob}, one  only needs to consider the integral of $\Ac_m^*(V_n^{\FS})$ outside the analytic set. More precisely, for a measurable subset $A$ of  $\mathcal M(X)$,
$$\bP_n(A)=\int_{A_\P}\dd V_n^{\FS} =\binom {n}{m}^{-1} \int_{\Ac_m^{-1}(A_\P)}  \Ac_m^*(V_n^{\FS})  $$
since  $\widehat\Ac_m$ is a ramified holomorphic covering of degree $\binom {n}{m}$.
So the Fubini-Study measure of any open ball in $\P H^0(X,\oL^n)$ can be computed in the complex manifold  $X^{(m)}$ using $\Ac_m^*(V_n^{\FS})$. But the problem is much more complicated for the probability measure  $\bP_{n,D}$.

\smallskip

Recall that $X$ is a compact Riemann surface of positive genus.
From our discussion in Subsection \ref{sec-divisor}, for every holomorphic section $s\in \P H^0(X,\oL^n)$, the zero set of $s$ consists of $n$ points $p_1,\dots,p_m,q_1,\dots,q_g$ counting with multiplicity. But for generic $s\in \P H^0(X,\oL^n)$,  $m$ points among the zeros of $s$ already determine the section $s$, which is exactly given by the map $\Ac_m$. Even though $p_1,\cdots, p_m$ are all belonging to $X\setminus D$, $q_k$ may be contained in $D$ for some $k$. Thus, usually we have
$$\int_{(X\setminus D)^{(m)}\setminus \bH_m} \Ac_m^*(V_n^{\FS}) >\binom {n}{m} \int_{\P H^0(X,\oL^n)_D} V^{\FS}_n. $$
So $\bP_{n,D} \big(\B_W^D(\sigma,r)\big)$ is hard to computed. But we will show that the difference of last two quantities is small  under $1/n^2 \log$ for $n$ large enough.

\smallskip

For the well-definess of the conditional probability measure $V^{\FS}_{n,D}$, we need to show that the measure $V^{\FS}_{n}$ on the conditional set is non-zero, i.e.,
$$V^{\FS}_{n} \big(\P H^0(X,\oL^n)_D\big)=\binom {n}{m}^{-1}\int_{\oQ_m}\Ac_m^*(V_n^{\FS})>0,$$
where $\oQ_m$ is the set defined in \eqref{defn-oR-m}. This is just a consequence of Proposition \ref{bound-Sm} for $m\geq m_0$. So for all $n\geq n_0:=m_0+g$, we put
$$K_n:=\int_{\oQ_m}\Ac_m^*(V_n^{\FS}) >0.$$
By Lemma \ref{lem-analytic-0} and the commutative digram \eqref{diagram-ob} again, for any measurable subset $A$ in $\mathcal M(X\setminus D)$, we have
\begin{equation}\label{conditional-formula}
\bP_{n,D}(A)=\int_{A_\P}\dd V_{n,D}^{\FS}  =\int_{\Ac_m^{-1}(A_\P)} {1\over K_n}\Ac_m^*(V_n^{\FS}),
\end{equation}

To prove Theorem \ref{thm-large-devia}, it is enough to show the following two inequalities.

\begin{equation}\label{large-devia-upper}
\limsup_{r\to 0} \limsup_{n\to \infty} {1\over n^2} \log \bP_{n,D} \big(\B_W^D(\sigma,r)\big)\leq      \min  \oI_{\omega,D} - \oI_{\omega,D}(\sigma);
\end{equation}
\begin{equation}\label{large-devia-lower}
\liminf_{r\to 0} \liminf_{n\to \infty} {1\over n^2} \log \bP_{n,D} \big(\B_W^D(\sigma,r)\big)\geq      \min  \oI_{\omega,D} - \oI_{\omega,D}(\sigma).
\end{equation}

\begin{proof}[Proof of  \eqref{large-devia-upper}]
		Let $\mathbf p=p_1+\cdots+p_m\in X^{(m)}\setminus \bH_m$ and  $\mathbf q=\oB_m(\mathbf p)$.
	The set $\Ac_m^{-1}\big(\B_W^D(\sigma,r)\big)$ consists of all the points $\mathbf p\in\oQ_m$ such that $\delta_{\mathbf p+\mathbf q}\in\B_W^D(\sigma,r)$. For every $r>0$. We have $\dist_W(\delta_{\mathbf p},\delta_{\mathbf p+\mathbf q})<r$ for $n$ large enough. In this case, $\dist_W(\delta_{\mathbf p},\sigma)< 2r$.
	Hence by \eqref{conditional-formula} and Proposition \ref{prop-formula}, we have 
	\begin{align*}
	\bP_{n,D} \big(\B_W^D(\sigma,r)\big)&\leq \int_{\mathbf p\in \oQ_m,\delta_{\mathbf p}\in\B_W^D(\sigma,2r) } {1\over K_n} \Ac_m^*(V_{n,D}^{\FS}) \\
	&=\int_{\mathbf p\in \oQ_m,\delta_{\mathbf p}\in\B_W^D(\sigma,2r) } {C_n \over K_n} \exp\Big[m^2\Big(\oE_m(\mathbf p)- {2(m+1)\over m} \oF_m(\mathbf p)  \Big) \Big] \kappa_n.
	\end{align*}
	
	For every $M>0$, we define the truncations
	$$G^M(x,y):=\max \big(G(x,y),-M\big)  \quad \text{and}\quad \oE^M(\mu):=\int_{X \times X}G^M(x,y) \,\dd \mu(x) \dd\mu(y).$$
	Observe that $G^M(x,y)$ is continuous from Lemma \ref{l:Green} and we have
	\begin{align*}
	\oE_m(\mathbf p)={1\over m^2} \sum_{j\neq k} G(p_j,p_k)&\leq {1\over m^2}\sum_{j\neq k} G^M(p_j,p_k)= {1\over m^2}\sum_{j, k} G^M(p_j,p_k) - {1\over m^2}\sum_{j=1}^m G^M(p_j,p_j)\\
	&\leq {1\over m^2}\sum_{j, k} G^M(p_j,p_k)+{M\over m}=\oE^M(\delta_{\mathbf p})+{M\over m}.
	\end{align*}
	Therefore, combining with Lemma \ref{lem-f-m-p}, yields
	$$\oE_m(\mathbf p)-{2(m+1)\over m} \oF_m(\mathbf p)   \leq  \oE^M(\delta_{\mathbf p})+{M\over m}- 2 \max U'_{\delta_{\mathbf p}}+{O (\log m) \over m} \quad\text{as}\quad m\to \infty.$$
	Thus, for every fixed $M$, 	${1\over n^2} \log \bP_{n,D} \big(\B_W^D(\sigma,r)\big)$ is 
	\begin{align}
	&\leq {1\over n^2} \log  \int_{\mathbf p\in \oQ_m,\delta_{\mathbf p}\in\B_W^D(\sigma,2r)}{C_n \over K_n} \exp\Big[m^2 \Big( \oE^M(\delta_{\mathbf p})+{M\over m}- 2 \max U'_{\delta_{\mathbf p}}+{O (\log m) \over m} \Big)\Big] \kappa_n \nonumber\\
	&\leq {1\over n^2}\log{C_n\over K_n} + \sup_{\delta_{\mathbf p}\in\B_W^D(\sigma,2r)}\Big( \oE^M(\delta_{\mathbf p})+{M\over m}- 2 \max U'_{\delta_{\mathbf p}}+{O (\log m) \over m} \Big)+ {1\over n^2}\log \int_{\oQ_m}\kappa_n  \nonumber\\
	&\leq{1\over n^2}\log{C_n\over K_n} + \sup_{\mu\in\B_W^D(\sigma,2r)}\Big( \oE^M(\mu)- 2 \max U'_{\mu}\Big)+{M\over m}+{O (\log m) \over m}+ {1\over n^2}\log \int_{\oQ_m}\kappa_n. \label{pnk-ineq-upper}
	\end{align}

	We need to compute the integral of $\kappa_n$ over $\oQ_m$. Firstly, we notice that 
	\begin{equation}\label{estiamte-exp-kappa}
	-4\sum_{k=1}^m \rho(p_k)-2(m+1)\sum_{j=1}^g \rho(q_j) =O(m)
	\end{equation}
	from the continuity of $\rho$, and
	$$ \prod_{k=1}^m\prod_{j=1}^g \norm{\mathbf 1_{\Oc(p_k)}(q_j)}^2_{\fh_{\omega,p_k}}=O(1)^{2mg}\prod_{k=1}^m\prod_{j=1}^g \dist(p_k,q_j)^2 \leq O(1)^{2mg} \diam(X)^{2mg}.$$	
	Together with Lemma \ref{lem-xi-bound}, yield
	\begin{align*}
	{1\over n^2}\log \int_{\oQ_m}\kappa_n \leq  {1\over n^2}\log \int_{X^{(m)}}\kappa_n  &\leq {O(1)\over n}+ {1\over n^2}\log \int_{X^{(m)}} i\dd z_1\wedge \dd\overline z_1 \wedge\cdots \wedge i\dd z_m\wedge \dd\overline z_m \\
	& = {O(1)\over n}+{1\over n^2}\log \big(O(1)^m\big)=  {O(1)\over n}.
	\end{align*}

	Thus, by taking $n\to \infty$, we deduce from \eqref{pnk-ineq-upper} that for every fixed $r$ and $M$,
	\begin{equation*} 
	\limsup_{n\to \infty} {1\over n^2} \log \bP_{n,D} \big(\B_W^D(\sigma,r)\big)  \leq \limsup_{n\to \infty}{1\over n^2}\log{C_n\over K_n}+   \sup_{\mu\in\B_W^D(\sigma,2r)}\Big( \oE^M(\mu)- 2 \max U'_{\mu}\Big).
	\end{equation*}
	
	From the continuity of $G^M(x,y)$, we see that $\oE^M(\mu)$ is continuous with respect to the weak topology of $\mathcal M (X\setminus D)$. Using the continuity of $\mu\mapsto\max U'_{\mu}$ from Lemma \ref{l:convex}, and taking $r\to \infty$, we conclude that
	$$\limsup_{r\to 0} \limsup_{n\to \infty} {1\over n^2} \log \bP_{n,D} \big(\B_W^D(\sigma,r)\big)  \leq \limsup_{n\to \infty}{1\over n^2}\log{C_n\over K_n}+   \oE^M(\sigma)- 2 \max U'_{\sigma}.$$
	The inequality holds for any $M$. By monotone convergence theorem, 
	$\oE^M(\sigma)\to \int_X  U'_\sigma \,\dd\sigma$ as $M\to\infty$. Using \eqref{second-formula-I}, we receive
	$$\limsup_{r\to 0} \limsup_{n\to \infty} {1\over n^2} \log \bP_{n,D} \big(\B_W^D(\sigma,r)\big)  \leq \limsup_{n\to \infty}{1\over n^2}\log{C_n\over K_n}-\oI_{\omega,D}(\sigma).$$
Thus, \eqref{large-devia-upper}  follows from Lemma  \ref{lem-cn-dn} below.
\end{proof}

\smallskip
 
\begin{proof}[Proof of \eqref{large-devia-lower}]
	According to Proposition \ref{l:reg-I}, we may assume $\sigma$ is smooth.
	For every fixed $r>0$, $\dist_W(\delta_{\mathbf p},\delta_{\mathbf p+\mathbf q})<r/2$ for $n$ large enough. Hence
	by \eqref{conditional-formula}  and Proposition \ref{prop-formula},
	\begin{align*}
	\bP_{n,D} \big(\B_W^D(\sigma,r)\big)
	&\geq \int_{\mathbf p\in \oQ_m,\delta_{\mathbf p}\in\B_W^D(\sigma,r/2) } {C_n \over K_n} \exp\Big[m^2\Big(\oE_m(\mathbf p)- {2(m+1)\over m} \oF_m(\mathbf p)  \Big) \Big] \kappa_n\\
	&\geq \int_{\mathbf p\in \oS_m^\zeta,\delta_{\mathbf p}\in\B_W^D(\sigma,r/2) } {C_n \over K_n} \exp\Big[m^2\Big(\oE_m(\mathbf p)- {2(m+1)\over m} \oF_m(\mathbf p)  \Big) \Big] \kappa_n,
	\end{align*}
	for $S_m^\zeta$ defined in \eqref{defn-oS-m} with $\zeta$ chosen according to Proposition \ref{bound-Sm} depending only on $\sigma$.
	Lemma \ref{lem-f-m-p}   gives that 
	$$ {2(m+1)\over m} \oF_m(\mathbf p)  \leq 2 \max U'_{\delta_{\mathbf p}}+{O (\log m) \over m}    \quad\text{as}\quad m\to \infty.$$
	Moreover, from the continuity of $\mu\mapsto\max  U'_\mu$ in Lemma \ref{l:convex}, we have for $\delta_{\mathbf p}\in \B_W^D (\sigma,r/2)$,  $$|\max U'_{\delta_{\mathbf p}} -\max U'_\mu|=o(1) \quad\text{as} \quad r\to 0.$$
Thus, using  \eqref{second-formula-I}  and Lemma \ref{sigma-delta-potential}, we get that for every $\mathbf p\in \oS_m^\zeta$ and $\delta_{\mathbf p}\in\B_W^D(\sigma,r/2)$,
	$$ \oE_m(\mathbf p)- {2(m+1)\over m} \oF_m(\mathbf p)  \geq -\oI_{\omega,D}(\sigma)-{O (\log m) \over m}-C_{\sigma,\zeta} \sqrt{r/2}-o_r(1). $$
	Therefore, ${1\over n^2} \log \bP_{n,D} \big(\B_W^D(\sigma,r)\big)$ is 
	\begin{align}
	&\geq {1\over n^2} \log \int_{\mathbf p\in \oS_m^\zeta,\delta_{\mathbf p}\in\B_W^D(\sigma,r/2) } {C_n \over K_n} \exp\Big[m^2\Big(-\oI_{\omega,D}(\sigma)-{O (\log m) \over m}-C_{\sigma,\zeta}\sqrt{r/2} -o_r(1)\Big) \Big] \kappa_n \nonumber	\\
	&\geq {1\over n^2}\log{C_n\over K_n}  -\oI_{\omega,D}(\sigma) -{O (\log m) \over m}-C_{\sigma,\zeta}\sqrt{r/2}-o_r(1)+{1\over n^2}\log \int_{\mathbf p\in \oS_m^\zeta,\delta_{\mathbf p}\in\B_W^D(\sigma,r/2) }\kappa_n. \label{pnk-ineq-lower}
	\end{align}
	
 By definition of $\oS_m^\zeta$,  we have for $\mathbf p \in \oS_m^\zeta$,
	$$ \prod_{k=1}^m\prod_{j=1}^g \norm{\mathbf 1_{\Oc(p_k)}(q_j)}^2_{\fh_{\omega,p_k}} =O(1)^{2mg}\prod_{k=1}^m\prod_{j=1}^g  \dist(p_k,q_j)^2  \geq O(1)^{2mg} m^{-2mg}.$$	 
	Hence, together with \eqref{estiamte-exp-kappa}, Lemma \ref{lem-xi-bound} and Proposition \ref{bound-Sm}, the last term of \eqref{pnk-ineq-lower} is 
	\begin{align*}
	&\geq {O(1)\over n}+  {1\over n^2}\log \int_{\mathbf p\in \oS_m^\zeta,\delta_{\mathbf p}\in\B_W^D(\sigma,r/2) }m^{-2mg} \, i\dd z_1\wedge \dd\overline z_1 \wedge\cdots \wedge i\dd z_m\wedge \dd\overline z_m \\
	& \geq {O(1)\over n} -  {O(\log m)\over n}+{1\over n^2}\log (m^{-5m})=-{O(\log n)\over n}.
	\end{align*}

	Thus, by taking $n\to \infty$, we deduce from \eqref{pnk-ineq-lower} that 
	\begin{equation*}\label{pnk-ineq-lower-conclude}
	\liminf_{n\to \infty} {1\over n^2} \log \bP_{n,D} \big(\B_W^D(\sigma,r)\big)  \geq \liminf_{n\to \infty}{1\over n^2}\log{C_n\over K_n}-   \oI_{\omega,D}(\sigma) -C_{\sigma,\zeta} \sqrt{r/2}-o_r(1).
	\end{equation*}
	Then letting $r\to 0$, we receive
	$$\liminf_{r\to 0} \liminf_{n\to \infty} {1\over n^2} \log \bP_{n,D} \big(\B_W^D(\sigma,r)\big)  \geq \liminf_{n\to \infty}{1\over n^2}\log{C_n\over K_n}- \oI_{\omega,D}(\sigma).$$
	Thus, \eqref{large-devia-lower} also follows from Lemma \ref{lem-cn-dn} below.
\end{proof}

In order to complete the proof of Theorem \ref{thm-large-devia}, we need to compute the limit of the term involving $C_n/K_n$.

\begin{lemma}\label{lem-cn-dn}
	We have 
	$$ \lim_{n\to \infty}{1\over n^2}\log{C_n\over K_n}= \min \oI_{\omega,D}.$$
\end{lemma}

\begin{proof}
	Replacing $\B_W^D(\sigma,r)$ by $\mathcal M(X\setminus D)$ in \eqref{pnk-ineq-upper}, we get 
	$${1\over n^2}\log{C_n\over K_n} + \sup_{\mathcal M (X\setminus D)}\Big( \oE^M(\mu)- 2 \max U'_{\mu}\Big)+{M\over m^2}+{O(\log n) \over n} + {1\over n^2}\log \int_{\oQ_m}\kappa_n\geq 0$$
	It follows that
	$$\liminf_{n\to \infty}{1\over n^2}\log{C_n\over K_n} \geq \inf_{\mu\in\mathcal M (X\setminus D) }  \Big( -\oE^M(\mu)+2 \max U'_{\mu}\Big).$$ 
	Taking $M\to \infty$ and monotone convergence theorem give that
	$$ \liminf_{n\to \infty}{1\over n^2}\log{C_n\over K_n} \geq  \min \oI_{\omega,D}.$$
	
	On the other hand, choosing $\sigma_j$ to be a sequence of smooth probability measures in $\cM(X\setminus D)$ converging weakly to the equilibrium measure $\nu$  with $\lim_{j\to \infty} \oI_{\omega, D}(\sigma_j)=\oI_{\omega, D}(\nu)$. Such a sequence exists because of Proposition \ref{l:reg-I}. Then replacing $\sigma$ by $\sigma_j$ in \eqref{pnk-ineq-lower}, yields
	$${1\over n^2}\log{C_n\over K_n}-   \oI_{\omega,D}(\sigma_j)-C_{\sigma_j,\zeta_j} \sqrt{r/2} -o_r(1)-{O(\log n)\over n}\leq {1\over n^2} \log \bP_{n,D} \big(\B_W^D(\sigma_j,r)\big),$$
 where $\zeta_j$ depends on $\sigma_j$.
	Notice that the right hand side is always non-positive. We let $n\to \infty$ first. Then we take $r\to 0$ and $j\to \infty$, getting
	$$ \limsup_{n\to \infty}{1\over n^2}\log{C_n\over K_n} \leq \oI_{\omega,D}(\nu)=\min \oI_{\omega,D}.$$
	The proof of the lemma is finished.
\end{proof}

\begin{remark}\rm
		The quantity $K_n$ above is equal to the hole probability of random holomorphic sections non-vanishing on $D$, times $\binom {n}{m}$. 
Lemma \ref{lem-cn-dn} also gives the estimate of hole probability	in \cite[Theorem 1.4]{shi-zel-zre-ind} for Riemann surfaces case.
\end{remark}

So far, we have completed the proof of Theorem \ref{thm-large-devia} for the case $\deg (\oL)=1$. The higher degree case can be reduced to this case by the next lemma.

\begin{lemma}
	If $\oL'$ is a positive line bundle on $X$ of degree $d>1$. Then there exists a positive line bundle $\oL$ of degree $1$ such that $\oL^d \simeq \oL '$.
\end{lemma}

\begin{proof}
	We consider the line bundle $\oL'^g$. By Jacobi inversion theorem, we can find $p_1,\dots, p_{gd}\in X$ such that $\Oc (x_1+\dots+ x_{gd})\simeq \oL'^g$. Let $[\mathbf z]:=\oA_{gd}(x_1+\cdots + x_{gd})$. Fix a representative $\mathbf z\in \C^g$ of $[\mathbf z]$.  Using Jacobi inversion theorem again, there exist $y_1, \dots, y_g\in X$ such that $\oA_g (y_1+ \cdots+ y_g)=\mathbf z/(gd)$. Abel's theorem implies that $$\Oc (gdy_1 +\cdots +gd y_g)  \otimes \Oc( -g^2 d  p_\star  )\simeq \Oc (x_1+\cdots+x_{gd}) \otimes \Oc (  - gd p_\star   ).$$
	Hence we take $\oL:= \Oc (y_1 +\cdots + y_g-gp_\star +p_\star)$ and finish the proof.
\end{proof}

\begin{remark}\rm
We also need to give the remark for the case $g=0$, equivalently, $X=\P^1$. In this case, $m$ is equal to $n$. The formula in Proposition \ref{prop-formula} remains valid  after we setting $g=0$ (cf.\ \cite[Proposition 3]{zei-zel-imrn}). Moreover, all the objects related to $\mathbf q$, such as $\oB_m,\Ac_m$, are not meaningful any more. The proof will be much simpler because the set of holomorphic sections with zeros only in $X\setminus D$ can be exactly parameterized by  the set $(X\setminus D)^{(n)}$ in $X^{(n)}$. In particular, we do not need the technical Proposition \ref{bound-Sm} any more. 
\end{remark}

\medskip

In the remaining part of this section, we will give the proof of Theorem \ref{t:main-2}, using the large deviation principle we derived a while before. 

\smallskip

Theorem \ref{thm-large-devia} gives the large deviation principle of $\bP_{n,D}$ for every small open ball in $\mathcal M(X\setminus D)$. In  view of Lemma \ref{l:convex}, $\oI_{\omega,D}-\min \oI_{\omega,D}$ is a \textit{good rate function}, i.e., it is lower semicontinuous and non-negative.
Together with \cite[Theorem 4.1.11]{dem-zei-book}, we can deduce the weak large deviation principle for compact sets. More precisely,  for every compact subset $A$ in $\mathcal M(X\setminus D)$, we have
\begin{equation}\label{compact-LDP}
\limsup_{n\to \infty}{1\over n^2} \log \bP_{n,D} (A) \leq    \min  \oI_{\omega,D}- \min_{\mu\in A}\oI_{\omega,D}(\mu).
\end{equation}

Fix a smooth test function $\phi$ on $X$.
For every $t\geq 0$, define the auxiliary set
$$\oT_t:=\Big\{\mu\in \mathcal M(X\setminus D):\, \Big| \int_{X}\phi \,\dd\mu-  \int_{X}\phi \,\dd\nu  \Big|\geq t      \Big\}.$$
Observe that $\oT_t$ is empty for $t>2\norm{\phi}_\infty$.

\begin{lemma}\label{lema-t-I-K}
	There exists a constant $c_\phi>0$ independent of $t$ such that for all $\mu\in \oT_t$,
	$$\oI_{\omega,D}(\mu)-\oI_{\omega,D}(\nu)\geq c_\phi t^2.$$
\end{lemma}

\begin{proof}
	 Notice that $\dd \phi \wedge \dc \phi$ is a positive measure on $X$ with mass $C_\phi$. By Stoke's formula and Cauchy-Schwarz inequality,
	\begin{align*}
	&\Big|\int_{X}\phi \,\dd\mu-  \int_{X}\phi \,\dd\nu\Big|=\Big|\int_{X}\phi \, \ddc(U'_\mu-U'_\nu )\Big|=
	\Big|\int_X \dd \phi \wedge \dc(U'_\mu-U'_\nu )  \Big|\\
	& \leq \Big(\int_X \dd\phi\wedge \dc\phi \Big)^{1/2} \Big(\int_X \dd(U'_\mu-U'_\nu )\wedge \dc (U'_\mu-U'_\nu )\Big)^{1/2}  \\
	&=C^{1/2}_\phi \Big(-\int_X  (U'_\mu-U'_\nu )\, \ddc (U'_\mu-U'_\nu )\Big)^{1/2}=C_\phi^{1/2} \Big(-\int_X  (U'_\mu-U'_\nu )\,\dd(\mu-\nu)\Big)^{1/2}. 
	\end{align*}
	
	On the other hand, by \eqref{Unu-Umu-max},
	$$-\int_X U'_\mu\,\dd\mu +2\int_X U'_\mu\,\dd \nu -\int_X U'_\nu\,\dd\nu \leq  -\int_X U'_\mu\,\dd\mu +2\max U'_\mu+\int_X U'_\nu\,\dd \nu -2\max U'_\nu.$$
	From Lemma \ref{commut-potential} and \eqref{second-formula-I}, we see that last inequality is equivalent to 
	$$- \int_X  (U'_\mu-U'_\nu )\,\dd(\mu-\nu)\leq \oI_{\omega,D}(\mu)-\oI_{\omega,D}(\nu). $$
	Therefore, we conclude that for $\mu \in \oT_t$,
	$$ t\leq \Big|\int_{X}\phi \,\dd\mu-  \int_{X}\phi \,\dd\nu\Big| \leq C^{1/2}_\phi \big( \oI_{\omega,D}(\mu)-\oI_{\omega,D}(\nu) \big)^{1/2}.$$
	We finish the proof by taking $c_\phi := C_\phi^{-2}$.
\end{proof}

Let $\E_{n,D}$ denote the expectation of $\bP_{n,D}$. 

\begin{proof}[Proof of Theorem \ref{t:main-2}]
		It is not hard to see that  $\oT_t$ is closed, and hence compact in $\mathcal M(X\setminus D)$ for every $t\geq 0$.
	So \eqref{compact-LDP} can be applied to $\oT_t$, yielding
	$$\limsup_{n\to \infty}{1\over n^2} \log \bP_{n,D} (\oT_t) \leq   \oI_{\omega,D}(\nu) - \min_{\mu\in \oT_t}\oI_{\omega,D}(\mu) .  $$
	Lemma \ref{lema-t-I-K} gives that
	$$\limsup_{n\to \infty}{1\over n^2} \log \bP_{n,D} (\oT_t) \leq -c_\phi t^2.  $$
	So for every $t\geq 0$, there exists a constant $N_t\in \N$, such that for all $n\geq N_t$, 
	\begin{equation}\label{P-n-K-leq-exp}
	\bP_{n,D}(\oT_t)\leq e^{-c_\phi t^2n^2/2}.
	\end{equation}

   For the first statement of Theorem \ref{t:main-2},	we need to show
	\begin{equation*}
	\lim_{n\to \infty}\int_{\P H^0(X,\oL^n)} \int_X \phi \,\dd (\llbracket Z_s \rrbracket -\nu )\,  \dd V_{n,D}^{\FS}(s) =0.
	\end{equation*}
	By Fubini's theorem,
	\begin{align*}
	\int_{\P H^0(X,\oL^n)} \Big|\int_X \phi \,\dd (\llbracket Z_s \rrbracket -\nu )\Big|\,  \dd V_{n,D}^{\FS}(s) =\E_{n,D}\Big(    \Big|\int_X \phi \,\dd (\mu -\nu )\Big|   \Big) \\
	=\int_0^\infty  \bP_{n,D}\Big(    \Big|\int_X \phi \,\dd (\mu -\nu )\Big|\geq t   \Big) \, \dd t      =\int_0^{2\norm{\phi}_\infty} \bP_{n,D}(  \oT_t  )\, \dd t. 
	\end{align*}
	Then \eqref{P-n-K-leq-exp} and Lebesgue dominated convergence theorem give the result.

\smallskip

For the second statement, 	we define the sequence of random variables $Y_n$ on $\prod_{n\geq n_0} \P H^0(X,\oL^n)$ as follows:
for every sequence $\mathbf s=\{s_n\}_{n\geq n_0}$ in  $\prod_{n\geq n_0} \P H^0(X,\oL^n)$,  let 
$$Y_n (\mathbf s):=  \Big|\int_X \phi \,\dd(\llbracket Z_s \rrbracket -\nu )\Big|,   $$
where $\phi$ is the test function we fixed above.
Our goal is showing that $Y_n$ converges to $0$ almost surely.
According to Borel–Cantelli lemma,
it is enough to prove that for every $t>0$, 
$$ \sum_{n\geq n_0} \bP(Y_n\geq t)= \sum_{n\geq n_0} \bP_{n,D}\Big(    \Big|\int_X \phi \,\dd (\llbracket Z_s \rrbracket -\nu )\Big|\geq t   \Big)    <\infty. $$
Equivalently, 
$ \sum_{n\geq n_0} \bP_{n,D}(\oT_t)   <\infty$.
This follows easily from  \eqref{P-n-K-leq-exp}.
\end{proof}

%%%%%%%%%%%%%%%%%%%%%%%

\medskip

\section{Examples on the Riemann sphere and on a torus} \label{s:example}

In this section, we will give some examples on the Riemann sphere $\P^1$ and the complex torus $\C/(\Z+\Z i)$.

\subsection{Riemann sphere}

Let $X=\P^1$ and $(\oL,\fh)=(\Oc(1),\omega_{\FS})$. We will consider $D$ a connected circular domain.

\smallskip

Since $\P=\C\cup \{\infty\}$ and clearly,  the equilibrium measure $\nu$ of $D$ has no mass at $\infty$. So we can just work on $\C$ and assume that $D$ is a connected circular domain centered at $0$. On the coordinate
$\C$, the Fubini-study form 
$\omega_\FS(z)=(2\pi)^{-1} (1+|z|^2)^{-2} \,i \dd z\wedge \dd \overline z$.
Hence by definition, the quasi-potential of type M 
$$U_\nu(z)=\int_{\C} \log|z-w|\, \dd\nu(w) -\log (1+|z|^2)^{1/2}+ C_\nu   \quad\text{for}\quad z\in\C$$
with some normalized constant $C_\nu$. We have  $ U_\nu\leq 0$, $\max  U_\nu =0$ and $\int_{\C} \log|z-w|\, \dd\nu(w)$ is a subharmonic function on $\C$. 

From Theorem \ref{t:main-1}, we know that the equilibrium measure $\nu$ is the sum of $\nu_{b D}$ and $\nu_S=\omega_\FS |_S$. Moreover,  $bD \cap \supp(\nu_S)=\varnothing$ since $D$ has smooth boundary.  Now we analyze the behavior of $U_\nu$ outside $b D \cup \supp(\nu_S)$.

We first notice that $\nu$ is also circular because $D$ is invariant under rotations with center $0$ and the  equilibrium measure is unique. Fix a point $x\notin bD \cup \supp(\nu_S)$. 
We want to study the term $\int_{\C} \log|z-w|\, \dd\nu(w)$ near $x$. Since $\nu$ is circular, after changing to polar coordinates and using the fact that
\begin{equation}\label{integral-circle}
{1\over 2\pi}\int_{0}^{2\pi} \log|z-re^{i\theta}|\,\dd\theta =\max(\log|z|,\log r ),   
\end{equation}
we deduce that for $z$ in a small neighborhood of $x$ who does not intersect $b D \cup \supp(\nu_S)$, 
\begin{equation}\label{linear-potential-nu}
\int_{\C} \log|z-w|\, \dd\nu(w) = \alpha \log |z| +\beta,  
\end{equation}
where  $0\leq \alpha\leq 1$ and $\beta\in \R$ are constants. In other words, locally, $\int_{\C}\log|z-w|\, \dd\nu(w)$ is an increasing linear function of $\log|z|$ with slope bounded by $1$.
Here, $\alpha,\beta$ depend on the connected component of $\C\setminus(bD \cup \supp(\nu_S))$ which $x$ belongs to.  

\smallskip

\noindent \textbf{Case 1:} $D=\D$. 
Since $\nu$ has no mass on $\D$, applying \eqref{integral-circle} again, yields that $\int_{\C} \log|z-w|\, \dd\nu(w)+ C_\nu$ is constant for $|z|\leq 1$.   This constant should be $0$ because $U_\nu\leq 0$ and $\max _{\overline \D} U_\nu =0$ by Theorem \ref{t:main-1}.

On the other hand, the function $\int_{\C} \log|z-w|\, \dd\nu(w)+ C_\nu$ can only  be non-differentiable on $bD\cup b S$, as a function of $\log|z|$. This is due to the facts that $\nu$ has no mass outside $bD \cup S$ and $\nu$ is smooth on $S^o$.  But we observe that $b S$ is the union of circles since $\nu$ is circular. Hence $\nu$ does not have mass on $b S$. This means that under the variable $\log|z|$, the graph of $\int_{\C} \log|z-w|\, \dd\nu(w)+ C_\nu$ is tangent to the graph of $\log (1+|z|^2)^{1/2}$ at the points in $b S$. In particular, the function $\int_{\C} \log|z-w|\, \dd\nu(w)+ C_\nu$ can only  be non-differentiable on $bD$, as a function of $\log|z|$.

Therefore, according to \eqref{linear-potential-nu}, we can draw the graphs of $\log (1+|z|^2)^{1/2}$ and $\int_{\C} \log|z-w|\, \dd\nu(w)+ C_\nu$, see the figure below.  The blue curve represents $\log (1+|z|^2)^{1/2}$ under the variable $\log|z|$. The red curve represents the non-zero part of $\int_{\C} \log|z-w|\, \dd\nu(w)+ C_\nu$. The implies that $S=\{0,\infty\}$, which gives that $\nu_S=0$. Hence we conclude that $\nu$ is the Lebesgue probability measure on the unit circle.

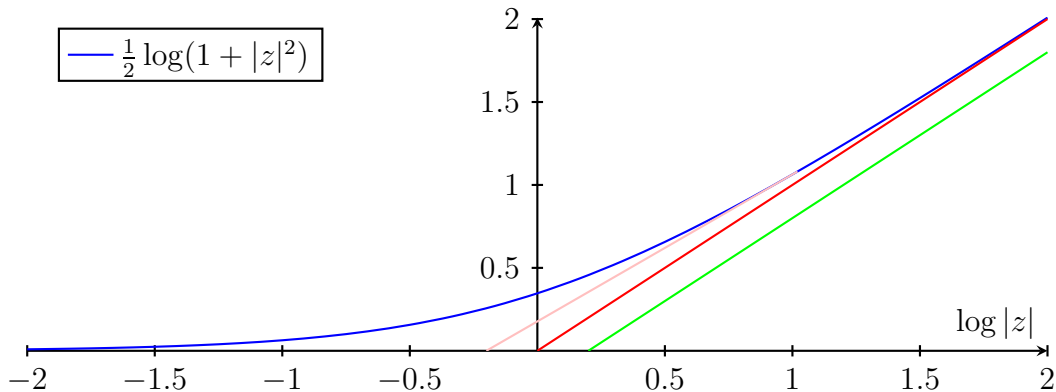
\begin{tikzpicture}
\begin{axis}[xmin=-2,xmax=2,
width=15cm, 
height=6cm,
domain=-2:2,
xlabel=$\log|z|$,
smooth,thick,
axis lines=center,
every tick/.style={thick},
legend style={cells={anchor=west}},
legend pos=north west]

\addplot[color=blue,domain=-2:2]{1/2 * ln(1+e^x* e^x)};

\addplot[color=red,domain=0:2]{x};
\addplot[color=green,domain=0.2:2]{x-0.2};
\draw [color=pink] (-0.2 ,0 ) -- (1.02, 1.08112);
\legend{${1\over 2}\log(1+|z|^2)$}

\end{axis}
\end{tikzpicture}

\smallskip

\noindent \textbf{Case 2:} $D=\D(0,r)$ for $r>1$. Repeating the same argument as  \textbf{Case 1}, we can draw the green curve in the figure above  representing $\int_{\C} \log|z-w|\, \dd\nu(w)+ C_\nu$ in this case.  This gives that $S=\{0\}$. So we conclude that $\nu$ is the Lebesgue probability measure on the circle $b \D(0,r)$.

\smallskip

\noindent \textbf{Case 3:} $D=\D(0,r)$ for $r<1$. Through the point $\log|z|=\log r$, we can find a unique pink line tangent to the blue curve at some point $\big(\log t, {1\over 2 } \log (1+t^2) \big) $ with $t>0$, see the figure above.
The pink curve represents $\int_{\C} \log|z-w|\, \dd\nu(w)+ C_\nu$ in this case. We see that $S$ is the union of $\{0\}$ and an open ball $B$ centered at $\infty$. So we conclude that $\nu_S=\omega_\FS |_B$ and $\nu_{b D}$ is a Lebesgue measure on the circle $b \D(0,r)$ with mass  $1-\omega_\FS(B)$.

%{\color{red}
%	The solution of $t$ is given by
%	$${\log(1+e^{2t})   \over \log t -\log r}={2 e^{2t}\over 1+e^{2t}}.  $$
%}

\smallskip

\noindent \textbf{Case 4:} $D$ is a symmetric annulus on $\P^1$. In this case, $D=\{r <|z|<1/r\}$ on $\C$ for some $0<r <1$.  By the same reason,  $\int_{\C} \log|z-w|\, \dd\nu(w)$ is linear on $\log |z|$ for $z\in \C \setminus bD$, and it can only be non-differentiable on $bD$. Moreover, $U_\nu (z) =0$ at $|z|=1$ since $D$ is symmetric. So we obtain the graphs below for $r<1/2$ (pink), $r=1/2$ (red),  $r>1/2$ (green).

%{\color{red}
%	 The left corner point is $\big(\log r, 1/2 \log (2r)\big)$. Hence the left tangent point $|z|=\log t$ is given by the equation
%	$${\log (2r)-\log (1+e^{2t})\over \log r -t}={2e^{2t}\over 1+e^{2t}}.$$	
%}

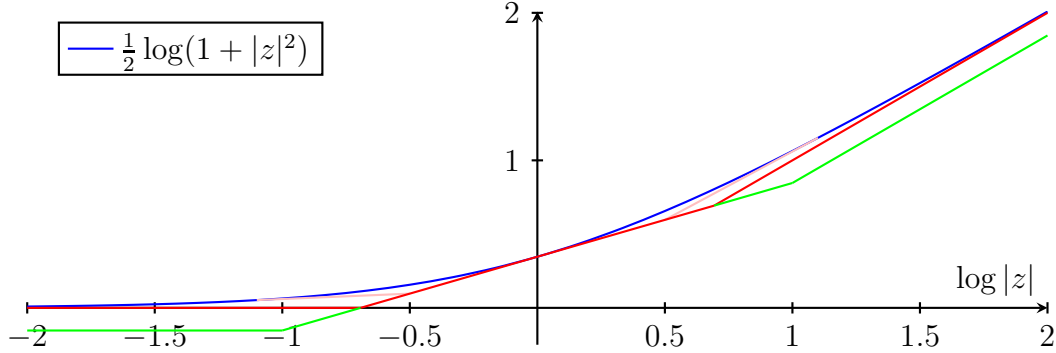
\begin{tikzpicture}
\begin{axis}[xmin=-2,xmax=2,ymin=-0.25,
width=15cm, 
height=6cm,
domain=-2:2,
xlabel=$\log|z|$,
smooth,thick,
axis lines=center,
every tick/.style={thick},
legend style={cells={anchor=west}},
legend pos=north west]

\addplot[color=blue,domain=-2:2]{1/2 * ln(1+e^x* e^x)};

\draw [color=pink](-1.1, 1/2 * ln 1.1108 ) -- (-0.5 , -1/4 +1/2 * ln 2 );
\draw [color=pink] (0.5 , 1/4 +1/2 * ln 2 ) -- (1.1, 1/2 * ln 10.025 );

\addplot[color=red,domain=ln 2  :  2]{x};
\addplot[color=red,domain=-ln 2  :  ln 2]{1/2 * x +1/2 * ln(2)};
\draw [color=red] (-2 , 0 ) -- (-ln 2, 0 );

\draw [color=green] (-2 , 1/2 * ln 2 - 1/2 ) -- (-1, 1/2 * ln 2- 1/2 );
\draw [color=green] (-1, 1/2 * ln 2- 1/2 ) -- (-ln 2 , 0);
\draw [color=green] (ln 2 ,  ln 2 ) -- ( 1 , 1/2 * ln 2 + 1/2  );
\draw [color=green] ( 1 , 1/2 * ln 2 + 1/2  )--( 2, 3/2+ 1/2 * ln 2   );

\legend{${1\over 2}\log(1+|z|^2)$}

\end{axis}
\end{tikzpicture}

\smallskip

\noindent \textbf{Case 5:}  $D=\{r_1\leq |z|\leq r_2\}$ with $0<r_1<r_2<\infty$, is a non-symmetric annulus  on $\P^1$. In this case, the computation is quite long, because the set $D \cap S$ is not obvious. We just make the conclusion here. For the smooth part $\nu_S$, $\supp(\nu_S)$ may vanish, or $\supp(\nu_S)=\overline \D(0,r_3)$, or  $\supp(\nu_S)=\overline \D(\infty,r_4)$, or $\supp(\nu_S)=\overline \D(0,r_5) \cup \overline  \D(\infty,r_6)$.

\smallskip

Summing up, when $D$ is an open ball larger than the unit disc, the equilibrium measure $\nu$  only has the singular part $\nu_{bD}$. When $D$ is an annulus large enough, $\nu_S$ vanishes also.

\subsection{Elliptic curve}
Let $X=\C/(\Z+\Z i)$  and $(\oL,\fh)=(\Oc(z_0),\omega)$, where $z_0=0$ and $\omega$ is induced by the standard K\"ahler form $1/2 \, i \dd z \wedge \dd \overline z$ from $\C$. We use $x,y$ to denote the real coordinates and  we will consider $D$ a strip with hight $r<1$ and width $1$ on some fundamental domain.   After a translation, we can assume $D=\{0\leq x\leq r/2\}\cup \{1-r/2\leq x\leq 1\}$. 

\smallskip

First we notice that $U_\nu$ is a function independent of $y$ and $2\pi x^2$ is a local potential of $\omega$. Hence locally, we can write $U_\nu = V_\nu -2\pi x^2$, where $V_\nu$ is a local potential of $\nu$.  For $z$ outside $bD\cup S$, $V_\nu$ is harmonic. Thus, $V_\nu$ is piecewise linear with respect to the variable $x$ on $X\setminus (bD\cup S)$ because  $V_\nu$ is a function with only one variable $x$. 

On the other hand, by Theorem \ref{t:main-1} and using that the non-differentiable points of $U_\nu$ can only appear on $bD \cup bS$, we get that $D\cap S=\{x=0\}$. 
Clearly, $\nu$ has no mass on $bS$. By the same reason as in the $\P^1$ case,  we see that as a function of $x$, the graph of $V_\nu$ should be tangent to the graph of $2\pi x^2$ at every $x'$, which is the $x$-coordinate of $S$.

Then we draw the figures below. The advantage of choosing $D=\{0\leq x\leq r/2\}\cup \{1-r/2\leq x\leq 1\}$ is that we do not need to care about lattice points $\{x=0\}$ and $\{x=1\}$ to make sure that we actually define a function on the torus. The blue curve represents the function $2\pi x^2$.

\medskip

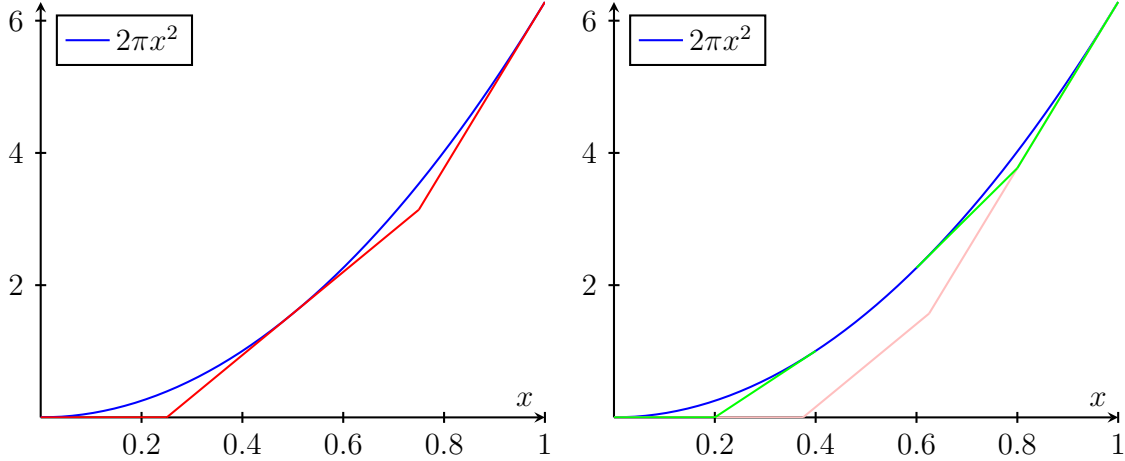
\begin{minipage}[b]{0.45\textwidth}  
	\begin{adjustbox}{width=\linewidth}
		\begin{tikzpicture}
		\begin{axis}[
		domain=0:1,
		xlabel=$x$,
		smooth,thick,
		axis lines=center,
		every tick/.style={thick},
		legend style={cells={anchor=west}},
		legend pos=north west]

		\addplot[color=blue]{2*pi*x^2};
		\addplot[color=red,domain=0:1/4]{0};
		\addplot[color=red,domain=1/4:3/4]{2*pi*(x-1/4)};
		\addplot[color=red,domain=3/4:1]{4*pi*(x-1/2)};

		\legend{$2\pi x^2$}
		\end{axis}
		\end{tikzpicture}
	\end{adjustbox}
\end{minipage}
\begin{minipage}[b]{0.45\textwidth}  
	\begin{adjustbox}{width=\linewidth}
		\begin{tikzpicture}
		\begin{axis}[
		domain=0:1,
		xlabel=$x$,
		smooth,thick,
		axis lines=center,
		every tick/.style={thick},
		legend style={cells={anchor=west}},
		legend pos=north west]

		\addplot[color=blue]{2*pi*x^2};
		\addplot[color=pink,domain=0:3/8]{0};
		\addplot[color=pink,domain=3/8:5/8]{2*pi*(x-3/8)};
		\addplot[color=pink,domain=5/8:1]{4*pi*(x-1/2)};
		\addplot[color=green,domain=0:0.2]{0};
		\addplot[color=green,domain=0.2:0.4]{1.6 * pi * x -1.6 * pi *0.2};
		\addplot[color=green,domain=0.6:0.8]{2.4 * pi * x -2.4 * pi *0.6+ 2 * pi * 0.36};
		\addplot[color=green,domain=0.8:1]{4*pi*(x-1/2)};
		\legend{$2\pi x^2$}
		\end{axis}
		\end{tikzpicture}
	\end{adjustbox} 
\end{minipage}

When $r=1/2$, $S=\{x=0\}\cup \{x=1/2\}$, see the left figure above. Hence in this case, $\nu_S=0$ and $\nu$ is the Lebesgue probability measure on $\{x=1/4\}\cup \{x=3/4\}$.

When $r<1/2$, $S$ is the union of $\{x=0\}$ and the strip $\{r<x<1-r\}$, see the green curve in the right figure above. So in this case, $\nu_S=\mathbf 1_{\{r<x<1-r\}} \omega$ and $\nu_{bD}$ is the Lebesgue probability measure on $\{x=r/2\}\cup \{x=1-r/2\}$ with mass $2r$.

When $r>1/2$, $S=\{x=0\}$, see the pink curve in the right figure above. So in this case, $\nu_S=0$ and $\nu$ is the Lebesgue probability measure on $\{x=r/2\}\cup \{x=1-r/2\}$.

Summing up, when $r\geq 1/2$, the  equilibrium measure $\nu$   is the Lebesgue probability measure on the boundary of $D$. When $r <1/2$, $\nu$ is the sum of a Lebesgue measure on $bD$ and $\omega |_B$ for some strip $B$.

\medskip

As shown in the above examples, the measure $\nu_S$, which is the smooth part of the  equilibrium measure, vanishes for some special choices of  $D$. In fact, such a $D$ always exists for any compact Riemann surface. So we end our paper with the following proposition.

\begin{proposition}\label{prop-nu-S-vanish}
	For any compact Riemann surface $X$ and any positive line bundle $(\oL,\fh)$ on $X$ associated to a K\"ahler metric $\omega$, there exists an open subset $D$ with smooth boundary such that the  equilibrium measure  of $D$ is supported on the boundary of $D$. 
\end{proposition}

\begin{proof}
	Recall the decomposition $\nu=\nu_{bD}+\nu_S$ from Theorem \ref{t:main-1} and $S=\big\{ U'_\nu=\max U'_\nu \big\}$.  It is enough to find a $D$ such that $S$ is contained in $\overline D$, i.e.,  $U'_\nu$ does not attend its maximum on $X\setminus \overline D$. This clearly implies that $\nu_S$ vanishes. We fix a point $x\in X$ and consider $D=X\setminus \overline{ \B(x,r)}$. We will show that for any probability measure $\mu$ supported on $\overline{\B(x,r)}$, $U'_\mu$ does not attend its maximum on   $\overline{\B(x,r)}$ when $r$ is small enough.
	
	Fix another point $y$ away from $x$. It suffices to prove that $U'_\mu(z)<U'_\mu(y)$ for any $z\in \overline{\B(x,r)}$. By Stoke's formula, we have
	\begin{align*}
	U'_\mu(z)-U'_\mu(y)&=\int_X U'_\mu \,\dd (\delta_z-\delta_y)=\int_X (U'_{\delta_z}-U'_{\delta_y} )\,\ddc U_\mu \\
	&=\int_X (U'_{\delta_z}-U'_{\delta_y} )\, \dd(\mu-\omega)=\int_X (U'_{\delta_z}-U'_{\delta_y} )\, \dd\mu.
	\end{align*}
	Observe that as $r\to 0$, $U'_{\delta_z}$ is very negative on $\overline{\B(x,r)}$ since $z\in \overline{\B(x,r)}$. While $U'_{\delta_y}$ is bounded on  $\overline{\B(x,r)}$. Therefore, using that $\supp(\mu)\subset \overline{\B(x,r)}$,  we conclude that $\int_X (U'_{\delta_z}-U'_{\delta_y} )\, \dd\mu<0$. This ends the proof of the proposition. 
\end{proof}


\medskip

\begin{thebibliography}{10}
	
	\bibitem{bay-ind}
	Turgay Bayraktar.
	\newblock Equidistribution of zeros of random holomorphic sections.
	\newblock {\em Indiana Univ. Math. J.}, 65(5):1759--1793, 2016.
	
	\bibitem{survey-random-hol}
	Turgay Bayraktar, Dan Coman, Hendrik Herrmann, and George Marinescu.
	\newblock A survey on zeros of random holomorphic sections.
	\newblock {\em Dolomites Res. Notes Approx.}, 11:1--19, 2018.
	
	\bibitem{bay-com-mar-TAMS}
	Turgay Bayraktar, Dan Coman, and George Marinescu.
	\newblock Universality results for zeros of random holomorphic sections.
	\newblock {\em Trans. Amer. Math. Soc.}, 373(6):3765--3791, 2020.
	
	\bibitem{bedford-1982}
	Eric Bedford and B.~A. Taylor.
	\newblock A new capacity for plurisubharmonic functions.
	\newblock {\em Acta Math.}, 149(1-2):1--40, 1982.
	
	\bibitem{Zei-esaim}
	G\'{e}rard Ben~Arous and Ofer Zeitouni.
	\newblock Large deviations from the circular law.
	\newblock {\em ESAIM Probab. Statist.}, 2:123--134, 1998.
	
	\bibitem{semianalytic}
	Edward Bierstone and Pierre~D. Milman.
	\newblock Semianalytic and subanalytic sets.
	\newblock {\em Inst. Hautes \'{E}tudes Sci. Publ. Math.}, (67):5--42, 1988.
	
	\bibitem{ble-di-jsp}
	Pavel Bleher and Xiaojun Di.
	\newblock Correlations between zeros of a random polynomial.
	\newblock {\em J. Statist. Phys.}, 88(1-2):269--305, 1997.
	
	\bibitem{ble-shi-zel-invent}
	Pavel Bleher, Bernard Shiffman, and Steve Zelditch.
	\newblock Universality and scaling of correlations between zeros on complex
	manifolds.
	\newblock {\em Invent. Math.}, 142(2):351--395, 2000.
	
	\bibitem{blo-plya-plms}
	A.~Bloch and G.~P\'{o}lya.
	\newblock On the {R}oots of {C}ertain {A}lgebraic {E}quations.
	\newblock {\em Proc. London Math. Soc. (2)}, 33(2):102--114, 1931.
	
	\bibitem{nlo-lev-poten}
	T.~Bloom and N.~Levenberg.
	\newblock Random polynomials and pluripotential-theoretic extremal functions.
	\newblock {\em Potential Anal.}, 42(2):311--334, 2015.
	
	\bibitem{blo-tho}
	Thomas Bloom.
	\newblock Random polynomials and {G}reen functions.
	\newblock {\em Int. Math. Res. Not.}, (28):1689--1708, 2005.
	
	\bibitem{blo-shi-mrl}
	Thomas Bloom and Bernard Shiffman.
	\newblock Zeros of random polynomials on {$\mathbb C^m$}.
	\newblock {\em Math. Res. Lett.}, 14(3):469--479, 2007.
	
	\bibitem{bog-boh-leb-prl}
	E.~Bogomolny, O.~Bohigas, and P.~Leb\oe~uf.
	\newblock Distribution of roots of random polynomials.
	\newblock {\em Phys. Rev. Lett.}, 68(18):2726--2729, 1992.
	
	\bibitem{bog-boh-leb-jsp}
	E.~Bogomolny, O.~Bohigas, and P.~Leboeuf.
	\newblock Quantum chaotic dynamics and random polynomials.
	\newblock {\em J. Statist. Phys.}, 85(5-6):639--679, 1996.
	
	\bibitem{buc-nis-ron-sod-ptrf}
	Jeremiah Buckley, Alon Nishry, Ron Peled, and Mikhail Sodin.
	\newblock Hole probability for zeroes of {G}aussian {T}aylor series with finite
	radii of convergence.
	\newblock {\em Probab. Theory Related Fields}, 171(1-2):377--430, 2018.
	
	\bibitem{coman-marin-Nguyen}
	Dan Coman, George Marinescu, and Viet-Anh Nguyen.
	\newblock {H}olomorphic sections of line bundles vanishing along subvarieties.
	\newblock {\em {\tt arXiv:1909.00328}}, 2019.
	
	\bibitem{demailly:agbook}
	Jean-Pierre Demailly.
	\newblock {\em Complex Analytic and Differential Geometry}.
	\newblock
	\url{http://www-fourier.ujf-grenoble.fr/~demailly/manuscripts/agbook.pdf}.
	
	\bibitem{demailly:ptbook}
	Jean-Pierre Demailly.
	\newblock {\em Potential theory in several complex variables}.
	\newblock
	\url{http://www-fourier.ujf-grenoble.fr/~demailly/manuscripts/nice_cimpa.pdf}.
	
	\bibitem{dem-zei-book}
	Amir Dembo and Ofer Zeitouni.
	\newblock {\em Large deviations techniques and applications}, volume~38 of {\em
		Applications of Mathematics (New York)}.
	\newblock Springer-Verlag, New York, second edition, 1998.
	
	\bibitem{din-ma-mar-jfa}
	Tien-Cuong Dinh, Xiaonan Ma, and George Marinescu.
	\newblock Equidistribution and convergence speed for zeros of holomorphic
	sections of singular {H}ermitian line bundles.
	\newblock {\em J. Funct. Anal.}, 271(11):3082--3110, 2016.
	
	\bibitem{din-ma-ngu-ens}
	Tien-Cuong Dinh, Xiaonan Ma, and Vi\^{e}t-Anh Nguy\^{e}n.
	\newblock Equidistribution speed for {F}ekete points associated with an ample
	line bundle.
	\newblock {\em Ann. Sci. \'{E}c. Norm. Sup\'{e}r. (4)}, 50(3):545--578, 2017.
	
	\bibitem{din-mar-sch-jsp}
	Tien-Cuong Dinh, George Marinescu, and Viktoria Schmidt.
	\newblock Equidistribution of zeros of holomorphic sections in the non-compact
	setting.
	\newblock {\em J. Stat. Phys.}, 148(1):113--136, 2012.
	
	\bibitem{din-sib-cmh}
	Tien-Cuong Dinh and Nessim Sibony.
	\newblock Distribution des valeurs de transformations m\'{e}romorphes et
	applications.
	\newblock {\em Comment. Math. Helv.}, 81(1):221--258, 2006.
	
	\bibitem{dinh-sibony:cime}
	Tien-Cuong Dinh and Nessim Sibony.
	\newblock Dynamics in several complex variables: endomorphisms of projective
	spaces and polynomial-like mappings.
	\newblock In {\em Holomorphic dynamical systems}, volume 1998 of {\em Lecture
		Notes in Math.}, pages 165--294. Springer, Berlin, 2010.
	
	\bibitem{dre-liu-mar}
	Alexander Drewitz, Bingxiao Liu, and George Marinescu.
	\newblock {L}arge deviations for zeros of holomorphic sections on punctured
	riemann surfaces.
	\newblock {\em {\tt arXiv:2109.09156}, to appear in Michigan Math. J.}, 2021.
	
	\bibitem{ede-kos-bams}
	Alan Edelman and Eric Kostlan.
	\newblock How many zeros of a random polynomial are real?
	\newblock {\em Bull. Amer. Math. Soc. (N.S.)}, 32(1):1--37, 1995.
	
	\bibitem{erdos-turn-annals}
	P.~Erd\"{o}s and P.~Tur\'{a}n.
	\newblock On the distribution of roots of polynomials.
	\newblock {\em Ann. of Math. (2)}, 51:105--119, 1950.
	
	\bibitem{erdos-offord-plms}
	Paul Erd\"{o}s and A.~C. Offord.
	\newblock On the number of real roots of a random algebraic equation.
	\newblock {\em Proc. London Math. Soc. (3)}, 6:139--160, 1956.
	
	\bibitem{harmonic-book}
	John~B. Garnett and Donald~E. Marshall.
	\newblock {\em Harmonic measure}, volume~2 of {\em New Mathematical
		Monographs}.
	\newblock Cambridge University Press, Cambridge, 2005.
	
	\bibitem{gho-nis-con}
	Subhroshekhar Ghosh and Alon Nishry.
	\newblock Point processes, hole events, and large deviations: random complex
	zeros and {C}oulomb gases.
	\newblock {\em Constr. Approx.}, 48(1):101--136, 2018.
	
	\bibitem{gho-nis-cpam}
	Subhroshekhar Ghosh and Alon Nishry.
	\newblock Gaussian complex zeros on the hole event: the emergence of a
	forbidden region.
	\newblock {\em Comm. Pure Appl. Math.}, 72(1):3--62, 2019.
	
	\bibitem{cho-zei-imrn}
	Subhroshekhar Ghosh and Ofer Zeitouni.
	\newblock Large deviations for zeros of random polynomials with i.i.d.
	exponential coefficients.
	\newblock {\em Int. Math. Res. Not. IMRN}, (5):1308--1347, 2016.
	
	\bibitem{PAG}
	Phillip Griffiths and Joseph Harris.
	\newblock {\em Principles of algebraic geometry}.
	\newblock Wiley Classics Library. John Wiley \& Sons, Inc., New York, 1994.
	\newblock Reprint of the 1978 original.
	
	\bibitem{gun-book}
	R.~C. Gunning.
	\newblock {\em Lectures on {R}iemann surfaces, {J}acobi varieties}.
	\newblock Mathematical Notes, No. 12. Princeton University Press, Princeton,
	N.J.; University of Tokyo Press, Tokyo, 1972.
	
	\bibitem{ham-pro}
	J.~M. Hammersley.
	\newblock The zeros of a random polynomial.
	\newblock In {\em Proceedings of the {T}hird {B}erkeley {S}ymposium on
		{M}athematical {S}tatistics and {P}robability, 1954--1955, vol. {II}}, pages
	89--111. Univ. California Press, Berkeley-Los Angeles, Calif., 1956.
	
	\bibitem{gaussian-book}
	J.~Ben Hough, Manjunath Krishnapur, Yuval Peres, and B\'{a}lint Vir\'{a}g.
	\newblock {\em Zeros of {G}aussian analytic functions and determinantal point
		processes}, volume~51 of {\em University Lecture Series}.
	\newblock American Mathematical Society, Providence, RI, 2009.
	
	\bibitem{ibr-zei-tams}
	Ildar Ibragimov and Ofer Zeitouni.
	\newblock On roots of random polynomials.
	\newblock {\em Trans. Amer. Math. Soc.}, 349(6):2427--2441, 1997.
	
	\bibitem{kac-bams}
	M.~Kac.
	\newblock On the average number of real roots of a random algebraic equation.
	\newblock {\em Bull. Amer. Math. Soc.}, 49:314--320, 1943.
	
	\bibitem{kac-plms}
	M.~Kac.
	\newblock On the average number of real roots of a random algebraic equation.
	{II}.
	\newblock {\em Proc. London Math. Soc. (2)}, 50:390--408, 1949.
	
	\bibitem{kri-jsp}
	Manjunath Krishnapur.
	\newblock Overcrowding estimates for zeroes of planar and hyperbolic {G}aussian
	analytic functions.
	\newblock {\em J. Stat. Phys.}, 124(6):1399--1423, 2006.
	
	\bibitem{lit-efford-jlms}
	J.~E. Littlewood and A.~C. Offord.
	\newblock On the {N}umber of {R}eal {R}oots of a {R}andom {A}lgebraic
	{E}quation.
	\newblock {\em J. London Math. Soc.}, 13(4):288--295, 1938.
	
	\bibitem{lit-offord-rec}
	J.~E. Littlewood and A.~C. Offord.
	\newblock On the number of real roots of a random algebraic equation. {III}.
	\newblock {\em Rec. Math. [Mat. Sbornik] N.S.}, 12/54:277--286, 1943.
	
	\bibitem{mas-aka-2}
	N.~B. Maslova.
	\newblock The distribution of the number of real roots of random polynomials.
	\newblock {\em Teor. Verojatnost. i Primenen.}, 19:488--500, 1974.
	
	\bibitem{mas-aka}
	N.~B. Maslova.
	\newblock The variance of the number of real roots of random polynomials.
	\newblock {\em Teor. Verojatnost. i Primenen.}, 19:36--51, 1974.
	
	\bibitem{nis-imrn}
	Alon Nishry.
	\newblock Asymptotics of the hole probability for zeros of random entire
	functions.
	\newblock {\em Int. Math. Res. Not. IMRN}, (15):2925--2946, 2010.
	
	\bibitem{nis-jdm}
	Alon Nishry.
	\newblock Hole probability for entire functions represented by {G}aussian
	{T}aylor series.
	\newblock {\em J. Anal. Math.}, 118(2):493--507, 2012.
	
	\bibitem{Nishry-Wennman}
	Alon Nishry and Aron Wennman.
	\newblock {T}he forbidden region for random zeros: appearance of quadrature
	domains.
	\newblock {\em {\tt arXiv:2009.08774}, to appear in Comm. Pure Appl. Math.},
	2020.
	
	\bibitem{per-vir-acta}
	Yuval Peres and B\'{a}lint Vir\'{a}g.
	\newblock Zeros of the i.i.d. {G}aussian power series: a conformally invariant
	determinantal process.
	\newblock {\em Acta Math.}, 194(1):1--35, 2005.
	
	\bibitem{sak-ASNS}
	Makoto Sakai.
	\newblock Regularity of free boundaries in two dimensions.
	\newblock {\em Ann. Scuola Norm. Sup. Pisa Cl. Sci. (4)}, 20(3):323--339, 1993.
	
	\bibitem{Shabat-book}
	B.~V. Shabat.
	\newblock {\em Introduction to complex analysis. {P}art {II}}, volume 110 of
	{\em Translations of Mathematical Monographs}.
	\newblock American Mathematical Society, Providence, RI, 1992.
	\newblock Functions of several variables, Translated from the third (1985)
	Russian edition by J. S. Joel.
	
	\bibitem{she-van-tams}
	Larry~A. Shepp and Robert~J. Vanderbei.
	\newblock The complex zeros of random polynomials.
	\newblock {\em Trans. Amer. Math. Soc.}, 347(11):4365--4384, 1995.
	
	\bibitem{shiiffman-jga}
	Bernard Shiffman.
	\newblock Asymptotic expansion of the variance of random zeros on complex
	manifolds.
	\newblock {\em J. Geom. Anal.}, 31(8):8607--8631, 2021.
	
	\bibitem{shi-zel-cmp}
	Bernard Shiffman and Steve Zelditch.
	\newblock Distribution of zeros of random and quantum chaotic sections of
	positive line bundles.
	\newblock {\em Comm. Math. Phys.}, 200(3):661--683, 1999.
	
	\bibitem{shi-zel-crelle}
	Bernard Shiffman and Steve Zelditch.
	\newblock Asymptotics of almost holomorphic sections of ample line bundles on
	symplectic manifolds.
	\newblock {\em J. Reine Angew. Math.}, 544:181--222, 2002.
	
	\bibitem{shi-zel-imrn}
	Bernard Shiffman and Steve Zelditch.
	\newblock Equilibrium distribution of zeros of random polynomials.
	\newblock {\em Int. Math. Res. Not.}, (1):25--49, 2003.
	
	\bibitem{shi-zel-gafa}
	Bernard Shiffman and Steve Zelditch.
	\newblock Number variance of random zeros on complex manifolds.
	\newblock {\em Geom. Funct. Anal.}, 18(4):1422--1475, 2008.
	
	\bibitem{shi-zel-pamq}
	Bernard Shiffman and Steve Zelditch.
	\newblock Number variance of random zeros on complex manifolds, {II}: smooth
	statistics.
	\newblock {\em Pure Appl. Math. Q.}, 6(4):1145--1167, 2010.
	
	\bibitem{shi-zel-zh-jus}
	Bernard Shiffman, Steve Zelditch, and Qi~Zhong.
	\newblock Random zeros of complex manifolds: conditional expectations.
	\newblock {\em J. Inst. Math. Jussieu}, 10(3):753--783, 2011.
	
	\bibitem{shi-zel-zre-ind}
	Bernard Shiffman, Steve Zelditch, and Scott Zrebiec.
	\newblock Overcrowding and hole probabilities for random zeros on complex
	manifolds.
	\newblock {\em Indiana Univ. Math. J.}, 57(5):1977--1997, 2008.
	
	\bibitem{sod-tsi-ijm-1}
	Mikhail Sodin and Boris Tsirelson.
	\newblock Random complex zeroes. {I}. {A}symptotic normality.
	\newblock {\em Israel J. Math.}, 144:125--149, 2004.
	
	\bibitem{sod-tsi-isrj}
	Mikhail Sodin and Boris Tsirelson.
	\newblock Random complex zeroes. {III}. {D}ecay of the hole probability.
	\newblock {\em Israel J. Math.}, 147:371--379, 2005.
	
	\bibitem{sod-tsi-ijm-2}
	Mikhail Sodin and Boris Tsirelson.
	\newblock Random complex zeroes. {II}. {P}erturbed lattice.
	\newblock {\em Israel J. Math.}, 152:105--124, 2006.
	
	\bibitem{tsuji-book}
	M.~Tsuji.
	\newblock {\em Potential theory in modern function theory}.
	\newblock Chelsea Publishing Co., New York, 1975.
	\newblock Reprinting of the 1959 original.
	
	\bibitem{yau-cpam}
	Shing~Tung Yau.
	\newblock On the {R}icci curvature of a compact {K}\"{a}hler manifold and the
	complex {M}onge-{A}mp\`ere equation. {I}.
	\newblock {\em Comm. Pure Appl. Math.}, 31(3):339--411, 1978.
	
	\bibitem{zei-zel-imrn}
	Ofer Zeitouni and Steve Zelditch.
	\newblock Large deviations of empirical measures of zeros of random
	polynomials.
	\newblock {\em Int. Math. Res. Not. IMRN}, (20):3935--3992, 2010.
	
	\bibitem{zel-imrn}
	Steve Zelditch.
	\newblock Large deviations of empirical measures of zeros on {R}iemann
	surfaces.
	\newblock {\em Int. Math. Res. Not. IMRN}, (3):592--664, 2013.
	
	\bibitem{zhu-anapde}
	Junyan Zhu.
	\newblock Hole probabilities of {${\rm SU}(m+1)$} {G}aussian random
	polynomials.
	\newblock {\em Anal. PDE}, 7(8):1923--1968, 2014.
	
	\bibitem{zre-michigan}
	Scott Zrebiec.
	\newblock The zeros of flat {G}aussian random holomorphic functions on
	{$\mathbb C^n$}, and hole probability.
	\newblock {\em Michigan Math. J.}, 55(2):269--284, 2007.
	
\end{thebibliography}
\end{document}